\documentclass [leqno, 10pt]{amsart}

\usepackage{amssymb, amsmath, amsfonts, amsthm, amsbsy, amscd}
\usepackage[bbgreekl]{mathbbol} 
\usepackage[mathscr]{euscript}
\usepackage{esint}

\usepackage{url}
\usepackage[linktocpage]{hyperref}

\newtheorem{theorem}{Theorem}[section]
\newtheorem*{theorem*}{Theorem}{\,}
\newtheorem{corollary}{Corollary}[section]
\newtheorem{lemma}{Lemma}[section]

\newtheorem*{definition*}{Definition}{}

\theoremstyle{definition}

\numberwithin{equation}{section}

\newcommand{\con}{\mathscr{C}}
\newcommand{\bphi}{\bar{\phi}}
\newcommand{\inabla}{\bnabla}
\newcommand{\muf}{\mu}
\newcommand{\amc}{\mathcal{H}}
\newcommand{\ac}{\operatorname{\eta}}
\newcommand{\cone}{\operatorname{C}}

\newcommand{\ak}{K}
\newcommand{\aks}{\kappa}
\newcommand{\blap}{\mathscr{L}}
\newcommand{\bric}{R(\infty)}
\newcommand{\nric}{R(N)}

\newcommand{\std}{\mathbb{V}^{\ast}}

\newcommand{\simplex}{\mathscr{S}}
\newcommand{\orthant}{\mathscr{Q}}
\newcommand{\tube}{\mathscr{T}}

\newcommand{\Om}{\Omega}
\newcommand{\lmg}{\mathscr{L}}

\newcommand{\cy}{\varphi}

\newcommand{\lc}{\Sigma}

\renewcommand{\H}{\mathsf{H}}

\newcommand{\cut}{\mathsf{Cut}}

\newcommand{\om}{\omega}

\newcommand{\hess}{\operatorname{\mathsf{Hess}}}

\newcommand{\vol}{\mathsf{vol}}

\newcommand{\delbar}{\bar{\del}}

\newcommand{\ka}{\kappa}

\newcommand{\sW}{\mathbb{W}}

\newcommand{\Id}{\text{Id}}

\newcommand{\sff}{\Pi}

\newcommand{\dum}{\,\cdot\,\,}
\newcommand{\Ga}{\Gamma}

\newcommand{\Q}{\mathcal{Q}}

\newcommand{\nm}{\mathsf{W}}
\newcommand{\nk}{\mathsf{K}}

\newcommand{\lap}{\Delta}

\renewcommand{\j}{\mathsf{i}}

\newcommand{\sOmega}{\mathbb{\Omega}}
\newcommand{\sG}{\mathbb{G}}
\newcommand{\tphi}{\tilde{\phi}}
\newcommand{\la}{\lambda}
\newcommand{\ep}{\epsilon}

\newcommand{\reat}{\mathbb{R}^{\times}}

\newcommand{\reap}{\mathbb{R}^{+}}

\newcommand{\cinf}{C^{\infty}}

\newcommand{\eno}{\text{End}}

\newcommand{\si}{\sigma}

\newcommand{\pr}{\partial}

\newcommand{\bnabla}{\bar{\nabla}}

\newcommand{\en}{[\nabla]}

\newcommand{\lie}{\mathfrak{L}}

\newcommand{\Aff}{\mathbb{Aff}}

\newcommand{\cp}{\mathcal{C}}

\newcommand{\ste}{\mathbb{V}}

\newcommand{\stez}{\mathbb{V} \setminus \{0\}}

\newcommand{\projp}{\operatorname{\mathbb{P}^{+}}}

\newcommand{\al}{\alpha}
\newcommand{\be}{\beta}

\newcommand{\hnabla}{\widehat{\nabla}}

\newcommand{\eul}{\mathbb{X}}

\newcommand{\nat}{\mathbb{N}}

\DeclareMathOperator{\Aut}{Aut}

\newcommand{\del}{\partial}
\newcommand{\tensor}{\otimes}

\newcommand{\rea}{\mathbb R}
\newcommand{\com}{\mathbb C}

\begin{document}
\title[Schwarz lemma for K\"ahler affine metrics]{A Schwarz lemma for K\"ahler affine metrics and the canonical potential of a proper convex cone}
\author{Daniel J.~F. Fox} 
\address{Departamento de Matem\'atica Aplicada\\ EUIT Industrial\\ Universidad Polit\'ecnica de Madrid\\Ronda de Valencia 3\\ 28012 Madrid Espa\~na}
\email{daniel.fox@upm.es}



\begin{abstract}
This is an account of some aspects of the geometry of K\"ahler affine metrics based on considering them as smooth metric measure spaces and applying the comparison geometry of Bakry-Emery Ricci tensors. Such techniques yield a version for K\"ahler affine metrics of Yau's Schwarz lemma for volume forms. By a theorem of Cheng and Yau there is a canonical K\"ahler affine Einstein metric on a proper convex domain, and the Schwarz lemma gives a direct proof of its uniqueness up to homothety. The potential for this metric is a function canonically associated to the cone, characterized by the property that its level sets are hyperbolic affine spheres foliating the cone. It is shown that for an $n$-dimensional cone a rescaling of the canonical potential is an $n$-normal barrier function in the sense of interior point methods for conic programming. It is explained also how to construct from the canonical potential Monge-Amp\`ere metrics of both Riemannian and Lorentzian signatures, and a mean curvature zero conical Lagrangian submanifold of the flat para-K\"ahler space.
\end{abstract}

\maketitle

\setcounter{tocdepth}{1}  

\section{Introduction}
By theorems of S.~Y. Cheng and S.~T. Yau, the geometry of a proper convex cone is encoded in a single function, called here the \textbf{canonical potential}, defined as the solution of a certain Monge-Amp\`ere equation on the cone and characterized by the property that its level sets are the hyperbolic affine spheres foliating the cone's interior. For a homogeneous cone, the canonical potential equals a scalar multiple of the logarithm of the usual characteristic function of the cone, and many of the nice properties of the characteristic function of a homogeneous cone are best understood as specializations of properties of the canonical potential of a general proper convex cone. The lift of the canonical potential to the tube domain over the cone is a potential for the unique complete K\"ahler Einstein metric on the tube, and Cheng and Yau's proofs of the existence and uniqueness of the canonical potential use their solution of the associated complex Monge-Amp\`ere equation and Yau's Schwarz lemma for volume forms on Hermitian manifolds. 

One theme of this article is that it is useful to regard a K\"ahler affine metric as a kind of smooth metric measure space. In particular, for a proper convex cone, the comparison geometry of Bakry-Emery Ricci tensors and the associated Bakry-Qian comparison theorem for the modified Laplacian can be used in place of the comparison geometry of a K\"ahler metric on the associated tube domain. As a principal example, this approach is used to prove a version for K\"ahler affine metrics of Yau's Schwarz lemma for volume forms on Hermitian manifolds. This yields a direct proof of the uniqueness of the canonical potential. It also gives a clear illustration of the philosophy behind the comparison geometry of Bakry-Emery Ricci tensors. One thinks of bounds on a Bakry-Emery Ricci tensor as corresponding to bounds on the ordinary Ricci tensor of some metric fibering over the given metric measure space. Here this idea has a concrete realization - the fiber space is the aforementioned tube domain. In this case, as in several others considered in the text, the novelty is more the point of view than the ultimate conclusion, although some of the results described do not seem to be as well known as they deserve to be.

It seems that for theoretical purposes the canonical potential of a cone is at least as useful as the usual characteristic function. As an illustration, it is proved that the canonical potential of an $n$-dimensional proper convex cone is an $n$-normal barrier for the convex cone in the sense of interior point methods for conic programming. The same result, with a universal constant times the dimension in lieu of the dimension alone, has been proved by Y. Nesterov and A. Nemirovskii for a multiple of the logarithm of the characteristic function called by them the \textit{universal barrier}. That the canonical potential has this property is a consequence of the nonpositivity of the ordinary Ricci tensor of the Hessian metric that it generates. While this can be deduced from the nonpositivity of the Ricci curvature of the equiaffine metric of its level sets proved by E. Calabi and a splitting theorem of J. Loftin, here it is proved directly as a further illustration of the utility in this context of the comparison geometry of the associated metric measure space described in the preceeding paragraph.

A secondary theme is that a number of interesting geometric structures associated to a convex cone can be constructed from its canonical potential. Beyond the affine spheres arising as the level sets of the canonical potential, there are a pair of Monge-Amp\`ere metrics with potentials equal to a function of the canonical potential, and the graph of the differential of the canonical potential is a mean curvature zero spacelike Lagrangian submanifold of the flat para-K\"ahler structure on the product of the cone with its dual.

\subsection{}
The remainder of the introduction describes the contents in detail and states the main results. This paragraph records some notational conventions used in the paper. The abstract index and summation conventions are used throughout the paper. In particular, indices serve as formal labels indicating tensor symmetries, and do not refer to any particular choice of basis unless so indicated. For example, $\delta_{i}\,^{j}$ indicates the canonical pairing between vectors and covectors, and $A_{p}\,^{p}$ indicates the trace of the endomorphism $A_{i}\,^{j}$. Enclosure of indices in square brackets (resp. parentheses) indicates complete skew-symmetrization (resp. symmetrization) over the enclosed indices, so that, for instance, the decomposition of a covariant bivector into its skew and symmetric parts is $A_{ij} = A_{[ij]} + A_{(ij)}$. An enclosed index delimited by vertical bars is to be skipped in such a (skew)-symmetrization, e.g. $2A_{[i|j|k]} = A_{ijk} - A_{kji}$.

\subsection{}\label{kahleraffinesection}
Let $M$ be a smooth manifold. Following \cite{Cheng-Yau-realmongeampere}, a pair $(\nabla, g)$ comprising a flat affine connection $\nabla$ and a pseudo-Riemannian metric $g$ on $M$ such that $\nabla_{[i}g_{j]k} = 0$ is called a \textbf{K\"ahler affine metric}. This means that every point of $M$ has an open neighborhood in which there is a \textbf{potential} function $F$ such that $g_{ij} = \nabla_{i}dF_{j}$. Everywhere in what follows, indices are raised and lowered using $g_{ij}$ and the inverse symmetric bivector $g^{ij}$. A K\"ahler affine metric will be said to be a \textbf{Hessian metric} if there is a globally defined potential. While this terminological distinction between \textit{K\"ahler affine} and \textit{Hessian} metrics is not completely standard, it seems to be useful. Background on these structures can be found in H. Shima's book \cite{Shima}. K\"ahler affine metrics were first studied as such by S.~Y. Cheng and S.~T. Yau in \cite{Cheng-Yau-realmongeampere} and H. Shima in \cite{Shima-compacthessian}. A slightly restricted special case had been studied earlier in \cite{Koszul-convexite} and \cite{Koszul-deformations}, where J.~L. Koszul proved that the universal cover of a compact flat affine manifold $(M, \nabla)$ is a divisible convex domain if and only if $M$ admits a closed one-form $\al$ such that $\nabla \al$ is a Riemannian (and so K\"ahler affine) metric (by a theorem from J. Vey's \cite{Vey}, the convex domain must be cone).

For a flat affine connection $\nabla$ and a smooth (meaning $\cinf$) function $F$, write 
\begin{align}
F_{i_{1}\dots i_{k}} = \nabla_{i_{1}}\dots\nabla_{i_{k-1}}dF_{i_{k}}.
\end{align} 
Since the tensor $g_{ij} = F_{ij}$ and its covariant derivatives are globally defined on a K\"ahler affine manifold, even though the potential $F$ need not be, for any $k \geq 0$ it makes sense to write $F_{i_{1}\dots i_{k}ij}$ for $\nabla_{i_{1}}\dots\nabla_{i_{k}}g_{ij}$, and this can be interpreted as a derivative of any local potential of $g$, or of a global potential of $g$, should such exist. In particular, the completely symmetric cubic tensor $F_{ijk} = \nabla_{i}g_{jk}$ and its contraction $F_{ip}\,^{p} = g^{pq}F_{ipq}$ will be used frequently.

Given a flat affine connection $\nabla$, there is associated to any nowhere vanishing density $\om$ of nonzero weight the closed one-form $\nabla\log \om = \om^{-1}\tensor \nabla \om$. For a K\"ahler affine structure $(\nabla, g)$, $\det g$ is a nonvanishing $2$-density, and the associated $1$-form $H_{i} = \nabla_{i} \log \det g = F_{ip}\,^{p}$ is the \textbf{Koszul form} of the K\"ahler affine structure. The \textbf{(K\"ahler affine) Ricci tensor} $\ak_{ij}$ and \textbf{(K\"ahler affine) scalar curvature} $\aks$ of a K\"ahler affine metric are defined by $\ak_{ij} = -\nabla_{i}H_{j}$ and $\aks = g^{ij}\ak_{ij}$. Note that the K\"ahler affine Ricci tensor and scalar curvature are not the same as the Ricci tensor and scalar curvature of $g$. A K\"ahler affine metric is said to be \textbf{Einstein} if $\ak_{ij}$ is a multiple of $g_{ij}$. Lemma \ref{keakslemma} shows that in this case $\aks$ is constant.

Given a fixed $\nabla$-parallel volume form $\Psi$, define a differential operator $\H(F)$ on functions by $\det \nabla dF = \H(F)\Psi^{2}$. Where $\H(F)$ is nonzero, $g_{ij} = F_{ij}$ is a pseudo-Riemannian Hessian metric and $H_{i} = d_{i}\log\H(F)$. In this case $\ak_{ij} = -\nabla_{i}d\log\H(F)_{j}$. In particular, any solution of $\H(F) = ae^{bF}$, where $a \in \reat$ and $b \in \rea$, gives a Hessian metric which is Einstein in the K\"ahler affine sense.

The automorphism group $\Aff(n+1, \rea)$ of $\nabla$ comprises the affine transformations of $\rea^{n+1}$. On $\rea^{n+1}$, $\H(F)$ is defined with respect to the standard volume form, which is always denoted $\Psi$. Let $\ell:\Aff(n+1, \rea) \to GL(n+1, \rea)$ be the projection onto the linear part, so that the kernel of the character $\det^{2}\ell:\Aff(n+1, \rea) \to \reap$ comprises the \textbf{unimodular} affine transformations preserving $\Psi$ up to sign. 
For $F \in \cinf(\rea^{n+1})$ and $g \in \Aff(n+1, \rea)$ define $(g \cdot F)(x) = F(g^{-1}x)$. Then $\H(F)$ is affinely covariant and equiaffinely invariant in the sense that  
\begin{align}\label{hgf}
g\cdot \H(F) = \det\,\!\!^{2}\ell(g) \H(g\cdot F).
\end{align}
In particular the K\"ahler affine Ricci tensor of the metric generated by $g\cdot F$ is the pullback via the action of $g$ of that generated by $F$.

A \textbf{smooth} or \textbf{Riemannian metric measure space} $(M, k, \phi)$ is a manifold $M$ equipped with a Riemannian metric $k$ and a positive function $\phi \in \cinf(M)$, which is identified with the volume form $\phi\vol_{k}$. Here the modifier \textit{smooth} will usually be omitted and these spaces will be referred to simply as \textit{metric measure spaces} or \textbf{mm-spaces}. A slightly more general structure is given by a Riemannian metric $k$ and a closed one-form $\al$. There seems to be no established terminology for a manifold $M$ equipped with such a pair $(k, \al)$; here it will be called a \textbf{local (smooth, Riemannian) mm-space}. The mm-space $(M, k, \phi)$ is a local mm-space with one-form $\al = d\log \phi$. A local mm-structure $(k, \phi)$ is \textbf{complete} if $k$ is complete. 

The \textbf{$(N+n)$-dimensional Bakry-Emery Ricci tensor} $\nric_{ij}$ associated to the $n$-dimensional local mm-space $(M, k, \al)$ is defined by 
\begin{align}\label{bedefined}
\nric_{ij} = R_{ij} - D_{i}\al_{j} - \tfrac{1}{N}\al_{i}\al_{j}, 
\end{align}
in which $R_{ij}$ is the Ricci tensor of the Levi-Civita connection $D$ of $k_{ij}$. If $\al = d\log\phi$, then $\nric_{ij} =  R_{ij} - N\phi^{-1/N}D_{i}d(\phi^{1/n})_{j}$ and the \textbf{$\infty$-(Bakry-Emery)-Ricci tensor} $\bric_{ij} = R_{ij} - D_{i}\al_{j}$ results formally when $N \to \infty$. The $\infty$-Ricci tensor was introduced by A. Lichnerowicz in \cite{Lichnerowicz-ctensor} where he used it to prove a generalization of the Cheeger-Gromoll splitting theorem.

The idea related to these tensors relevant in what follows is that a lower bound on the $(N+n)$-dimensional Bakry Emery Ricci tensor of the metric measure structure $(k, \phi)$ on the $n$-manifold $M$ corresponds to a lower bound on the usual Ricci curvature of some metric associated to $(k, \phi)$ on some $(n+N)$-dimensional manifold fibering over $M$. In the example relevant here, $N = n$ and the fibering manifold is the tube domain over an $n$-dimensional convex cone.

One aim of this paper is to explain that it is profitable to regard a Riemannian Hessian metric $g_{ij} = \nabla_{i}dF_{j}$ as the mm-space determined by $g_{ij}$ in conjunction with the volume form $\H(F)\Psi = \H(F)^{1/2}\vol_{g}$, which is the pullback via the differential $dF$ of the parallel volume form dual to $\Psi$ on the dual affine space $\rea^{n+1\,\ast}$. More generally, a K\"ahler affine structure $(g, \nabla)$ is identified with the local mm-structure formed by $g$ and the half Koszul one-form $\tfrac{1}{2}H_{i}$. 

The specialization to local mm-spaces of a distance comparison theorem of D. Bakry and Z. Qian is stated here as Theorem \ref{bqcomparisontheorem}. The proof of the Schwarz lemma for volume forms on Hermitian manifolds given by Yau and N. Mok in \cite{Mok-Yau} is adapted to the setting of K\"ahler affine metrics by substituting for the usual Riemannian distance comparison theorem the Bakry-Qian distance comparison theorem applied to the local mm-space associated to the K\"ahler affine metric. This theorem involves the modified Laplacian $\lap_{g} + \tfrac{1}{2}H^{p}D_{p}$ instead of simply $\lap_{g}$. There results Theorem \ref{schwarztheorem}, a version for K\"ahler affine metrics of Yau's Schwarz lemma for volume forms (see \cite{Yau-schwarz}, \cite{Mok-Yau}, \cite{Chen-Cheng-Lu}, and \cite{Tosatti-schwarz}).

\begin{theorem}\label{schwarztheorem}
Let $\nabla$ be a flat affine connection on the $(n+1)$-dimensional manifold $M$ and let $g_{ij}$ be a complete Riemannian metric on $M$ constituting with $\nabla$ a K\"ahler affine structure such that the K\"ahler affine Ricci curvature $\ak_{ij}$ of $g$ is bounded from below by a multiple of $g$ and the K\"ahler affine scalar curvature $\aks$ of $g$ is bounded from below by $-A(n+1)$ for some positive constant $A$. Suppose $\om$ is a $C^{2}$ volume form on $M$ such that $-\nabla_{i}\nabla_{j}\log \om$ is negative definite and satisfies $2\det \nabla \nabla\log \om \geq B\om^{2}$ for some positive constant $B$. Then $(\om/\vol_{g})^{2} \leq A^{n+1}/B$.
\end{theorem}
\noindent
For a Hessian metric $g_{ij} = \nabla_{i}dF_{j}$ with $\nabla$-parallel volume form $\Psi$, the conclusion of Theorem \ref{schwarztheorem} is stated as follows. Consider a volume form $\om$ and write $\om^{2} = V\Psi^{2}$. The hypotheses on $\om$ are equivalent to the negativity of $-\nabla_{i}d\log V_{j}$ and $\H(\log V) \geq BV$. The conclusion is that $V/\H(F) \leq A^{n+1}/B$.

\subsection{}
A \textbf{domain} means a nonempty open subset of $\rea^{n+1}$. A convex domain $\Omega \subset \rea^{n+1}$ is \textbf{proper} if its closure contains no complete affine line. A subset $\Omega \subset \rea^{n+1}$ is a \textbf{cone} if $e^{t}x \in\Omega$ whenever $x \in \Omega$. A convex cone is proper if and only if the open dual cone
\begin{align*}
\Omega^{\ast} =\{y \in \rea^{n+1\,\ast}: x^{p}y_{p} > 0 \,\,\text{for all}\,\, x \in \bar{\Omega}\setminus\{0\}\} 
\end{align*}
is not empty (for background on convex cones see chapter I of \cite{Faraut-Koranyi} or \cite{Vey}). The \textbf{automorphism group} $\Aut(\Om)$ comprises $g \in \Aff(n+1,\rea)$ preserving $\Om$.

When $\Om$ is a proper convex domain, Theorem \ref{schwarztheorem} yields
\begin{corollary}\label{gfcorollary}
Let $\Omega \subset \rea^{n+1}$ be a proper convex domain and let $F \in \cinf(\Omega)$ be a convex solution of $\H(F) = e^{2F}$ such that $g_{ij} = \nabla_{i}dF_{j}$ is a complete Riemannian metric on $\Omega$. If $G \in C^{2}(\Omega)$ is convex and satisfies $\H(G) \geq e^{2G}$ then $G\leq F$.
\end{corollary}
\begin{proof}
Theorem \ref{schwarztheorem} applies with $\om = e^{G}\Psi$, $B = 2^{n+1}$, and $A = 2$.
\end{proof}
\noindent
It follows from Corollary \ref{gfcorollary} that if there is a complete K\"ahler affine Einstein metric on $\Om$, then it is unique up to homothety, and $\Aut(\Om)$ acts on it by isometries (see \eqref{cyhom} of Theorem \ref{cytheorem}). In particular, in the case that $\Om$ is a proper convex cone, taking $g \in \Aut(\Om)$ to be a radial dilation, this implies that $F$ is logarithmically homogeneous.

 The K\"ahler affine Ricci tensor is defined in analogy with the Ricci form of a K\"ahler metric, and K\"ahler affine metrics were introduced as real analogues of K\"ahler metrics. Actually the relation is more precise, as is explained now, following \cite{Cheng-Yau-realmongeampere}. 
The \textbf{tube domain}  $\pi:\tube_{\Omega} \to \Omega$ over $\Omega \subset \rea^{n+1}$ is the subset $\tube_{\Omega} = \Omega + \j \rea^{n+1}$ of $\com^{n+1}$ fibered over $\Om$ via the projection $\pi(z^{j}) = \tfrac{1}{2}(z^{j} + \bar{z}^{j})$. The inclusion $\Omega \to \tube_{\Omega}$ is totally real. If the manifold $M$ is equipped with a flat affine connection $\nabla$, then the tube domains over affine coordinate charts patch together to give a complex manifold $\tube_{M}$ and a fibration $\pi:\tube_{M} \to M$ with totally real section $M \to \tube_{M}$. If $A \in \cinf(\Omega)$ is convex then $\om = dd^{c}\pi^{\ast}(A) = 2\j \del \delbar \pi^{\ast}(A) = \tfrac{\j}{2}A_{ij}(x)dz^{i}\wedge d\bar{z}^{j}$ is the K\"ahler form of the K\"ahler metric $G_{i\bar{j}} = A_{ij}(x)dz^{i}\tensor d\bar{z}^{j}$, and $G_{i\bar{j}}$ is complete if and only if the Hessian metric with potential $A$ is complete. In the same way a K\"ahler affine metric on $M$ lifts to a K\"ahler metric on $\tube_{M}$. The Ricci form of the K\"ahler affine metric with local potential $A$ is simply the restriction to $M$ of the Ricci form of the K\"ahler metric on $\tube_{M}$ with K\"ahler potential $\pi^{\ast}(A)$ (this accounts for the sign in the definition of $\ak_{ij}$).

Corollary \ref{gfcorollary} can also be proved by considering the K\"ahler metrics generated on $\tube_{\Omega}$ by the potentials $\pi^{\ast}(F)$ and $\pi^{\ast}(G)$ and applying Yau's Schwarz lemma for volume forms. This is explained in detail in section \ref{uniquenessexplanationsection}. The lower bound on the Ricci curvature of the K\"ahler metric on $\tube_{\Omega}$ plays a key role in the proof of Yau's Schwarz lemma, as it gives, via the usual Riemannian comparison theorems, control over volume and distance on $\tube_{\Omega}$. However, it is not clear what these conditions mean directly for the geometry of the underlying Hessian metric. This is precisely the sort of situation for which the Bakry-Emery Ricci tensors and the corresponding comparison geometry are well suited. On page $350$ of \cite{Cheng-Yau-realmongeampere} the analytic difficulties associated with the non-self-adjointness  with respect to $\vol_{g}$ of the modified Laplacian $\lap_{g} + \tfrac{1}{2}H^{p}D_{p}$ associated to the K\"ahler affine metric $(\nabla, g)$ are given as motivation for working instead on the tube domain. On the other hand, the operator $\lap_{g} + \tfrac{1}{2}H^{p}D_{p}$ is formally self-adjoint with respect to the volume $\H(F)\Psi = \H(F)^{1/2}\vol_{g}$, and so the same considerations suggest the viability of working directly with the mm-space $(g, \H(F)^{1/2})$. Via \eqref{cd1}, a lower bound on the K\"ahler affine Ricci tensor yields a lower bound on the $2(n+1)$-dimensional Bakry-Emery Ricci tensor of this mm-space, and using this bound, the corresponding distance comparison theorem for the modified Laplacian $\lap_{g} + \tfrac{1}{2}H^{p}D_{p}$ due to Bakry-Qian leads to the Schwarz lemma for volume forms of Hessian metrics, Theorem \ref{schwarztheorem}. The overall strategy of the proof is just as in the K\"ahler case, with the Bakry-Emery Ricci tensor and modified Laplacian in place of the usual ones. This provides a compelling example illustrating the utility of the formalism of metric measure spaces and Bakry-Emery Ricci tensors, in which it can be seen quite clearly how bounds on the Bakry-Emery Ricci tensor encode bounds on curvatures obtained from some manifold fibering over the one of interest. 

\subsection{}
From the affine covariance \eqref{hgf} it follows that the K\"ahler affine geometry of a Hessian metric with potential $F$ is closely related to the equiaffine geometry of the level sets of $F$. Let $I \subset \rea$ be a connected open interval and let $\psi:I \to \rea$ be a $C^{2}$ diffeomorphism. Let $\Omega_{I} = F^{-1}(I)\cap \Omega$. The level sets of $F$ and $\psi \circ F$ are the same, just differently parameterized, in the sense that for $r \in I$ there holds $\lc_{r}(F, \Omega_{I}) = \lc_{\psi(r)}(\psi \circ F, \Omega)$. For a diffeomorphism $\psi:I \to \rea$ there holds
\begin{align}\label{hpsif}
\H(\psi(F)) = \dot{\psi}^{n+1}(1 + \ac(\psi)|dF|^{2}_{g})\H(F).
\end{align}
in which $\ac(\psi) = d\log \dot{\psi} = \ddot{\psi}/\dot{\psi}$. As a consequence of \eqref{hpsif}, conditions like the K\"ahler affine Einstein condition tend also to fix the external reparameterization $\psi$. That is, such a condition imposes some coherence condition on the family of level sets of $F$. In particular, for a logarithmically homogeneous potential $F$ the K\"ahler affine Einstein equations are essentially equivalent to the statement that its level sets are affine spheres. After recalling the necessary terminology and introducing some notation, the precise statement is given below as Theorem \ref{ahtheorem}.

Let $\nabla$ be the standard flat affine connection on $\rea^{n+1}$. It preserves the standard volume form $\Psi$. A co-oriented hypersurface $\Sigma$ in flat affine space is \textbf{nondegenerate} if its second fundamental form is nondegenerate. A transverse vector field $N$ defined in a neighborhood of $\Sigma$ determines a splitting of the pullback of $T\rea^{n+1}$ over $\Sigma$ as the direct sum of $T\Sigma$ and the span of $N$. Via this splitting, the connection $\nabla$ induces on $\Sigma$ a connection $\inabla$, while via $N$, the second fundamental form is identified with a symmetric covariant two tensor $h$ on $\Sigma$. Additionally there are determined the affine shape operator $S \in \Ga(\eno(T\Sigma))$ and the connection one-form $\tau \in \Ga(T^{\ast}\Sigma)$. For vector fields $X$ and $Y$ tangent to $\Sigma$ the associated connection $\inabla$ and tensors $h$, $S$, and $\tau$ are related by
\begin{align}\label{induced}
&\nabla_{X}Y = \inabla_{X}Y + h(X, Y)N,& &\nabla_{X}N = -S(X) + \tau(X)N.
\end{align}
Here, as in what follows, notation indicating the restriction to $\Sigma$, the immersion, the pullback of $T\rea^{n+1}$, etc. is omitted. Tensors on $\Sigma$ are labeled using capital Latin abstract indices. Let $h^{IJ}$ be the bivector dual to $h_{IJ}$. The \textbf{equiaffine} normal vector field $\nm$ is determined uniquely by the requirements that it be co-oriented, that $n\tau_{I} + h^{PQ}\inabla_{I}h_{PQ} = 0$, and that the induced volume density $|i(\nm)\Psi|$ equal the volume density of the metric $h$. These conditions imply, in particular, that the equiaffine connection one-form $\tau_{I}$ vanishes identically. The corresponding $h_{IJ}$ and $S_{I}\,^{J}$ are the \textbf{equiaffine metric} and \textbf{equiaffine shape operator}.  
The \textbf{affine mean curvature} is the arithmetic mean $\amc = \tfrac{1}{n}S_{I}\,^{I}$ of the eigenvalues of the equiaffine shape operator. The distinguished affinely invariant line field on $\Sigma$ spanned by $\nm$ is called the \textbf{affine normal distribution}.

A co-orientable locally uniformly convex hypersurface $\Sigma \subset \rea^{n+1}$ is co-oriented so that a co-oriented transverse vector field points to the \textbf{convex} (or \textbf{interior}) side of $\Sigma$, namely that side in the direction of which a parallel translate of a supporting hyperplane intersects the hypersurface. In this case the equiaffine metric is Riemannian and $\Sigma$ is said to be \textbf{complete} if the equiaffine metric is complete.

A nondegenerate hypersurface $\Sigma$ is a \textbf{proper affine sphere} if its equiaffine normals meet in a point, its \textbf{center}, and an \textbf{improper affine sphere}, if they are parallel, in which case $\Sigma$ is said to have \textbf{center at infinity}. An equivalent definition is that the affine shape operator is a multiple of the identity. It follows from the Gauss-Codazzi equations that in this case $\amc$ is constant. A locally uniformly convex affine sphere is \textbf{hyperbolic}, \textbf{parabolic}, or \textbf{elliptic}, according to whether its affine mean curvature is negative, zero, or positive. Clearly $\Sigma$ is a parabolic affine sphere if and only if its equiaffine normals are parallel, while $\Sigma$ is an elliptic or hyperbolic affine sphere if and only if it is proper and its center is in its interior or exterior, respectively. 

A function $F$ is \textbf{$\al$-logarithmically homogeneous} on the open subset $\Om \in \rea^{n+1}$ if $F(e^{t}x) = F(x) + \al t$ for all $t \in \rea$ and $x \in \Om$ such that $e^{t}x \in \Om$. The space of $\cinf$ smooth $\al$-logarithmically homogeneous functions on $\Om$ is written $\lmg_{\al}(\Om)$. For an open subset $\Om \subset \rea^{n+1}$ and $F \in \cinf(\Omega)$, let $\lc_{r}(F, \Omega) = \{x \in \Omega: F(x) = r\}$. 

\begin{theorem}\label{ahtheorem}
Let $\al < 0$. Let $\Omega \subset \rea^{n+1}$ and $I \subset \rea$ be nonempty, connected open subsets. For $F \in \lmg_{\al}(\Omega)$, let $\Omega_{I} = F^{-1}(I)\cap \Omega$ and suppose $g_{ij} = \nabla_{i}dF_{j}$ is positive definite on $\Omega_{I}$. The following are equivalent:
\begin{enumerate}
\item \label{aht1} For all $r \in I$ each connected component of $\lc_{r}(F, \Omega_{I})$ is a hyperbolic affine sphere with center at the origin.
\item \label{aht2} There is a nonvanishing function $\phi:I \to \rea$ such that $F$ solves $\H(F) = \phi(F)$ on $\Omega_{I}$.
\end{enumerate}
In this case there is a constant $B\neq 0$ such that $\phi(r) = Be^{-2(n+1)r/\al}$ and the affine mean curvature of $\lc_{r}(F, \Omega_{I})$ is 
\begin{align}\label{ahmc}
\begin{split}
&-|\al|^{-(n+1)/(n+2)}|B|^{1/(n+2)}e^{-2(n+1)r/\al(n+2)}.
\end{split}
\end{align} 
\end{theorem}
\noindent
Note that no regularity assumptions are made on the function $\phi$ of \eqref{aht2}, although such conditions follow automatically from the logarithmic homogeneity of $F$. Theorem \ref{ahtheorem} is proved at the end of section \ref{affinesection}. With minor modifications, Theorem \ref{ahtheorem} is true in arbitrary signatures (see \cite{Fox-autoiso}). (Such a modification is necessary even to include the case of elliptic affine spheres, as these are the level sets of a potential for a Lorentzian Hessian metric.)

If $F$ solves $\H(F) = \phi(F)$ then $t\cdot F(x) = F(e^{-t}x)$ solves $\H(t\cdot F) = e^{-2(n+1)t}t\cdot \H(F)$, so $B$ can be taken to be $\pm 1$ by replacing $F$ by $t\cdot F(x)$ with $2(n+1)t = \log|B|$. Put in another manner, the value of the level has no intrinsic meaning, although the difference or ratio of the values of two levels has.

\subsection{}
In fact, the classification of complete hyperbolic affine spheres can be founded on the study of the equation $\H(F) = e^{2F}$. While this must have been understood in some form to Cheng and Yau, it does not seem well known, except for accounts by J. Loftin in \cite{Loftin-affinekahler} and \cite{Loftin-survey}, and so is described here in some detail. 

The locally uniformly convex affine spheres are constructed in full generality due to work of Cheng and Yau in \cite{Cheng-Yau-mongeampere}, \cite{Cheng-Yau-realmongeampere}, and \cite{Cheng-Yau-affinehyperspheresI}. Precisely, they show that there is a unique foliation of the interior of a proper open convex domain $\Omega \subset \rea^{n+1}$ by affine complete properly embedded hyperbolic affine spheres having center at the vertex of $\Om$ and asymptotic to its boundary. A precise statement containing some additional information is given as Theorem \ref{affinespheretheorem} below. Results of this sort were conjectured quite precisely by E. Calabi in \cite{Calabi-completeaffine} and \cite{Calabi-nonhomogeneouseinstein}. The underlying ideas originate with C. Loewner and L. Nirenberg's \cite{Loewner-Nirenberg}. Cheng and Yau attribute part of their theorem to independent unpublished work of Calabi and Nirenberg, and presumably for this reason did not publish all the details of the argument in a single place, so it is useful also to consult \cite{Sasaki}, \cite{Loftin-survey},  \cite{Gigena-integralinvariants}, \cite{Gigena}, \cite{Li-1}, \cite{Li-2}, and \cite{Trudinger-Wang-survey}. 

A complete proof of the Cheng-Yau Theorem has three basic parts: the existence of the affine spheres, obtained by solving some Monge-Amp\`ere equation; the uniqueness, obtained by some sort of Schwarz lemma; and the extrinsic claims regarding completeness and the asymptotic properties. These last claims and the relation between affine and Euclidean completeness will not be discussed here, as they are now understood in a more general context due to the work of N. Trudinger and X.-J. Wang; see \cite{Trudinger-Wang-survey} and \cite{Trudinger-Wang-affinecomplete}. As is recalled briefly in section \ref{slicesection}, the existence part of the theorem is usually proved by appealing to a theorem of Cheng-Yau which produces a negative convex solution to the Dirichlet problem for the equation $u^{n+2}\H(u) = (-1)^{n}$ on the set of rays $\projp(\Om)$ in $\Om$. In section \ref{cysection} below there is described how the existence can be based on the following theorem of Cheng and Yau (resolving a conjecture made by Calabi on page $19$ of \cite{Calabi-nonhomogeneouseinstein}).
\begin{theorem}[S.Y.Cheng and S.T.Yau, \cite{Cheng-Yau-completekahler}, \cite{Cheng-Yau-realmongeampere}, \cite{Cheng-Yau-affinehyperspheresI}]\label{cytheorem}
On a proper open convex domain $\Omega \subset \rea^{n+1}$ there exists a unique smooth convex function $F:\Omega \to \rea$ solving $\H(F) = e^{2F}$, tending to $+\infty$ on the boundary $\pr \Omega$ of $\Omega$, and such that $g_{ij} = \nabla_{i}dF_{j}$ is a complete Riemannian metric on $\Omega$ forming with the standard flat affine connection $\nabla$ a K\"ahler affine Einstein metric with K\"ahler affine scalar curvature $-2$. Moreover, this $F$ has the following properties:
\begin{enumerate}
\item \label{cyhom} $F(gx) = F(x) - \log\det \ell(g)$ for all $g \in \Aut(\Omega)$. This implies $g_{ij}$ is $\Aut(\Omega)$ invariant. 
\item \label{cysup} For all $x \in \Om$,
\begin{align} 
F(x) = \sup\{G(x): G \in C^{2}(\Omega), G \,\, \text{convex in} \,\,\Omega, \,\,\text{and}\,\, \H(G) \geq e^{2G}\}
\end{align}
\end{enumerate}
\end{theorem}
\noindent
Given a proper open convex domain $\Omega$, the unique solution $F$ of $\H(F) = e^{2F}$ such that $g_{ij} = \nabla_{i}dF_{j}$ is a complete metric on $\Omega$ given by Theorem \ref{cytheorem} will be called the \textbf{canonical potential} of $\Omega$, and the Hessian metric $g_{ij}$ will be called the \textbf{canonical metric} of $\Om$. Sometimes there will be written $F_{\Om}$ to indicate the dependence of $F$ on $\Om$.

\begin{proof}[Proof of Theorem \ref{cytheorem}]
The solvability of the equation on a bounded convex domain follows from Corollary $7.6$ of \cite{Cheng-Yau-completekahler}. In section $4$ of \cite{Cheng-Yau-realmongeampere}, Cheng and Yau resolved the unbounded case, constructing the solution as a limit of solutions on bounded convex domains exhausting the domain $\Omega$. A self-contained proof of the existence is given also in \cite{Guan-Jian}. The uniqueness of the solution follows from Corollary \ref{gfcorollary}. Although the uniqueness in the unbounded case was not stated in \cite{Cheng-Yau-realmongeampere}, it was surely known to the authors, as it follows from Yau's generalized Schwarz lemma for volume forms on Hermitian manifolds, in the form stated in \cite{Mok-Yau} or \cite{Chen-Cheng-Lu}, applied on the tube domain over $\Omega$, as is explained in section \ref{uniquenessexplanationsection}.

Let $g \in\Aff(n+1, \rea)$ and let $F$ be the canonical potential of $\Om$. Since $\nabla d(g\cdot F - \log |\det \ell(g)|) = \nabla d(g\cdot F) = L_{g^{-1}}^{\ast}(\nabla dF)$, where $L_{g}$ denotes the action of $g$ by left multiplication action, the Hessian metric determined by $g\cdot F - \log |\det \ell(g)|$ is complete since the Hessian metric determined by $F$ is. By \eqref{hgf}, if $g \in\Aff(n+1, \rea)$, then $\H(g \cdot F - \log |\det \ell(g)|) = e^{2(g \cdot F - \log |\det \ell(g)|)}$. It follows from the uniqueness of the canonical potential that  $g\cdot F - \log |\det \ell(g)|$ is the canonical potential of $g\Om$. In particular, if $g \in \Aut(\Om)$ it must be $g\cdot F - \log |\det \ell(g)| = F$. This proves \eqref{cyhom}.

The characterization \eqref{cysup} of $F(x)$ as the supremum of $G(x)$ taken over all $C^{2}$ convex functions $G$ satisfying $\H(G) \geq 2G$ is immediate from Corollary \ref{gfcorollary}. 
\end{proof}

\subsection{}
The canonical potential $F$ of the proper convex cone $\Om$ is a counterpart to the usual (Koecher-Koszul-Vey) \textbf{characteristic function} $\phi_{\Omega}$ of $\Omega$, which is the positive homogeneity $-n-1$ positive function defined by 
\begin{align}\label{usualchardefined}
\phi_{\Omega}(x)  = \int_{\Omega^{\ast}}e^{-x^{p}y_{p}}dy
= n!\int_{S\cap \Omega^{\ast}} (x^{p}v_{p})^{-n-1}\,d\si(v) = n!\, \vol(\{y \in \Omega^{\ast}: x^{p}y_{p} = 1\}),
\end{align}
where $S$ is the Euclidean unit sphere in $\rea^{n+1\,^{\ast}}$, and $d\si$ is the induced volume on $S$. For the convergence of the integral defining $\phi_{\Omega}$ and the basic properties of $\phi_{\Omega}$ see \cite{Vinberg}, \cite{Faraut-Koranyi}, or \cite{Vey}. Relevant here are that $\phi_{\Omega}(x) \to +\infty$ uniformly as $x$ tends to the boundary of $\Omega$, that $\nabla d \phi_{\Omega}$ and $\nabla d \log \phi_{\Omega}$ are positive definite, and that $g\cdot \phi_{\Omega} = |\det \ell(g)|\phi_{g\cdot \Omega}$ for any $g \in \Aff(n+1, \rea)$. 

As is explained in section \ref{homogeneoussection}, for a homogeneous convex cone $\Om$ the function $\log \phi_{\Om}$ solves $\H(G) = ce^{2G}$ for some positive constant $c$. It follows that in this case, as a consequence of the uniqueness of the canonical potential, $\phi_{\Om}$ equals a constant multiple of $e^{F_{\Om}}$. In section \ref{characteristicfunctionsection} it is shown that most of the nice properties of the usual characteristic function for a homogeneous cone are valid on any proper convex cone for $e^{F_{\Om}}$, and so should be understood as consequences of the identification of $\phi_{\Om}$ with a multiple of $e^{F_{\Om}}$ in the homogeneous case.

\subsection{}
Theorem \ref{parakahlertheorem} of section \ref{parakahlersection} shows that for the canonical potential $F$ of a proper open convex cone $\Om$ the graph of minus the differential of the positive homogeneity $2$ function $u = -((n+1)/2)e^{-2F/(n+1)}$ is a mean curvature zero nondegenerate conical Lagrangian submanifold of the canonical flat para-K\"ahler structure on $\Om \times \Om^{\ast}$. Closely related constructions have been made by the author in \cite{Fox-ahs} and \cite{Fox-2dahs} and by R. Hildebrand in \cite{Hildebrand-parakahler} and \cite{Hildebrand-crossratio}. By \eqref{hpsif}, $u$ satisfies $\H(u) = -1$, and so the ordinary graph of $u$ is an improper affine sphere. Although the Hessian of $u$ has Lorentzian signature, this result is suggestive in light of M. Warren's result in \cite{Warren} showing that the graph of the differential of a $C^{2}$ convex function over a bounded, simply connected domain with $C^{1}$ boundary is, in the metric induced by the ambient flat para-K\"ahler metric, the unique volume maximizer among spacelike, oriented submanifolds in its homology class if the graph of the function itself is an open subset of an improper affine sphere. Other interesting properties of the function $u$ are given in Theorem \ref{lmatheorem} below.

\subsection{}
The following notions were introduced by Y. Nesterov and A. Nemirovskii, \cite{Nesterov-Nemirovskii}, (see also chapter $4$ of \cite{Nesterov} and the survey \cite{Nemirovski-Todd}) in the context of interior-point methods for the resolution of convex programming problems. For $\al > 0$, a function $F$ on an open convex set $\Om \subset \rea^{n+1}$ is \textbf{$\al$-self-concordant} if $F$ is at least three times differentiable and convex on $\Om$ and there hold
\begin{enumerate}
\item $F$ is a \textbf{barrier} for $\Om$ in the sense that $F(x_{i}) \to \infty$ for every sequence $\{x_{i}\} \in \Om$ converging to a point of the boundary $\pr \Om$.
\item 
\begin{align}\label{sc1}
\al (v^{i}v^{j}v^{k}F_{ijk}(x))^{2} \leq 4(v^{i}v^{j}F_{ij}(x))^{3} = 4|v|_{g}^{6},
\end{align}
for all $v \in \rea^{n+1}$ and $x \in \Om$, where $g_{ij} = \nabla_{i}dF_{j}$. 
\end{enumerate}
The notion of self-concordance is affinely invariant in the sense that if $F$ is self-concordant on $\Om$ then $g \cdot F$ is self-concordant on $g \Om$ for $g \in \Aff(n+1, \rea)$. By Corollary $2.1.1$ of \cite{Nesterov-Nemirovskii}, a self-concordant function is nondegenerate if its Hessian is nondegenerate at a single point. Note that if $F$ is $\al$-self-concordant then $\al^{-1} F$ is $1$-self-concordant. 
A $1$-self-concordant function for which there is a constant $\nu \geq 1$ such that for all $x \in \Om$ and $v \in \rea^{n+1}$ there holds
\begin{align}\label{sc2}
(v^{i}F_{i}(x))^{2} \leq \nu F_{ij}(x)v^{i}v^{j} = \nu |v|_{g}^{2},
\end{align}
is called a \textbf{self-concordant barrier with parameter $\nu$} for (the closure of) $\Om$. By the Schwarz inequality, the condition \eqref{sc2} is automatic if $F$ is $(-\nu)$-logarithmically homogeneous. A $(-\nu)$-logarithmically homogeneous $1$-self-concordant function for $\Om$ is called a \textbf{$\nu$-normal barrier} for $\Om$. 

\begin{theorem}\label{sctheorem}
The canonical potential $F$ of the proper open convex cone $\Om \in \rea^{n+1}$, is an $(n+1)$-normal barrier function for $\Om$.
\end{theorem}
I thank Roland Hildebrand for bringing to my attention that he has independently obtained Theorem \ref{sctheorem}, which appears as Theorem $1$ of his paper \cite{Hildebrand-hessian} (see also remarks in the introduction to his \cite{Hildebrand-norm}). O. G\"uler has informed me that he conjectured a result like Theorem \ref{sctheorem} more than a decade ago. Although here no direct use is made of their results, G\"uler's papers \cite{Guler-barrierfunctions}, and \cite{Guler-hyperbolic}, and \cite{Guler-selfconcordance} brought to my attention the connection between self-concordance and the canonical potential. Theorem $2.5.1$ of \cite{Nesterov-Nemirovskii} shows that an appropriate multiple of the logarithm of the characteristic function $\phi_{\Om}$, called there the \textbf{universal barrier}, is a $c(n+1)$-normal barrier for $\Om$ for some absolute constant $c$ not depending on $n$. This shows that an arbitrary proper open convex cone in $(n+1)$-dimensional space admits an $O(n+1)$-self-concordant barrier. Theorem \ref{sctheorem} shows that in fact the $O(n+1)$ can be replaced by exactly $n+1$. While this is interesting theoretically, as it sharpens the result of Nesterov-Nemirovskii on the existence of barriers, it is not clear whether it apports anything in terms of practical applications of interior point methods, as the calculation of the canonical potential requires the solution of a Monge-Amp\`ere equation. In any case, the theoretical gain is interesting from the geometrical point of view. Also the proof of Theorem \ref{sctheorem}, in which the self-concordance is deduced from the nonpositivity of the Ricci curvature of the canonical Hessian metric, is perhaps more understandable, at least to a geometer, than the original proof of Theorem $2.5.1$ of \cite{Nesterov-Nemirovskii}. 

For a homogeneous cone, the universal barrier is a constant multiple of the canonical potential. For nonhomogeneous cones, it appears reasonable to expect the canonical potential to dominate the universal barrier in a precise sense. Were it possible to show that the universal barrier satisfies an inequality of the form $\H(G) \geq ae^{bG}$ then an inequality relating the two would follow from Corollary \ref{gfcorollary}. While it seems plausible that such an inequality is true, I do not know how to show it. In this regard, O. G\"uler's paper \cite{Guler-selfconcordance} seems relevant; it gives an alternative proof of the self-concordance of the universal barrier and obtains inequalities of the form \eqref{sc2} for the derivatives of the universal barrier of all orders. Actually it seems plausible that the canonical potential can be characterized as the maximal $1$-self-concordant barrier function on $\Om$.

The content of Theorem \ref{sctheorem} is the verification of the inequality \eqref{sc1}. For a logarithmically homogeneous function $F$ to demonstrate \eqref{sc1} for some $\al$ it is, by \eqref{hessiancurvature}, enough to show that the scalar curvature of the Hessian metric with potential $F$ is bounded from above. However, arguing in this way only shows that $(n+1)F$ is a normal barrier with parameter at most $(n+1)^{2}$. To demonstrate Theorem \ref{sctheorem} the following stronger result, interesting in its own right, is needed.
\begin{theorem}\label{riccicurvaturetheorem}
The Ricci curvature $R_{ij}$ of the canonical metric $g_{ij} = \nabla_{i}dF_{j}$ of a proper open convex cone $\Om\subset \rea^{n+1}$ satisfies
\begin{align}\label{riccibound}
0 \geq R_{ij} \geq -\tfrac{n-1}{n+1}(g_{ij} - \tfrac{1}{n+1}F_{i}F_{j}) \geq \tfrac{1-n}{n+1}g_{ij}.
\end{align}
In particular, $R_{ij}$ is bounded from below and nonpositive. Consequently the scalar curvature $R_{g}$ of $g$ satisfies $0 \geq R_{g} \geq -n(n-1)/(n+1)$. If, moreover, $R_{g}$ is constant and equal to either $0$ or $-n(n-1)/(n+1)$ then $\Om$ is homogeneous.
\end{theorem}
The case $R_{g} = -n(n-1)/(n+1)$ occurs for the canonical metric of the Lorentz cone, which is a Riemannian cone over the hyperbolic metric, while the case $R_{g} = 0$ occurs for the canonical metric of the positive orthant, which is a flat metric. Theorem \ref{sctheorem} follows from Theorem \ref{riccicurvaturetheorem} and the explicit expression for the Ricci curvature of the canonical metric given in equation \eqref{hessiancurvature}. These are proved in section \ref{scalarcurvaturesection}. 
A straightforward corollary of Theorem \ref{riccicurvaturetheorem}, essentially equivalent to it is the following.
\begin{theorem}\label{scalarcurvature2theorem}
The Ricci curvature of the equiaffine metric $h$ of a complete hyperbolic (locally uniformly convex) affine sphere is nonpositive. Moreover, if the scalar curvature $R_{h}$ of $h$ is identically $0$ then the affine sphere is homogeneous.
\end{theorem}
The nonpositivity claim of theorem \ref{scalarcurvature2theorem} was proved by Calabi in \cite{Calabi-completeaffine}. The nonpositivity claim of Theorem \ref{riccicurvaturetheorem} follows from this result and a result of J. Loftin, given here as \eqref{loftinlemma} of Theorem \ref{affinespheretheorem}, showing that the canonical metric on $\Om$ is a Riemannian product of the equaiffine metric on a level set of the canonical potential and the flat metric on a ray. Here, instead, Theorem \ref{riccicurvaturetheorem} is proved directly. The estimates are adaptations of Calabi's estimates applied to a modification of $F_{ijk}$, essentially its trace-free horizontal part (see \eqref{aijkdefined2}). In some sense this amounts to replacing a projective picture by the affine picture in one dimension higher. The arguments here might be described as passing in an invariant way from the inhomogeneous coordinates to homogeneous coordinates. Similarly, the characterization of the homogeneous case in Theorem \ref{scalarcurvature2theorem} could be deduced from Calabi's estimates and Theorem $2$ of \cite{Dillen-Vrancken-Yaprak} characterizing the homogeneous affine spheres as those nondegenerate affine hypersurfaces having parallel Pick form, but a direct proof has been given instead. It seems likely that similar arguments could yield analogous control of the higher order derivatives of $F$. However, the computations become complicated to organize.

Seen as a condition on the affine geometry of the level sets of $F$, the self-concordance of $F$ is a suprisingly rich condition, having the flavor of nonpositivity conditions on curvatures, and deserving of further exploration from the purely geometric point of view. In this vein, observe that it makes sense to define a K\"ahler affine metric to be \textbf{$\al$-self-concordant} if there holds the inequality \eqref{sc1}, in which $F_{ijk}$ is well-defined globally as $\nabla_{i}g_{jk}$. 

\subsection{}
Comparison of the canonical potential and the logarithm of the characteristic function of a proper convex cone suggests the following questions, whose affirmative resolutions would have useful implications.
\begin{enumerate}
\item\label{qt1} Is the exponential of the canonical potential of a proper open convex cone a completely monotone function?
\item\label{qt2} Does the logarithm $\log \phi_{\Om}$ of the characteristic function $\phi_{\Om}$ of a proper open convex cone satisfy an inequality of the form $\H(G) \geq ae^{bG}$?
\item\label{qt3} In the Schwarz lemma (Theorem \ref{schwarztheorem}) can the hypothesis $\H(G) \geq e^{2G}$ be replaced by some condition such as that $G$ be a $1$-self-concordant barrier for $\Om$?
\end{enumerate}
Question \eqref{qt1} asks whether $(-1)^{k}v(1)^{i_{1}}\dots v(k)^{i_{k}}F_{i_{1}\dots i_{k}}(x) \geq 0$ for all $k \in \nat$, all $x \in \Om$, and all $v(i)^{i} \in \Om$. It is true for $k = 0, 1, 2$, and the general case should be tractable if the $k = 3$ case is. The affirmative resolution of \eqref{qt1} would mean that the canonical potential could be represented as the Laplace transform of a measure on the dual cone. By Corollary \ref{gfcorollary}, the affirmative resolution of \eqref{qt2} would imply an inequality between the canonical potential and the logarithm of the characteristic function of a proper open convex cone. Note that an affirmative answer to \eqref{qt3} would provide the same conclusion even in the absence of an affirmative answer for \eqref{qt2}.

On a proper open convex polyhedral region $P$ of the form $P = \{x \in \rea^{n+1}: \ell_{\al}(x) > 0\}$, where $1 \leq \al \leq d$, $d \geq n+1$, and $\ell_{\al}(x) = a_{\al i}x^{i} - b_{\al}$, the logarithmic barrier function $G = -\sum_{\al}\log \ell_{\al}(x) - c$, where $c$ is some constant to be determined, is convex with a unique minimum. The Hessian of $G$ is $G_{ij} = \sum_{\al = 1}^{d}a_{\al i}a_{\al j} \ell_{\al}(x)^{-2}$, which has the form $A^{t}SA$ where $A$ is the $d \times (n+1)$ matrix with elements $a_{\al i}$ and $S$ is the $d \times d$ diagonal matrix with entries $\ell_{\al}(x)^{-2}$. Since $G_{ij}v^{i}v^{j}$ is a sum of squares that vanishes if and only if $a_{\al p }v^{p} = 0$ for all $\al$, $G_{ij}$ is positive definite on $P$ if and only if $A$ has full rank; this is necessarily the case because, by assumption, $P$ has a vertex. Let $\la_{1}$ be the smallest eigenvalue of $A^{t}A$, and note that $\la_{1} > 0$, since $A$ has full rank. If $P$ is assumed bounded, then each product $\prod_{\al \in I}\ell_{\al}^{-2}$, where $I \subset \{1, \dots, d\}$, has a positive minimum on $P$. Let $Q = \min_{|I| = d - n-1}\min\{\prod_{\al \notin I}\ell_{\al}(x)^{-2}: x\in P\} > 0$. By Ostrowski's theorem for rectangular matrices, Theorem $3.2$ of \cite{Higham-Cheng}, for each $x \in P$ there is a cardinality $n+1$ index set $I$, determined by the requirement that if $\al \in I$ and $\be \notin I$ then $\ell_{\al}(x)^{-2} \leq \ell_{\be}(x)^{-2}$, such that $\det A^{t}SA \geq \la_{1}^{n+1}\prod_{\al \in I}\ell_{\al}^{-2}$ at $x$. Since $\prod_{\al \in I}\ell_{\al}^{-2} = e^{-2c}e^{2G}\prod_{\al \notin I}\ell_{\al}^{-2}$ there results $\det A^{t}SA \geq \la_{1}^{n+1}e^{-2c}e^{2G}\prod_{\al \notin I}\ell_{\al}^{-2} \geq \la_{1}^{n+1}e^{2c}Qe^{2G}$ for all $x \in P$. It follows that when $P$ is bounded the equality $\H(G) \geq e^{2G}$ can be arranged by choosing $c$ so that $e^{-2c} = \la_{1}^{n+1}Q$. Note that $c$ depends only on the coefficients $a_{\al i}$, and is computable in practice. For example, for the planar triangle with vertices $(0,0)$, $(0, 1)$, and $(1, 0)$ the function $G(x, y) = - \log(xy(1 - x- y)) - \log\sqrt{3}$ satisfies $\H(G) \geq e^{2G}$, and for the unit square with vertices $(0,0)$, $(0,1)$, $(1, 1)$, and $(1, 0)$, the function $G(x, y) = -\log(xy(1-x)(1-y)) - \log(2)$ satisfies $\H(G) \geq e^{2G}$.  From Corollary \ref{gfcorollary} it follows that the canonical potential $F$ of $P$ satisfies $F \geq G$ for the appropriate constant $c$. The boundedness of $P$ is probably not necessary for the preceeding conclusion; here it has been needed only for the particular method of proof. For example in the unbounded region $P = \{(x, y) \in \rea^{2}: x > 0, y > 0, x + y > 1\}$ the function $G(x, y) = -\log(xy(x+y-1)) - \log\sqrt{2}$ satisfies $\H(G) \geq e^{2G}$. Although in general the function $G$ is not a multiple of the characteristic function of $P$ unless $P$ is an affine image of the standard orthant (in which case $d = n+1$), the preceeding discussion lends some plausibility to the idea that question \eqref{qt2} has an affirmative resolution.

\subsection{}\label{mamsection}
A K\"ahler affine metric with vanishing K\"ahler affine Ricci tensor is called a \textbf{Monge-Amp\`ere metric}. These can be seen as real analogues of Calabi-Yau manifolds. There is interest in finding explicit examples of such metrics because of their role in various formulations of homological mirror symmetry, where they arise by considering degenerating families of Calabi-Yau manifolds in some limit; see e.g. \cite{Kontsevich-Soibelman}, \cite{Kontsevich-Soibelman-nonarchimedean}, and \cite{Loftin-Yau-Zaslow} for background and references. In section \ref{mongeamperesection} it is shown how straightforward adaptation of an argument of Calabi yields
\begin{theorem}\label{mongeamperetheorem}
For each hyperbolic affine sphere $\Sigma$ asymptotic to the boundary of the proper open convex cone $\Omega \subset \rea^{n+1}$ there is a Monge-Amp\`ere Riemannian metric defined on the open subset of the interior of $\Omega$ bounded by $\Sigma$ and the boundary of $\Omega$, that is, the region formed by the union of the open line segments contained in $\Omega$ and running from the origin to some point of $\Sigma$. 
\end{theorem}
Theorem \ref{mongeamperetheorem} is equivalent to Proposition $1$ of the unpublished erratum \cite{Loftin-Yau-Zaslow-erratum}; see the remarks in section \ref{mongeamperesection}. In section \ref{mongeamperesection} it is additionally shown that a similar construction yields on $\Omega$ a globally hyperbolic Lorentzian signature Monge-Amp\`ere metric admitting the hyperbolic affine sphere $\Sigma$ as a Cauchy hypersurface. Precisely,
\begin{theorem}\label{lmatheorem}
Let $F$ be the canonical potential of a nonempty proper open convex cone $\Om \subset \rea^{n+1}$. The function $u = -((n+1)/2)e^{-2F/(n+1)}$ solves $\H(u) = -1$, and $k_{ij} = \nabla_{i}du_{j}$ is a Lorentzian signature Monge-Amp\`ere metric on $\Omega$. Moreover, $k_{ij}$ is globally hyperbolic, a level set $\lc_{r}(F, \Omega)$ being a complete Cauchy hypersurface.
\end{theorem}

\section{Affine geometry of level sets}\label{affinesection}
Theorem \ref{ahtheorem} is proved in this section. First, the explicit expression for the equiaffine normal of a level set is recalled. The formulas recorded below can be deduced from similar ones in J. Hao and H. Shima's \cite{Hao-Shima} (see also \cite{Shima}). Another derivation is given in \cite{Fox-autoiso}.

The standard flat affine connection and parallel volume form on $\rea^{n+1}$ are written $\nabla$ and $\Psi$. The vector field on $\rea^{n+1}$ generating the radial flow by dilations by a factor of $e^{t}$ is written $\eul$ and satisfies $\nabla_{i}\eul^{j} = \delta_{i}\,^{j}$. Let $\Omega$ be a connected component with nonempty interior of the region on which $F$, $dF$, and $\H(F)$ are nonvanishing and $g_{ij}$ is positive definite, and for $r \in \rea$ let $\lc_{r}(F, \Omega) = \{x \in \Omega: F(x) = r\}$. By assumption $\H(F)$ does not change sign on $\lc_{r}(F, \Omega)$ and the $g$-gradient $F^{i}= g^{ij}F_{j}$ is nonzero on $\Omega$, so is a convenient transversal to $\lc_{r}(F, \Omega)$. The rank $n$ symmetric tensor
\begin{align}\label{sffdefined}
\begin{split}
&\sff_{ij}  = \H(F)^{-1/(n+2)}|dF|_{2}^{-2/(n+2)}\left(g_{ij} - |dF|_{g}^{-2}F_{i}F_{j}\right), 
\end{split}
\end{align}
satisfies $F^{i}\sff_{ij} = 0$, and its restriction to $\lc_{r}(F, \Omega)$ is the equiaffine metric. Define a one-form $\muf_{i}$ by $(n+2)\muf_{i} = H_{i} + d_{i}\log|dF|^{2}_{g}$. The vector field $\nk^{i}$ defined by
\begin{align}\label{nkdefined}
\begin{split}
\nk^{i} &= (1 - F^{p}\muf_{p})F^{i} + |dF|_{g}^{2}\muf^{i} = F^{i} + |dF|_{g}^{2}(g^{ij} - |dF|_{g}^{-2}F^{i}F^{j})\muf_{j}.
\end{split}
\end{align}
spans the affine normal distribution, and the equiaffine normal field of $\lc_{r}(F, \Omega)$ is 
\begin{align}\label{affnorm}
\nm^{i} = -\H(F)^{1/(n+2)}|dF|_{g}^{-2(n+1)/(n+2)}\nk^{i}.
\end{align}

\noindent
For $F \in \lmg_{\al}(\Om)$ such that $\lc_{r}(F, \Omega)$ is nondegenerate there hold $dF(\eul) =  \al$, $\eul^{p}F_{pi} =  - F_{i}$, and $\eul^{p}H_{p} = -2(n+1)$, and so
\begin{align}\label{hompol}
&F^{i} = -\eul^{i},& & |dF|^{2}_{g} = -\al,& &(n+2)\muf_{i} = H_{i},& &(n+2)(1 - F^{p}\muf_{p}) = -n.
\end{align}
In particular, if $g_{ij}$ is to be positive definite, $\al$ must be negative. Substituting \eqref{hompol} into \eqref{affnorm} yields that along $\lc_{r}(F, \Omega)$ the equiaffine normal $\nm^{i}$ has the form
\begin{align}\label{nm2}
\begin{split}
\nm^{i} & =  -\tfrac{1}{n+2}\left|\al\right|^{-(n+1)/(n+2)}\left|\H(F)\right|^{1/(n+2)}\left(n\eul^{i} - \al H^{i}\right).
\end{split}
\end{align}

\begin{proof}[Proof of Theorem \ref{ahtheorem}]
Suppose that $F \in \lmg_{\al}(\Omega)$ and there is an open interval $I \subset \rea$ such that for all $r \in I$ each nonempty connected component of $\lc_{r}(F, \Omega)$ is a hyperbolic affine sphere with center at the origin. Along $\lc_{r}(F, \Om)$ there holds $\nm^{i} = -c\eul^{i}$ where the constant $c(r)$ is the affine mean curvature of $\lc_{r}(F, \Om)$. Contracting \eqref{nm2} with $F_{i}$ shows that
\begin{align}\label{hfrel}
c = -|\al|^{-(n+1)/(n+2)}\H(F)^{1/(n+2)}.
\end{align}
Since $\al < 0$, there results $\H(F) = -\al^{n+1}c^{n+2}$, which is constant on $\lc_{r}(F, \Omega)$. This holds for each $r \in I$, and so there is a function $\phi$ defined on $I$ such that $\H(F) = \phi(F)$ for $x \in \Omega_{I}$.

Now suppose $F \in \lmg_{\al}(\Omega)$ solves $\H(F) = \phi(F)$ for some nonvanishing function $\phi:I \to \rea$. Since for $x \in \lc_{r}(F, \Omega)$, $\eul^{i}F_{i}(x) = \al \neq 0$, $dF$ does not vanish on $\lc_{r}(F, \Omega)$, and so the level set $\lc_{r}(F, \Omega)$ is smoothly immersed and $\eul$ is transverse to $\lc_{r}(F, \Omega)$. Let $k$ be the representative of the second fundamental form corresponding to the transversal $\eul$. For $X$ and $Y$ tangent to $\lc_{r}(F, \Omega)$ there hold 
\begin{align}\label{hdl}
\begin{split}
&g(X, Y) = (\nabla_{X}dF)(Y) = -dF(\eul)k(X, Y) = -\al k(X, Y), \\
&g(X, \eul) = -dF(X) = 0, \qquad
g(\eul, \eul) = (\nabla_{\eul}dF)(\eul) = -\al,
\end{split}
\end{align}
along $\lc_{r}(F, \Omega)$. Since $\al< 0$, it follows from \eqref{hdl} and the assumption that $g_{ij}$ is positive definite, that $k$ is positive definite. Hence the equiaffine normal $\nm^{i}$ is defined on $\Omega_{I}$. Since $\H(F)$ is constant on each connected component of $\lc_{r}(F, \Omega_{I})$, it must be that $d\H(F) \wedge dF = 0$ on $\Omega_{I}$, so there is $q \in \cinf(\Omega_{I})$ such that $H_{i} = q F_{i}$. Pairing with $\eul$ yields $-2(n+1) = \al q$, so that $q$ is constant on $\lc_{r}(F, \Omega_{I})$. In \eqref{nm2} this yields that $\nm^{i}$ is a constant multiple of $\eul^{i}$ along $\lc_{r}(F, \Omega_{I})$, and so each connected component of $\lc_{r}(F, \Omega_{I})$ is an affine sphere, necessarily hyperbolic, by \eqref{hfrel}. 

Suppose now that there hold \eqref{aht1}-\eqref{aht2}. Since $F \in \lmg_{\al}(\Omega)$, $\H(F)$ has positive homogeneity $-2(n+1)$. It follows that $\phi(r + \al t) = e^{-2(n+1)t}\phi(r)$ for $r \in I$ and sufficiently small $t$. In particular, this shows $\phi$ is continuous on $I$. Since
\begin{align*}
\lim_{t \to 0} \tfrac{\phi(r + \al t) - \phi(r)}{\al t} = \lim_{t \to 0}\tfrac{(e^{-2(n+1)t} - 1)}{\al t}\phi(r)= \tfrac{-2(n+1)}{\al}\phi(r),
\end{align*}
$\phi$ is differentiable on $I$ and solves $\al \phi^{\prime}(r) =  -2(n+1)\phi(r)$. The general solution has the form $Be^{-2(n+1)r/\al}$ for a nonzero constant $B$. Substituting this into \eqref{hfrel} shows that the affine mean curvature of $\lc_{r}(F, \Omega_{I})$ has the form \eqref{ahmc}.
\end{proof}

\section{Hessian metrics as metric measure spaces}\label{hessianmetricsection}

\subsection{}
The Laplacian $\lap_{k}$ of the metric $k$ is the negative of the divergence of the exterior differential $d$, where the divergence is the adjoint of $d$ with respect to $\vol_{g}$. For the mm-structure $(k, \phi)$, replacing the divergence with the adjoint of $d$ with respect to $\phi\vol_{k}$ yields the operator $\blap_{k} = \lap_{k} + k^{ij}d\log\phi_{i}D_{j}$, which is the specialization of the modified \textbf{mm-Laplacian} $\blap_{k}$ of the local mm-structure $(k,\al)$ defined by 
\begin{align}\label{blapdefined}
\blap_{k} = \lap_{k} + k^{ij}\al_{i}D_{j}. 
\end{align}
Theorem \ref{bqcomparisontheorem} is  a distance comparison theorem for lower bounds on Bakry-Emery Ricci tensor. It is the specialization to the present setting of Theorem $4.2$ of D. Bakry and Z. Qian's \cite{Bakry-Qian-volume} (the explanation of this theorem in \cite{Wei-Wylie} may be more accessible to geometers). 
\begin{theorem}[D. Bakry and Z. Qian, \cite{Bakry-Qian-volume}]\label{bqcomparisontheorem}
Let $(k, \al)$ be a complete smooth local metric measure structure on the $n$-manifold $M$. Suppose the associated $(N+n)$-dimensional Bakry Emery Ricci tensor $\nric_{ij}$ satisfies a lower bound of the form $\nric_{ij} \geq -\ka^{2}(N+n-1)g_{ij}$ for some real constant $\ka > 0$. Let $p_{0} \in \Omega$ and let $r(p)$ be the $k$-distance from $p$ to $p_{0}$. Let $\blap_{k}$ be the mm-Laplacian defined by \eqref{blapdefined}. For $p$ in the complement $M\setminus\cut(p_{0})$ of the cut locus of $p_{0}$ there holds
\begin{align}\label{blapcomp}
r\blap_{k}r \leq (N+n-1)\ka r\coth(\ka r) \leq (N+n-1)(1 + \ka r).
\end{align}
\end{theorem}

\noindent
Theorem \ref{cyestimatetheorem} is proved by the adaptation to the metric measure context of the argument used by Cheng and Yau to prove Theorem $2$ of their \cite{Cheng-Yau-maximalspacelike}. For the reader's convenience a full proof is given. 

\begin{theorem}\label{cyestimatetheorem}
Let $(M, k, \al)$ be a complete $n$-dimensional smooth local metric measure space. Suppose the associated $(N+n)$-dimensional Bakry Emery Ricci tensor $\nric_{ij}$ satisfies a lower bound of the form $\nric_{ij} \geq -\ka^{2}(N+n-1)g_{ij}$ for some real constant $\ka > 0$. Let $\blap_{k}$ be the mm-Laplacian defined by \eqref{blapdefined}. Suppose $u \in C^{2}(M)$ is nonnegative and not identically $0$ and that wherever $u$ is not $0$ it satisfies $\blap_{k}u \geq Bu^{1 + \si} - Au$ for some constants $B > 0$, $\si > 0$, and $A \in \rea$. Then for any $x \in M$ at which $u(x) \neq 0$, and any $a >0$, on the open ball $B(x, a)$ of radius $a$ centered at $x$ there holds 
\begin{align}\label{cyestimate}
u \leq (a^{2} - r^{2})^{-2/\si}\left|(\tfrac{A}{B})a^{4} + (\tfrac{4 \ka(N+n-1)}{B\si })a^{3} + (\tfrac{4((N+n+2)\si + 4)}{B\si^{2}})a^{2}\right|^{1/\si}
\end{align}
in which $r = d(x, \dum)$ is the $k$-distance from $x$. In particular, letting $a \to \infty$, there holds $\sup_{M}u \leq |A/B|^{1/\si}$.
\end{theorem}

\begin{proof}
In this proof it is convenient to drop subscripts indicating dependence on $k$, writing $\lap$, $\blap$, $|\dum|^{2}$, etc. instead of $\lap_{k}$, $\blap_{k}$, $|\dum|^{2}_{k}$, etc. Suppose $u \in \cinf(M)$ is nonnegative and not identically zero, and choose $x$ so that $u(x) \neq 0$. Let $a > 0$ and $\al > 0$ and define $f = (a^{2} - r^{2})^{\al}u$ which is by assumption not identically zero on $B(x, a)$. Since $r$ is smooth on the complement of the cut locus $\cut(x)$ of $x$, $f$ is smooth on the complement of $\cut(x)$ in the ball $B(x, a)$, and there hold
\begin{align}
\label{cy1}Df & = \left(\tfrac{du}{u} - \tfrac{2\al r dr}{a^{2} - r^{2}} \right)f,\\
\label{cy2}
\begin{split}\tfrac{\lap f}{f} & = \left|\tfrac{du}{u} - \tfrac{2\al r dr}{a^{2} - r^{2}}\right|^{2} + \left(\tfrac{\lap u}{u} - \tfrac{|du|^{2}}{u^{2}} - \tfrac{2\al(r\lap r + 1)}{a^{2} - r^{2}} - \tfrac{4\al r^{2}}{(a^{2} - r^{2})^{2}} \right)\\
& = \left|\tfrac{du}{u} - \tfrac{2\al r dr}{a^{2} - r^{2}}\right|^{2} + \left(\tfrac{\blap u}{u} - \tfrac{|du|^{2}}{u^{2}} - \tfrac{2\al(r\blap r + 1)}{a^{2} - r^{2}} - \tfrac{4\al r^{2}}{(a^{2} - r^{2})^{2}} \right) + k\left(d\phi, \left(\tfrac{2\al r dr}{a^{2} - r^{2}} - \tfrac{du}{u}\right) \right).
\end{split}
\end{align}
Since by construction $f$ is not identically $0$ on $B(x, a)$, and vanishes on the boundary $\pr B(x, a)$, its restriction to the closure of $B(x, a)$ (which is compact, because $k$ is complete), attains its maximum at some $x_{0} \in B(x, a)$. First suppose $x_{0} \notin \cut(x)$. The proof in the case $x_{0} \in \cut(x)$ is similar, and will be indicated at the end. Since $f(x_{0}) \neq 0$, also $u(x_{0}) \neq 0$.  Since at $x_{0}$ there vanishes $df$, there holds $\blap f = \lap f$ at $x_{0}$. It follows from \eqref{cy1} that at $x_{0}$ there holds $\tfrac{du}{u} = \tfrac{2\al r dr}{a^{2} - r^{2}}$ and, as at $x_{0}$ there holds $0 \geq \lap f = \blap f$, in \eqref{cy2} this implies that at $x_{0}$ there holds
\begin{align}\label{cy3}
\tfrac{\blap u}{u}  \leq \tfrac{2\al(r\blap r + 1)}{a^{2} - r^{2}} + \tfrac{4\al (\al + 1)r^{2}}{(a^{2} - r^{2})^{2}}.
\end{align}
Since $B$ is positive, substituting \eqref{blapcomp} into \eqref{cy3} and rearranging the result shows that at $x_{0}$ there holds
\begin{align}\label{cy5}
u^{\si} \leq \tfrac{A}{B} + \tfrac{2\al(N+n + \ka(N+n-1)r)}{B(a^{2} - r^{2})} + \tfrac{4\al (\al + 1)r^{2}}{B(a^{2} - r^{2})^{2}}
\end{align}
Let $\al = 2/\si$ and multiply \eqref{cy5} by $(a^{2} - r^{2})^{2}$ to obtain that at $x_{0}$ there holds
\begin{align}\label{cy6}
\begin{split}
f^{\si} &\leq \tfrac{A}{B}(a^{2} - r^{2})^{2} + \tfrac{4(N+n + \ka(N+n-1)r)(a^{2} - r^{2})}{B\si} + \tfrac{8(2 + \si)r^{2}}{B\si^{2}} \\
&\leq (\tfrac{A}{B})a^{4} + (\tfrac{4 \ka(N+n-1)}{B\si})a^{3} + (\tfrac{4((N+ n + 2)\si  + 4)}{B\si^{2}} )a^{2},
\end{split}
\end{align}
which implies \eqref{cyestimate}. 
For the proof in the case $x_{0} \in \cut(x)$ the argument is modified using a device due to Calabi in \cite{Calabi-hopfmaximum}. For the reader's convenience this is recalled here following the end of the proof of the gradient estimate in \cite{Schoen-Yau}. There is a minimizing geodesic joining $x$ to $x_{0}$ the image $\si$ of which necessarily lies in $B(x, a)$. Let $\bar{x}$ be a point on $\si$ lying strictly between $x$ and $x_{0}$ at some distance $\ep> 0$ from $x$. Since $\si$ is minimizing, no point of $\si$ can be conjugate to $\bar{x}$. Were $x$ or $x_{0}$ conjugate to $\bar{x}$ then it would be in $\cut(\bar{x})$, and so $\bar{x}$ would be in its cut locus, which it is not because $x_{0} \in \cut(x)$ and $x \in \cut(x_{0})$. Thus no point of $\si$ is a conjugate point of $\bar{x}$ and hence there is some $\delta > 0$ for which there is an open $\delta$ neighborhood $N \subset B(x, a)$ of $\si$ containing no conjugate point of $\bar{x}$. Let $\bar{r} = d(\bar{x}, \dum)$. By the triangle inequality, $\bar{r} + \ep \geq r$. On the other hand $\bar{r}(x_{0}) + \ep = r(x_{0})$. Define $\bar{f} = (a^{2} - (\bar{r} + \ep)^{2})^{\al}u$. Then $\bar{f} \leq f$ on $N$ and $\bar{f}(x_{0}) = f(x_{0})$, so $\bar{f}$ attains its maximum value on $N$ at $x_{0}$. As $\bar{r}$ is smooth near $x_{0}$ the preceeding argument goes through with $\bar{f}$ in place of $f$, and letting $\ep \to 0$ at the end yields \eqref{cyestimate}.
\end{proof}

\subsection{}
Throughout this section $(\nabla, g)$ is a Riemannian signature K\"ahler affine metric on a smooth $(n+1)$-dimensional manifold $M$. 
The Levi-Civita connection $D$ of $g_{ij}$ is $D = \nabla + \tfrac{1}{2}F_{ij}\,^{k}$. The curvature tensor $R_{ijk}\,^{l}$ of the Levi-Civita connection $D$ of $g$ is defined by $2D_{[i}D_{j]}X^{k} = R_{ijp}\,^{k}X^{p}$. From $\nabla_{i}F_{jk}\,^{l} = F_{ijk}\,^{l} - F_{pi}\,^{l}F_{jk}\,^{p}$ it follows that the Riemann curvature $R_{ijkl} = R_{ijk}\,^{p}g_{pk}$, the Ricci tensor $R_{ij} = R_{pij}\,^{p}$, and the scalar curvature $R_{g} = g^{ij}R_{ij}$ of $g$ have the forms:
\begin{align}\label{hessiancurvature}
\begin{split}
&R_{ijkl} = g_{lp}\nabla_{[i}F_{j]k}\,^{p} + \tfrac{1}{2}F_{pl[i}F_{j]k}\,^{p} = -\tfrac{1}{2}F_{pl[i}F_{j]k}\,^{p},\\
R_{ij}  &= \tfrac{1}{4}\left( F_{ip}\,^{q}F_{jq}\,^{p} - F_{ij}\,^{p}H_{p}\right),\qquad
R_{g}  = \tfrac{1}{4}\left(|\nabla \nabla dF|_{g}^{2} - |H|_{g}^{2}\right).
\end{split}
\end{align}
As explained in the introduction, the K\"ahler affine metric $(\nabla, g)$ is identified with the local mm-space $(g, H)$ determined by $g$ in conjunction with the Koszul form $H_{i}$. In particular a Hessian metric $g_{ij} = \nabla_{i}dF_{j}$, with global potential $F \in \cinf(M)$, is identified with the mm-space $(g_{ij}, \H(F)\Psi) = (g_{ij}, \H(F)^{1/2}d\vol_{g})$. The associated modified Laplacian is $\blap_{g} = \lap_{g} + \tfrac{1}{2}H^{i}D_{i}$.

If $A \in C^{2}(\Omega)$ then $A_{ij} = D_{i}A_{j} + \tfrac{1}{2}F_{ij}\,^{p}A_{p}$ and so $\blap_{g}A = A_{p}\,^{p}$. In particular, $D_{i}F_{j} = g_{ij} - \tfrac{1}{2}F_{ij}\,^{p}F_{p}$ and $\blap_{g}F = (n+1)$. By \eqref{bedefined} and \eqref{hessiancurvature} the Bakry-Emery Ricci tensors $\bric_{ij}$ and $\nric_{ij}$ are given by
\begin{align}
\label{nricdefined}
\begin{split}
\bric_{ij} &=  R_{ij} - \tfrac{1}{2}D_{i}H_{j} = R_{ij} + \tfrac{1}{2}\ak_{ij} + \tfrac{1}{4}F_{ij}\,^{p}H_{p}= \tfrac{1}{4}F_{ip}\,^{q}F_{jq}\,^{p} + \tfrac{1}{2}\ak_{ij} \geq  \tfrac{1}{2}\ak_{ij},\\
\nric_{ij} & = R_{ij} - \tfrac{1}{2}D_{i}H_{j} - \tfrac{1}{4N}H_{i}H_{j}= \tfrac{1}{4}F_{ip}\,^{q}F_{jq}\,^{p} - \tfrac{1}{4N}H_{i}H_{j} + \tfrac{1}{2}\ak_{ij}.
\end{split}
\end{align}
In particular, a lower bound on the K\"ahler affine Ricci tensor $\ak_{ij}$ implies a lower bound on the $\infty$-Ricci tensor. Note also that the K\"ahler affine scalar curvature $\aks$ is given by
\begin{align}\label{aks}
\aks = g^{ij}\ak_{ij} = -\blap_{g}\log\H(F) = -D^{p}H_{p} - \tfrac{1}{2}|H|^{2}_{g}.
\end{align}
The nonnegativity of the norm of $(n+1)X^{p}F_{pij} - X^{p}H_{p}g_{ij}$ for any $X^{i}$ implies
\begin{align}\label{nr1}
(n+1)F_{ip}\,^{q}F_{jq}\,^{p} \geq H_{i}H_{j}.
\end{align}
Together \eqref{nr1} and \eqref{nricdefined} show that the $2(n+1)$-dimensional Bakry-Emery Ricci tensor is bounded from below by the K\"ahler affine Ricci tensor:
\begin{align}\label{cd1}
R(n+1)_{ij} \geq \tfrac{1}{2}\ak_{ij}.
\end{align}
When the potential $F$ of a Hessian metric satisfies $\H(F) = e^{2F}$, the inequality \eqref{cd1} yields the condition $R(n+1)_{ij} \geq -g_{ij}$, which is called a \textit{curvature-dimension inequality} $CD(-1, 2(n+1))$ (see e.g. \cite{Bakry-Qian-volume}). The $2(n+1)$ is the real dimension of the tube domain $\tube_{\Omega}$ over $\Omega$, reflecting that the inequality \eqref{cd1} is inherited from the lower bound on the Ricci curvature of the K\"ahler metric determined on $\tube_{\Omega}$ by $F$. Note that the lower bound \eqref{cd1} is stronger than the lower bound of $\bric_{ij}$ given in \eqref{nricdefined}.

\begin{theorem}\label{bqhessiantheorem}
Let $\nabla$ be a flat affine connection on the $(n+1)$-dimensional manifold $M$ and let $g_{ij}$ be a complete Riemannian metric forming with $\nabla$ a K\"ahler affine structure with K\"ahler affine Ricci curvature bounded from below by $-2Ag_{ij}$ for some positive constant $A$. Let $p_{0} \in \Omega$ and let $r(p)$ be the $g$-distance from $p$ to $p_{0}$. Let $\blap_{g} = \lap_{g} + \tfrac{1}{2}H^{i}D_{i}$ be the mm-Laplacian associated to the smooth local mm-space $(g_{ij}, \tfrac{1}{2}H_{i})$. For $p \in M\setminus\cut(p_{0})$ there holds
\begin{align}
r\blap_{g}r \leq \sqrt{A(2n+1)}r\coth(r\sqrt{\tfrac{A}{2n+1}}) \leq 2n+1 + r\sqrt{A(2n+1)}.
\end{align}
\end{theorem}
\begin{proof}
Equation \eqref{cd1} shows that the $2(n+1)$-dimensional Bakry Emery Ricci tensor $R(n+1)_{ij}$ associated to $\blap_{g}$ satisfies the lower bound $R(n+1)_{ij} \geq -Ag_{ij}$. The claim follows by specializing Theorem \ref{bqcomparisontheorem} with $N = n+1$ (and $n+1$ in place of $n$).
\end{proof}

For a symmetric tensor $\si_{ij} = \si_{(ij)}$ the nonnegativity of its trace-free part implies $(n+1)\si_{ij}\si^{ij} \geq (\si_{p}\,^{p})^{2}$. Applying this inequality to $A_{ij} = X^{p}F_{pij} - \be H_{(i}X_{j)}$, where $\be \in \rea$ and $X^{i}$ is a vector field on $\Omega$, and noting $A_{p}\,^{p} = (1-\be)X^{p}H_{p}$ yields 
\begin{align}\label{xxric}
\begin{split} 
 X^{i}X^{j}\left( F_{ip}\,^{q}F_{jq}\,^{p} - 2\be F_{ij}\,^{p}H_{p}\right)&  = A^{ij}A_{ij} - \tfrac{1}{2}\be^{2}|X|_{g}^{2}|H|^{2}_{g} - \tfrac{1}{2}\be^{2}(X^{p}X_{p})^{2}\\
&\geq -\tfrac{((n-1)\be^{2} + 4\be - 2)}{2(n+1)}\left(X^{p}H_{p}\right)^{2}- \tfrac{1}{2}\be^{2}|X|_{g}^{2}|H|^{2}_{g}.
\end{split}
\end{align}
Taking $\be = 1/2$ in \eqref{xxric} and comparing with \eqref{hessiancurvature} shows that 
\begin{align}\label{xxric3}
\begin{split}
R_{ij} &  \geq \tfrac{1-n}{32(n+1)}H_{i}H_{j} - \tfrac{1}{32}|H|^{2}_{g}g_{ij}\geq -\tfrac{n}{16(n+1)}|H|_{g}^{2}g_{ij}.
\end{split}
\end{align}
From \eqref{xxric3} it follows that an upper bound on $|H|^{2}_{g}$ suffices to bound the Ricci curvature from below. However, a bound on $|H|^{2}_{g}$ is the sort of thing one wishes to conclude rather than to assume. In this regard, the aspect of Theorem \ref{bqhessiantheorem} that is important here is that it is true even if $|H|_{g}^{2}$ is not assumed bounded, in particular without an a priori lower bound on the ordinary Ricci curvature of $g$.

\subsection{}\label{schwarzlemmasection}
 Theorem \ref{schwarztheorem}, the Schwarz lemma for K\"ahler affine metrics, is proved now. 

\begin{proof}[Proof of Theorem \ref{schwarztheorem}]
The function $u = (\om/vol_{g})^{2}$ is smooth and positive. In an open neighborhood of any $p \in M$ there can be chosen a smooth $\nabla$-parallel volume form $\mu$. There is a positive smooth function $V$ such that $\om^{2} = V\mu^{2}$. Restricting to a smaller open neighborhood of $p$ if necessary, choose a potential $F$ for $g$ and write $\vol_{g}^{2} = \det \nabla dF = \H(F)\mu^{2}$. Hence $u$ coincides with $V/\H(F)$ where the latter is defined. The hypothesis that $\det \nabla\nabla \log \om^{2} \geq B\om^{2}$ is equivalent to $\H(\log V) \geq BV$. By hypothesis and \eqref{aks}, $- \blap_{g}\log\H(F) = \aks \geq - A(n+1)$. By the inequality of the arithmetic and geometric means and the hypotheses on $F$ and $G$, 
\begin{align}\label{blapratio}
\begin{split}
\blap_{g}u & = \blap_{g}(\log V - \log\H(F))  \geq \blap_{g}\log V - A(n+1) = g^{ij}(\log V)_{ij} - A(n+1)\\
& \geq (n+1)\left(\H(\log V)/\H(F))^{1/(n+1)} - A\right)  \\
&\geq (n+1)\left((BV/\H(F))^{1/(n+1)} - A\right) = (n+1)\left((Bu)^{1/(n+1)} - A\right).
\end{split}
\end{align}
Since $p\in M$ was arbitrary, the conclusion of \eqref{blapratio} is valid on all of $M$. 

Let $x_{0} \in M$. Let $r$ be the distance from $x_{0}$ in the metric $g$, which is smooth on the complement of the cut locus $\cut(x_{0})$ of $x_{0}$. Let $B(x_{0}, a)$ be the open geodesic ball of radius $a$ centered on $x_{0}$. Let $a > 0$ and $\be > 0$ and define $v = (a^{2} - r^{2})^{\be}u$. On the complement $B(x_{0}, a) \setminus \cut(x_{0})$ in $B(x_{0}, a)$ of $\cut(x_{0})$ there hold
\begin{align}
 dv & = v\left(d\log u - 2\be r(a^{2} - r^{2})^{-1}dr\right),\\
\label{blapv} 
\begin{split}
\blap_{g}v & = v\left(\left|d\log u - \tfrac{2\be rdr}{a^{2} - r^{2}}\right|^{2} + \blap_{g} \log u  - 2\be\left( \tfrac{a^{2} + r^{2}}{(a^{2} - r^{2})^{2}} + \tfrac{r\blap_{g}r}{a^{2} - r^{2}}\right)\right).
\end{split}
\end{align}
Since by assumption $g$ is complete, the closed ball $\bar{B}(x_{0}, a)$ is compact. Since $v$ is not identically zero on $B(x_{0}, a)$ and vanishes on the boundary $\pr B(x_{0}, a)$, the restriction of $v$ to the closure $\bar{B}(x_{0}, a)$ attains its maximum at some point $p \in B(x_{0}, a)$. Suppose that $p\notin \cut(x_{0})$. The proof in the case $p \in\cut(x_{0})$ is similar, and is described at the end. At $p$ there hold $dv(p) = 0$ and $\blap_{g}v = \lap_{g}v(p) \leq 0$. In particular, by \eqref{blapv}, at $p$ there holds
\begin{align}\label{blv1}
 \left(Bu\right)^{1/(n+1)} \leq A + \tfrac{2\be}{n+1}\left( \tfrac{a^{2} + r^{2}}{(a^{2} - r^{2})^{2}} + \tfrac{r\blap_{g}r}{a^{2} - r^{2}} \right).
\end{align}
Because by assumption the K\"ahler affine Ricci tensor is bounded from below by a multiple of $g_{ij}$, it follows from Theorem \ref{bqhessiantheorem} that there is a constant $c$ such that $r\lap_{g}r \leq 2n +1 + cr$. In \eqref{blv1} this shows that at $p$ there holds
\begin{align}\label{blv2}
 \left(Bu\right)^{1/(n+1)} \leq A + \tfrac{2\be}{n+1}\left( \tfrac{a^{2} + r^{2}}{(a^{2} - r^{2})^{2}} + \tfrac{2n+1 + cr}{a^{2} - r^{2}} \right).
\end{align}
Set $\be = 2(n+1)$ and multiply both sides by $(a^{2} - r^{2})^{2}$ to obtain
\begin{align}
\begin{split}
(\sup_{B(x_{0}, a)}Bv)^{1/(n+1)} & \leq A(a^{2} - r^{2})^{2} + 4(2n + 1 + cr))(a^{2} - r^{2}) + 4(a^{2} + r^{2}) \\
& \leq Aa^{4} + 4ca^{3} + 8(n+2)a^{2}.
\end{split}
\end{align}
Hence, when $p$ is not in $\cut(x_{0})$, there holds on $B(x_{0}, a)$ the inequality
\begin{align}\label{evf}
Bu \leq  (a^{2} - r^{2})^{-2(n+1)}\left(Aa^{4} + 4ca^{3} + 8(n+2)a^{2}\right)^{n+1}.
\end{align}
Supposing \eqref{evf} proved also for $p \in \cut(x_{0})$, letting $a \to \infty$ in \eqref{evf} shows $u \leq A^{n+1}/B$ on $\Omega$.
The proof of \eqref{evf} in the case $p\in \cut(x_{0})$ is accomplished using the same trick from \cite{Calabi-hopfmaximum} as at the end of the proof of Theorem \ref{cyestimatetheorem}. Namely, on a minimizing geodesic joining $x_{0}$ to $p$ there is a point $\bar{x}_{0}$ lying strictly between $x_{0}$ and $p$, at some distance $\ep> 0$ from $x_{0}$, and there can be used $\bar{r} = d(\bar{x}_{0}, \dum)$ in place of $r$. By the triangle inequality, $\bar{r} + \ep \geq r$. On the other hand $\bar{r}(p) + \ep = r(p)$. Define $\bar{v} = (a^{2} - (\bar{r} + \ep)^{2})^{2(n+1)}u$. Then $\bar{v} \leq v$ on $N$ and $\bar{v}(p) = v(p)$, so $\bar{v}$ attains its maximum value on $N$ at $p$. As $\bar{r}$ is smooth near $p$ the preceeding argument goes through with $\bar{v}$ in place of $v$, and letting $\ep \to 0$ at the end yields \eqref{evf}.
\end{proof}

\section{The canonical potential of a proper convex cone}

\subsection{}\label{cysection}
This section begins with some preliminary material needed for the statement and proof of Theorem \ref{affinespheretheorem}. 

Let $\ste$ be an $(n+1)$-dimensional vector space, $\std$ its dual, and write $\projp(\ste)$ and $\projp(\std)$ for their oriented projectivizations (the associated projective spheres). A subset of $\projp(\ste)$ is \textbf{convex} if its intersection with every projective line is a connected interval (possibly a point, or empty), and is \textbf{proper} if it contains no pair of antipodal points. The \textbf{cone} $\cone(S)$ over a subset $S \subset \projp(\ste)$ is the pre-image of $S$ under the defining projection $\pi:\stez \to \projp(\ste)$. Given a proper open convex cone $\Omega \subset \ste$ and its dual $\Omega^{\ast} \subset \std$, let $\projp(\Omega)$ and $\projp(\Omega^{\ast})$ be their oriented projectivizations, which are properly convex open subsets of $\projp(\ste)$ and $\projp(\std)$, respectively. Clearly $\cone(\projp(\Omega)) = \Omega$, and $\projp(\cone(S)) = S$. 

That the canonical potential of the \textbf{standard orthant} $\orthant = \{x \in \rea^{n+1}: x^{s} > 0, 1 \leq s \leq n+1\}$ is $u(x) = -\sum_{i = 0}^{n}\log x^{i}$, and the associated complete Hessian metric is flat, are verified by straightforward computations. The image $ = A^{-1}\orthant$ of $\orthant$ under the inverse of $A_{i}\,^{j} \in GL(\rea^{n+1})$ is $\orthant_{A} = \{x \in \rea^{n+1}: \ell^{s}_{A}(x) > 0, 1 \leq s \leq n+1\}$ where $\ell^{i}_{A}(x) = A_{p}\,^{i}x^{p}$. That $u_{A}(x) = -\sum_{s = 1}^{n+1}\log\ell^{s}_{A}(x) + \log |\det A| = A^{-1}\cdot u + \log|\det A|$ is the canonical potential of $\orthant_{A}$ follows from the proof of Theorem \ref{cytheorem}. 

If $\Omega\subset \rea^{n+1}$ is a proper open convex cone, its projectivization $\projp(\Om)$ is a bounded convex domain, and so there are simplices $\Sigma$ and $\Sigma^{\prime}$ such that $\Sigma \subset \projp(\Om) \subset \Sigma^{\prime}$. The cones $\cone(\Sigma)$ and $\cone(\Sigma^{\prime})$ over these simplices have the forms $\orthant_{A}$ and $\orthant_{B}$ for some $A, B \in GL(\rea^{n+1})$, and so $\orthant_{A} \subset \Omega \subset \orthant_{B}$. It follows from Corollary \ref{gfcorollary} that $u_{B} \leq u_{A}$ on $\orthant_{B}$. Consequently the functions
\begin{align}\label{uomega}
\begin{split}
U_{\Om}(x) &= \sup\{u_{A}(x): \orthant_{A} \subset \Om\},\\
U^{\Om}(x) &= \sup\{G(x): G \,\, \text{convex in} \,\,\Omega, \,\,\text{and} \\
&\qquad G \leq u_{A} \,\,\text{for all}\,\, A \in GL(n+1, \rea) \,\,\text{such that}\,\, \Om \subset \orthant_{A}\}
\end{split}
\end{align}
are finite on the interior of $\Om$ and satisfy $U_{\Om}(x) \leq U^{\Om}(x)$ for $x \in \Om$. Since a supremum of convex functions is convex, both $U_{\Om}$ and $U^{\Om}$ are convex. Let $y$ be any point in the boundary of $\Omega$ and let $[y]$ be its image in $\projp(\Omega)$. There can be chosen simplices $\Sigma$ and $\Sigma^{\prime}$ such that $\Sigma \subset \projp(\Om) \subset \Sigma^{\prime}$ and so that $[y]$ lies in their boundaries, in which case the entire ray spanned by $y$ is contained in the boundaries of $\cone(\Sigma) = \orthant_{A}$ and $\cone(\Sigma^{\prime}) = \orthant_{B}$. It follows that $U_{\Om}(x_{i}) \to \infty$ and $U^{\Om}(x_{i})\to \infty$ as $x_{i} \in \Om$ tends to a point a of the boundary of $\Om$.

Suppose for the moment that the existence of the canonical potential of $\Om$ is not known. Since $\Om$ is proper, the set $\con(\Om)$ of $C^{2}$ convex functions $G$ on $\Om$ that are $-(n+1)$-logarithmically homogeneous and satisfy $\H(G) \geq e^{2G}$ on $\Om$ is nonempty for it contains the restriction to $\Om$ of $u_{A}$ for any $A \in GL(\rea^{n+1})$ such that $\Q \subset \orthant_{A}$. If $x_{0} \in \Om$ there are $A, B \in GL(\rea^{n+1})$ such that $x_{0} \in \orthant_{A} \cap \orthant_{B}$ and $\orthant_{A} \subset \Om \subset \orthant_{B}$. If $G \in \con(\Om)$ it follows from Corollary \ref{gfcorollary} applied to $G$ and $u_{A}$ on $\orthant_{A}$ that $G(x_{0}) \leq u_{A}(x_{0})$, and from Corollary \ref{gfcorollary} applied to $G$ and $u_{B}$ on $\Om$ that $u_{B}(x_{0}) \leq G(x_{0})$. It follows that $V(x) = \sup_{G \in \con(\Om)}G(x)$ is finite for all $x \in \Om$. Since a supremum of convex functions is convex, $V$ is a convex function on $\Om$. It is apparent that $V(x)$ is $-(n+1)$-logarithmically homogeneous. It follows from the preceeding that $U_{\Om} \leq V \leq U^{\Om}$. Note that it follows that $U(x_{i}) \to \infty$ as $x_{i} \in \Om$ tend to a point of the boundary of $\Om$. 

A ray $L$ is said to be an \textbf{asymptotic ray} of a hyperbolic affine sphere $\Sigma$ if it does not intersect $\Sigma$ and there is an unbounded sequence $\{x_{n}\}$ of points of $\Sigma$ such that the rays from the center of $\Sigma$ to $x_{n}$ converge to $L$. This implies that the image of $L$ in $\projp(\rea^{n+1})$ is contained in the boundary of $\projp(\Sigma)$. The affine sphere $\Sigma$ is said to be \textbf{asymptotic} to the boundary $\pr \Om$ of the cone $\Om$ if $\pr \Omega$ is a union of asymptotic rays of $\Sigma$. 

\begin{theorem}\label{affinespheretheorem}
Let $\Omega \subset \rea^{n+1}$ be a proper open convex cone with canonical potential $F$. Then: 
\begin{enumerate}
\item \label{cysupcone}
For all $x \in \Omega$,
\begin{align}
F(x) = \sup\{G(x): G \in C^{2}(\Omega) \cap \lmg_{-n-1}(\Om), G \,\, \text{convex in} \,\,\Omega, \,\,\text{and}\,\, \H(G) \geq e^{2G}\}.
\end{align}
and $U_{\Om}(x) \leq F(x) \leq U^{\Om}(x)$ for the functions $U_{\Om}$ and $U^{\Om}$ defined in \eqref{uomega}. 
\item \label{cylog} $F$ is $-(n+1)$-logarithmically homogeneous and is auto-harmonic in the sense that $\lap_{g}F = 0$.
\item \label{cylevel} The $r$-level set of $F$ is a complete hyperbolic affine sphere asymptotic to the boundary of $\Omega$, centered on the vertex of $\Omega$, and having affine mean curvature $-(n+1)^{-(n+1)/(n+2)}e^{2r/(n+1)}$.
\item\label{loftinlemma} [J. Loftin, \cite{Loftin-affinekahler}]
Let $h$ be the equiaffine metric on $\lc_{0}(F, \Omega)$ and equip $\reap \times \lc_{0}(F, \Omega)$ with the metric $k = (n+1)\left(r^{-2}dr^{2} + (n+1)^{-(n+1)/(n+2)}h\right))$. The maps 
\begin{align}
\begin{split}
x &\in \Omega \to (e^{-F(x)/(n+1)}, e^{F(x)/(n+1)}x) \in \reap \times \lc_{0}(F, \Omega),\\
(r, y) &\in \reap \times \lc_{0}(F, \Omega) \to ry \in \Omega,
\end{split}
\end{align}
are inverse isometries between the Riemannian manifolds $(\Omega, g)$ and $(\reap \times \lc_{0}(F, \Omega), k)$. Moreover, $\lc_{r}(F, \Omega)$ corresponds to $\{e^{-2r/(n+1)} \} \times \lc_{0}(F, \Omega) \subset \reap \times \lc_{0}(F, \Omega)$.
\item \label{cyfol} Any complete hyperbolic affine sphere with center at the vertex of $\Om$ and asymptotic to $\pr \Om$ is a level set of $F$.
\end{enumerate}
\end{theorem}
Modulo notation and terminology, \eqref{loftinlemma} of Theorem \ref{affinespheretheorem} is Theorem $3$ of \cite{Loftin-affinekahler}. The result for the usual characteristic function $\phi_{\Omega}$ of $\Omega$ analogous to \eqref{loftinlemma} is given in Proposition $11$ of \cite{Sasaki-characteristicfunctions} (see also \cite{Sasaki} and \cite{Tsuji}). 

\begin{proof}
Claim \eqref{cysupcone} was proved in the discussion preceeding the statement of the theorem. The logarithmic homogeneity of $F$ follows from \eqref{cyhom} of Theorem \ref{cytheorem} applied to the dilations preserving the cone. In particular, for $t \in \rea$  the function $(t\cdot F)(x) = F(e^{-t}x) - (n+1)t$ solves $\H(t\cdot F) = e^{2(t\cdot F)}$ on $\Omega$ and so equals $F$. By the logarithmic homogeneity, $dF(\eul) = -n-1$ and $|\eul|_{g}^{2} = n+1$. Let $D$ be the Levi-Civita connection of $g_{ij}$. Since $F^{i} = -\eul^{i}$ and $D_{i}F_{j} = 0$, the radial vector field $\eul^{i}$ is $g$-Killing. In particular $\lap_{g}F = 0$, showing \eqref{cylog}. For $r \in \rea$ consider $\lc_{r}(F, \Omega)$. By \eqref{affnorm} and $\eul^{i} = -F^{i}$, the positively co-oriented equiaffine normal is $\nm^{i} = (n+1)^{-(n+1)/(n+2)}e^{2F/(n+2)}\eul^{i}$, and the $g$-unit normal is $(n+1)^{-1/2}\eul^{i}$. The affine mean curvature of $\lc_{t}(F, \Omega)$ is $-(n+1)^{-(n+1)/(n+2)}e^{2t/(n+1)}$. This shows all of \eqref{cylevel} except for the completeness. Since $\eul^{i}$ is $g$-parallel, $\lc_{r}(F, \Omega)$ is a $g$-totally geodesic hypersurface, and its second fundamental form with respect to $-\eul^{i}$ is just the restriction $(n+1)^{-1}g_{IJ}$. Since, along $\lc_{t}(F, \Omega)$, $\nm^{i}$ is a homothetic rescaling of $\eul^{i}$, the second fundamental form with respect to $\nm^{i}$ is simply $h_{IJ} = (n+1)^{-1/(n+2)}e^{-2F/(n+2)}g_{IJ}$. It follows from these observations and \eqref{sffdefined} that the second fundamental form of $\lc_{r}(F, \Omega)$ is the restriction of the tensor 
\begin{align}
\label{hglobal2}
h_{ij} = (n+1)^{-1/(n+2)}e^{-2F/(n+2)}(g_{ij} - \tfrac{1}{n+1}F_{i}F_{j}). 
\end{align}
Observe that if $x \in \Omega$ then $e^{F(x)/(n+1)}x \in \lc_{0}(F, \Omega)$. This proves \eqref{loftinlemma}. Since $g$ is complete, it follows from the existence of the isometry of \eqref{loftinlemma} that $h$ is complete. By Theorem $3$ of \cite{Cheng-Yau-affinehyperspheresI} a hyperbolic affine sphere complete with respect to the affine metric is Euclidean complete (this also follows from the more general Theorem A of \cite{Trudinger-Wang-affinecomplete} showing that, for $n \geq 2$, an affine complete locally uniformly convex hypersurface in $\rea^{n+1}$ is Euclidean complete). By \cite{Vanheijenoort} and \cite{Sacksteder}, a Euclidean complete locally uniformly convex hypersurface is convex so lies on the boundary of a convex region. Since the space of rays is compact, given an unbounded sequence of points $x_{n}$ in a complete hyperbolic affine sphere $\Sigma$ with center $q$, after passing to a subsequence if necessary, the rays from $q$ to $x_{n}$ converge to a ray $L$ which does not intersect $\Sigma$, so is an asymptotic ray of $\Sigma$. By the homogeneity of $F$ the sets of asymptotic rays of the level sets $\lc_{r}(F, \Omega)$ coincide for different values of $r$ and lie in $\pr \Om$. Conversely, if $L$ is a ray contained in $\pr \Omega$, let $\{y_{n}\}$ be a sequence of points such that the rays passing through the $y_{n}$ converge to $L$. For any $r$ these rays intersect $\lc_{r}(F, \Om)$ in a sequence of points $\{x_{n}\}$ which cannot be bounded in $\lc_{r}(F, \Om)$ (were it, the rays would converge to a ray intersecting $\lc_{r}(F, \Om)$, which $L$ by assumption does not). Hence $L$ is an asymptotic ray of $\lc_{r}(F, \Om)$, and so $\pr \Om$ is a union of asymptotic rays of $\lc_{r}(F, \Om)$. This completes the proof of \eqref{cylevel}.

Suppose $\Sigma$ is an affine complete hyperbolic affine sphere with center at the vertex of $\Om$ and asymptotic to $\pr \Om$. For $r > 0$ let $r\Sigma = \{rx: x \in \Sigma\}$, which is also an affine complete hyperbolic affine sphere with center at the vertex of $\Om$ and asymptotic to $\pr \Om$. Again by Theorem $3$ of \cite{Cheng-Yau-affinehyperspheresI}, each $r\Sigma$ is Euclidean complete and so, as above, lies on the boundary of a convex region and can be represented as the graph of a convex function. It follows that $r\Omega$ and $s\Omega$ are disjoint if $r \neq s$ and that any $x \in \Om$ is contained in $r\Sigma$ for some $r > 0$. Hence the disjoint union $\cup_{r > 0}r\Sigma$ equals $\Om$. The function $G$ defined to equal $-(n+1)\log r$ on $r\Sigma$ is evidently smooth and so contained in $\lmg_{-n-1}$. As $G$ satisfies the hypotheses of Theorem \ref{ahtheorem} with $I = \rea$ and $\al = -(n+1)$, there holds $\H(G) = Be^{2G}$ for some constant $B \in \reat$. The co-orientation convention on the equiaffine normal to $\Sigma$ forces $\nm^{i} = c\eul^{i}$ with positive $c$, and a bit of calculation shows $B$ must be positive. By replacing $G$ by itself plus an appropriate constant, $B$ can be normalized to $1$. By Theorem \ref{schwarztheorem}, $G \leq F$. Since $G$ is logarithmically homogeneous, the same argument as that used above to show \eqref{loftinlemma} shows that $\Om$ equipped with the metric $\nabla_{i}dG_{j}$ is isometric to the product metric on the product of the real line with $\Sigma$, equipped with its equiaffine metric. Since the equiaffine metric on $\Sigma$ is assumed complete, it follows that $\nabla_{i}dG_{j}$ is complete, and so Theorem \ref{schwarztheorem} implies $G \geq F$. This proves that $\Sigma$ must be a level set of $F$, showing \eqref{cyfol}.
\end{proof}

\subsection{}\label{uniquenessexplanationsection}
Suppose $\Omega$ is a proper convex domain and $F \in \cinf(\Omega)$ solves $\H(F) = e^{2F}$ and $g_{ij} = \nabla_{i}dF_{j}$ is complete. It follows that the K\"ahler metric $G_{i\bar{j}}= F_{ij}(x)dz^{i}\tensor d\bar{z}^{j}$ on $\tube_{\Omega}$ determined by the potential $\pi^{\ast}(F)$ is a complete K\"ahler Einstein metric with negative scalar curvature. By Yau's Schwarz lemma for volume forms a biholomorphism between complete K\"ahler Einstein metrics having negative scalar curvature is an isometry. This can be used to deduce the uniqueness of $F$. An affine automorphism $g$ of $\Omega$ extends complex linearly to a biholomorphism of the tube $\tube_{\Omega}$. Writing the action of $g$ on $\tube_{\Omega}$ by $L_{g}$, it follows that $L_{g}^{\ast}(dd^{c}\pi^{\ast}(F)) = dd^{c}\pi^{\ast}(F)$. Hence $dd^{c}(\pi^{\ast}(g\cdot F - \log |\det \ell(g)|)) = L_{g}^{\ast}(dd^{c}\pi^{\ast}(F)) = dd^{c}\pi^{\ast}(F)$, from which it follows that $ L_{g}^{\ast}(\hess F) = \hess(g\cdot F - \log |\det \ell(g)|)=\hess F$, and so $g\cdot F - \log |\det \ell(g)| = \tfrac{1}{2}\log \H(g\cdot F - \log |\det \ell(g)|) = \tfrac{1}{2}\log \H(F)  = F$ and $g$ is an isometry of $\hess F$.

While the proof of the uniqueness of $F$ by passing to the tube domain and invoking the Schwarz lemma is a powerful demonstration of the efficacy of complex geometric methods, its relation to the geometry of the Hessian metric is perhaps somewhat obscure. Note that \textit{a priori} the ordinary Ricci curvature of $g_{ij}$ is not bounded below. When $\Om$ is a proper convex cone, such a bound is obtained as a consequence of the uniqueness, as this implies the logarithmic homogeneity of $F$, from which such a bound follows via \eqref{xxric3}. As was explained in the introduction, the more fundamental lower bound is that on the K\"ahler affine Ricci tensor, which corresponds directly to the Ricci tensor on $\tube_{\Omega}$ and which via \eqref{cd1} yields a lower bound on the $2(n+1)$-dimensional Bakry-Emery Ricci tensor. 

This point of view makes clear that the existence of the canonical potential of a proper convex cone $\Om$ and the associated foliation of the interior of $\Omega$ by hyperbolic affine spheres is closely related to the specialization to the tube domain $\tube_{\Omega}$ of the result of N. Mok and Yau in \cite{Mok-Yau} showing the existence of a negative K\"ahler Einstein metric on a bounded domain of holomorphy. Every tube domain over a proper convex domain is biholomorphic to a bounded domain of holomorphy because the base is contained in an affine image of an orthant, and the tube over an orthant is a product of disks. By the main theorem of \cite{Mok-Yau} a bounded domain of holomorphy admits a complete K\"ahler Einstein metric with negative scalar curvature which is, moreover, unique up to homothety. The K\"ahler metric determined on $\tube_{\Omega}$ by the pullback $\pi^{\ast}(F)$ of the canonical potential $F$ of $\Om$ is the special case of this K\"ahler Einstein metric when the domain of holomorphy is the tube over a proper open convex cone.

\subsection{}\label{slicesection}
The usual proof of the existence of the foliation of the interior of a proper convex cone by affine spheres goes through a different theorem of Cheng and Yau, resolving a different, but related Monge-Amp\`ere equation. Next it is briefly indicated how this is related to the approach described here. Let $F \in \cinf(\Om)$ and suppose that on $\lc_{r}(F, \Omega)$ there vanishes neither $dF$ nor $\H(F)$. After a rotation of the coordinates, there can be found an open domain $D \subset \rea^{n}$ and $f \in \cinf(D)$ such that the graph $\{(x, f(x)) \in D \times \rea: x \in D\}$ is contained in $\lc_{r}(F, \Om)$ and the partial derivative $F_{0}$ in the $x^{0}$ direction is not zero along this graph. Capital Latin indices indicate derivatives in the directions of the coordinates on $\rea^{n}$. Differentiating $r = F(x, f(x))$ yields
\begin{align}\label{Ff}
\begin{split}
0 & = F_{I} + F_{0}f_{I},\qquad
0  = F_{IJ} + F_{I0}f_{J} + F_{J0}f_{I} + F_{00}f_{I}f_{J} + F_{0}f_{IJ},
\end{split}
\end{align}
so that
\begin{align}\label{twofs}
\begin{split}
F_{0}^{3}f_{IJ} & = - F_{0}^{2}F_{IJ} + F_{0}(F_{I0}F_{J} + F_{J0}F_{I}) - F_{00}F_{I}F_{J}. 
\end{split}
\end{align}
By \eqref{Ff}, $x^{Q}F_{Q} = dF(\eul) - fF_{0} = -(n+1) - fF_{0}$, so that the Legendre transform $f^{\ast}$ of $f$ is
\begin{align}\label{fastdf}
f^{\ast} = x^{I}f_{I} - f = -x^{I}(F_{I}/F_{0}) - f = - dF(\eul)/F_{0}. 
\end{align}
Define $\H(f)$ and $\H(f^{\ast})$ in the same manner as $\H(F)$, though with respect to the coordinates $x^{I}$ on $\rea^{n}$, and the coordinates $y_{I} = \tfrac{\pr f}{\pr x^{I}}$ on the image of $D$ under $df$. Evaluating the determinant using elementary row and column operations yields
\begin{align}\label{ufdet}
\begin{vmatrix} F_{ij} & F_{i} \\ F_{j} & 0 \end{vmatrix}  = \begin{vmatrix} F_{ij} & 0 \\ F_{j}& - |dF|_{g}^{2}\end{vmatrix} = -\H(F)|dF|^{2}_{g}.
\end{align}
Combining \eqref{ufdet} and \eqref{Ff}, and computing as in the proof of Theorem $4$a of \cite{Sasaki} gives
\begin{align}\label{ufgraph}
\begin{split}
-\H(F)|dF|^{2}_{g} & = \begin{vmatrix} F_{IJ} & F_{I0} & F_{I} \\ F_{J0} & F_{00} & F_{0}\\ F_{J} & F_{0} & 0 \end{vmatrix} = \begin{vmatrix} -F_{0}f_{IJ} & F_{I0} + F_{00}f_{I} & 0 \\ F_{J0} + F_{00}f_{J} & F_{00} & F_{0}\\ 0 & F_{0} & 0\end{vmatrix}
\\& = - F_{0}\begin{vmatrix} -F_{0}f_{IJ} & 0 \\  F_{J0} + F_{00}f_{J} & F_{0}\end{vmatrix} 
= (-1)^{n+1}(F_{0})^{n+2}\H(f).
\end{split}
\end{align}
From \eqref{ufdet} and \eqref{ufgraph} it follows that
\begin{align}\label{hfast}
\begin{split}
\H(f^{\ast}) &= \H(f)^{-1} = (-1)^{n}\H(F)^{-1}|dF|^{-2}_{g}F_{0}^{n+2} = \H(F)^{-1}|dF|^{-2}_{g}dF(\eul)^{n+2}(f^{\ast})^{-n-2}.
\end{split}
\end{align}
As explained in the proof of Theorem \ref{cytheorem}, a level set of the canonical potential $F$ of a proper open convex cone $\Om$ can be written as a graph over $D = \rea^{n}$. In this case, substituting $dF(\eul) = -(n+1)$, $|dF|^{2}_{g} = n+1$, and $H(F) = e^{2F}$ in \eqref{hfast} yields
\begin{align}
\begin{split}
\H(f^{\ast})&= (n+1)^{n+1}e^{2F}(-f^{\ast})^{-n-2}.
\end{split}
\end{align}
Along a level set of $F$ the righthand side is a constant times $(-f^{\ast})^{-n-2}$, and it follows from Lemma \ref{aslemma} that the graph of $f$ and the radial graph of $f^{\ast}$ are hyperbolic affine spheres. 

\begin{lemma}\label{aslemma}
The following are equivalent for a locally uniformly convex function $u \in \cinf(\Omega)$.
\begin{enumerate}
\item\label{as1} There is $c \neq 0$ such that 
\begin{align}\label{cyma}
u^{n+2}\H(u) = c^{-n-2}.
\end{align}
\item\label{as2} The graph of the Legendre transform of $u$ is a mean curvature $c$ affine sphere.
\item The radial graph $\{(u(x)^{-1}, -u(x)^{-1}x^{i}) \in \rea^{n+1}\}$ is a mean curvature $c^{-1}$ affine sphere.
\end{enumerate}
\end{lemma}
\noindent
All the equivalences of Lemma \ref{aslemma} are proved by direct calculations. The equivalence \eqref{as1}$\iff$\eqref{as2} is Proposition $1.2$ of \cite{Calabi-completeaffine} (see also the proof of Theorem $3$ in \cite{Loftin-affinespheres}). The remaining equivalences were proved using different language in \cite{Gigena-integralinvariants} (see Proposition $4.3$).

Because of Lemma \ref{aslemma}, to construct an affine sphere it is enough to solve the equation \eqref{cyma}. Precisely, Cheng and Yau proved that for $c < 0$ the equation \eqref{cyma} has on a bounded convex domain in $\rea^{n}$ a unique negative convex solution vanishing on the boundary of the domain, and that the resulting affine sphere is properly embedded and asymptotic to the cone over the sphere (the $n = 2$ case is due to \cite{Loewner-Nirenberg}). The relation between the solution $u$ of $u^{n+2}\H(u) = c^{-n-2}$ and the canonical potential $F$ of $\Om$ was explained above. It can be seen as an instance of the passage from a projective picture (inhomogeneous coordinates) to an affine picture (homogeneous coordinates). Viewing $\Om$ as an $\reap$ principal bundle over $\projp(\Om)$, the function $u$ (that is, the Legendre transform $f^{\ast}$) can be understood invariantly as the section of the line bundle of $-1/(n+1)$ densities on $\projp(\Om)$ corresponding to the equivariant function $e^{-F/(n+1)}$. That the theory of hyperbolic affine spheres can be founded on the equation $\H(F) = e^{2F}$ rather than on the equation \eqref{cyma} seems not widely recognized beyond the articles \cite{Loftin-affinekahler} and \cite{Cheng-Yau-mongeampere} (in the latter, the connection is only implicit).

\subsection{}\label{characteristicfunctionsection}
In this section it is shown that the Legendre transform interchanges the canonical potentials of dual cones. The precise statement is Theorem \ref{fdualitytheorem}.

\begin{lemma}\label{bijectionlemma}
If $F$ is the canonical potential of the proper open convex cone $\Omega \subset \rea^{n+1}$, the map $x \to \cp_{\Omega}(x)= -dF(x)$ is a diffeomorphism from $\Omega$ onto $\Omega^{\ast}$.
\end{lemma}
\begin{proof}
For $x, \bar{x} \in \Omega$, since the plane tangent to the level set of $F$ passing through $e^{t}x$ is the parallel translate of the plane tangent to the level set of $F$ through $x$, there are a unique vector $v$ tangent to this plane, and a unique $t \in \rea$ such that $\bar{x} = e^{t}x + v$. Then $-\bar{x}^{p}F_{p}(x) = (n+1)e^{t} > 0$. Since $\bar{x}$ is an arbitrary element of $\Om$, this shows $-F_{i}(x) \in \Om^{\ast}$. Since the differential of $x \to -dF(x)$ is a multiple of the Hessian of $F$, which is nondegenerate, this map is a local diffeomorphism, and so to prove the claim it is enough to prove that $-dF(x)$ is a bijection. With the obvious modifications, this is shown just as in the proof of the analogous claim regarding the negative of the logarithmic differential of the characteristic function of $\Omega$, which can be found as Proposition I.$3.4$ of \cite{Faraut-Koranyi} or Proposition $4.1.12$ of \cite{Shima}. Given $y_{i} \in\Omega^{\ast}$ the subset $L = \{x^{i}\in \Omega: x^{p}y_{p} = n+1\}$ is compact and so $F$ assumes a local minimum at some $x$ in the interior of $L$. At such an $x$ there holds $F_{i}(x) = \la y_{i}$, and so $n+1 = -x^{p}F_{p}(x) = -\la x^{p}y_{p} = -(n+1)\la$, showing $\la = -1$ and so $y_{i} = -F_{i}(x)$. 
Since $F(z)$ goes to infinity as $z$ tends to the boundary of $\Omega$, it follows from the convexity of $F$ that it cannot have two relative minima on $L$. 
This suffices to prove the claim.
\end{proof}

For a convex cone $\Omega$ let $-\Omega = \{x \in \rea^{n+1}:-x \in \Omega\}$ and note that $(-\Omega)^{\ast} = -(\Omega^{\ast})$, so that $-\Omega^{\ast}$ has an unambiguous meaning. Let $F_{\Omega}$ be the canonical potential of the proper open convex cone $\Omega$. Since, as is easily verified, $F_{-\Omega}(-x)$ has all the same properties on $\Omega$ as has $F_{\Omega}(x)$, it follows from Theorem \ref{schwarztheorem} that $F_{-\Omega}(-x) = F_{\Omega}(x)$. By Lemma \ref{bijectionlemma} the differential $dF_{\Omega}$ maps $\Omega$ diffeomorphically onto $-\Omega^{\ast}$, and so the Legendre transform $F^{\ast}_{\Omega}$ is the smooth function on $-\Omega^{\ast}$ defined by $F^{\ast}_{\Omega}(dF_{\Omega}(x)) + F_{\Omega}(x) = - n - 1$. Since on $-\Omega^{\ast}$ there holds
\begin{align*}
\H(F^{\ast}_{\Omega} + n+1)(dF_{\Omega}(x)) & = \H(F^{\ast}_{\Omega})(dF_{\Omega}(x)) = (\H(F_{\Omega})(x))^{-1} = e^{-2F_{\Omega}(x)} = e^{2(F^{\ast}_{\Omega}(dF(x)) + n+1)},
\end{align*}
it follows from Corollary \ref{gfcorollary} that $F^{\ast}_{\Omega} + n + 1 \leq F_{-\Omega^{\ast}}$, in which $F_{-\Omega^{\ast}}$ is the canonical potential of $-\Omega^{\ast}$. From $F_{-\Omega}(-x) = F_{\Omega}(x)$ for $x \in \Om$ there follows $\cp_{-\Omega}(-x) = -\cp_{\Omega}(x) = dF_{\Om}(x)$. By a standard property of the Legendre transform it follows that for $x \in \Om$ there holds
\begin{align}
\begin{split}
\cp_{\Omega}^{\ast}(\nabla dF_{-\Omega}^{\ast})(x) & = -\cp_{-\Omega}^{\ast}(\nabla dF_{-\Om}^{\ast})(-x) = (\nabla dF_{-\Om})(-x) = (\nabla dF_{\Om})(x),
\end{split}
\end{align}
which shows that $\cp_{\Omega}:\Omega \to \Omega^{\ast}$ is an isometry with respect to the Hessian metrics determined by the potentials $F_{-\Omega}$ and $F_{\Omega}^{\ast}$. It follows that the Hessian of $F_{\Omega}^{\ast} + n+1$ is complete, and so a second use of Theorem \ref{schwarztheorem}, with the roles of $F_{-\Omega^{\ast}}$ and $F_{\Omega}^{\ast} + n+1$ interchanged shows that $F_{-\Omega^{\ast}} \leq F_{\Omega}^{\ast} + n+1$. 
\begin{theorem}\label{fdualitytheorem}
For a proper open convex cone $\Omega$ and the dual cone $\Omega^{\ast}$ the canonical potentials $F_{\Omega}$ and $F_{\Omega^{\ast}}$ are related by
\begin{align}\label{fomegaast}
F_{\Omega^{\ast}}(-y) - \tfrac{n+1}{2} = F_{-\Omega^{\ast}}(y) - \tfrac{n+1}{2} = F_{\Omega}^{\star}(y) + \tfrac{n+1}{2} = (F_{\Omega} - \tfrac{n+1}{2})^{\star}(y).
\end{align} 
for $y \in -\Omega^{\ast}$. Moreover, the inverse diffeomorphisms $\cp_{\Om}:\Om \to \Om^{\ast}$ and $\cp_{\Om^{\ast}}:\Om^{\ast} \to \Om$ are isometries with respect to the canonical Hessian metrics and there hold: 
\begin{enumerate}
\item\label{cf1} $\cp_{\Omega}(g \cdot x) = g^{-1} \cdot \cp_{\Omega}(x)$ for all $g \in \Aut(\Omega)$.
\item\label{cf2} $e^{F_{\Omega}(x)}e^{F_{\Omega^{\ast}}(\cp_{\Omega}(x))} = 1$. 
\item\label{cf3} The diffeomorphisms $\cp_{\Om}:\Om \to \Om^{\ast}$ and $\cp_{\Om^{\ast}}:\Om^{\ast} \to \Om$ are isometries with respect to the canonical metrics and are inverses, meaning $\cp_{\Omega^{\ast}}\circ \cp_{\Omega} = \Id$.
\end{enumerate}
\end{theorem}

\begin{proof}
Essentially all the claims were proved in the preceeding discussion. Note that \eqref{cf2} is a restatement of \eqref{fomegaast}, while \eqref{cf1} follows by differentiating $g\cdot F = F$ for $g \in \Aut(\Om)$. 
\end{proof}

\noindent
From Theorems \ref{fdualitytheorem} and \ref{affinespheretheorem} it follows that the affine spheres foliating the proper open convex cone $\Omega$ and those foliating the dual cone $\Omega^{\ast}$ are dual; this is essentially because along a level set of the canonical potential $F$ the map $\cp_{\Omega}$ is a constant multiple of the equiaffine conormal of a level set of $F$.

\subsection{}\label{homogeneoussection}
Suppose the proper open convex domain $\Om$ is homogeneous. Since for $g \in \Aut(\Omega)$ there hold $g \cdot \H(\log \phi_{\Omega}) = (\det \ell(g))^{2} \H(\log \phi_{\Omega})$ and $g\cdot \phi_{\Omega} = \det \ell(g)\phi_{ \Omega}$, the homogeneity $0$ function $\phi_{\Omega}^{-2}\H(\log \phi_{\Omega})$ is $\Aut(\Omega)$ invariant. Because $\Omega$ is homogeneous, $\phi_{\Omega}^{-2}\H(\log \phi_{\Omega})$ must be constant, and because $\nabla d\log \phi_{\Omega}$ is positive definite, the constant must be positive, so there is a constant $c > 0$ such that $\H(\log \phi_{\Omega}) = c^{2} \phi_{\Omega}^{2}$ (this observation is implicit in the proof of Theorem $4a$ of \cite{Sasaki}; see also equation $(6)$ of section $1$ of chapter II of \cite{Koecher} or exercise $7$ of chapter I of \cite{Faraut-Koranyi}). By \eqref{hpsif}, it follows from $\H(\log \phi_{\Omega}) = c^{2}\phi_{\Omega}^{2}$ that $\H(\phi_{\Omega}) = c^{2}(n+2)\phi_{\Omega}^{n+3}$ and $\H(\phi_{\Omega}^{-2}) = (-2)^{n}(4n+2)c^{2}\phi_{\Omega}^{-2n}$. By the uniqueness part of Theorem \ref{cytheorem}, for a homogeneous proper open convex cone there must hold $e^{F_{\Om}} = c \phi_{\Omega}$ for some nonzero constant $c$. This suggests the following question: \textit{if the canonical potential $F_{\Om}$ of the proper open convex cone $\Om$ differs from $\log \phi_{\Om}$ by a constant, must the cone be homogeneous?}

Let $\delta_{ij}$ be the standard Euclidean metric. For a proper open convex cone $\Omega$ define $\Omega^{\sharp} = \{\delta_{ip}x^{p}: x \in \Omega\}$ to be the proper open convex cone in $(\rea^{n+1})^{\ast}$ Euclidean dual to $\Omega$. Then $\Omega$ is (Euclidean) \textbf{self-dual} if $\Omega^{\ast} = \Omega^{\sharp}$. A proper open convex cone is \textbf{symmetric} if it is homogeneous and self-dual. It is known that for a symmetric convex cone $\Omega$ the function $\phi_{\Omega}^{-2}$ is a polynomial (see e.g. Theorems $14$ and $15$ in section $5$ of chapter VI of \cite{Koecher}). Precisely, $\phi_{\Omega}^{-2}$ is a multiple of the square of the quadratic representation of the Euclidean Jordan algebra associated to $\Omega$. 
This suggests the following question: \textit{if, for a proper open convex cone $\Om$ with canonical potential $F_{\Om}$ some power of $e^{F_{\Om}}$}\textit{ is a polynomial, must $\Om$ be symmetric?}

The main theorem of \cite{Hu-Li-Vrancken} can be rephrased as saying that the asymptotic cone of a homogeneous hyperbolic affine sphere is a product of (possibly one-dimensional) symmetric cones. In fact, this could be deduced from the polynomiality of $P = e^{2F_{\Om}}$ and the Sato-Kimura classification of prehomogeneous vectors spaces, for the action of $\Aut(G)$ on $\Om$ is a real prehomogeneous vector space underlying a complex prehomogeneous vector space having $P$ as a relative invariant. This argument has the virtue of applying to homogeneous affine spheres with equiaffine metrics of indefinite signature. As a full treatment requires some care, and discussion of the meaning for affine spheres of the castling transform used to define equivalence of prehomogeneous vector spaces, the details will be reported elsewhere.

For a homogeneous convex cone $\Omega$ and the usual characteristic function $\phi_{\Omega}$, the conclusions \eqref{cf1}-\eqref{cf3} of Theorem \ref{fdualitytheorem} are well known (see Proposition I.$1.4$ and the notes to chapter I of \cite{Faraut-Koranyi}). By \eqref{cylog} of Theorem \ref{affinespheretheorem}, $\cy_{\Omega} = e^{F}$ has positive homogeneity $-n-1$ and satisfies $\H(\cy_{\Omega}) = (n+2)\cy_{\Omega}^{n+3}$. As is illustrated by Theorem \ref{fdualitytheorem}, many things true for the usual characteristic function on a homogeneous convex cone are in fact true for $\cy_{\Om}$ on any convex cone, and so should be understood as consequences of the identity $\cy_{\Omega} = c \phi_{\Omega}$ in the homogeneous case. This suggests the point of view that homogeneity of $\Om$ forces $e^{F_{\Om}}$ to have a very nice integral representation. Is there always an integral representation for $e^{F_{\Om}}$ similar to \eqref{usualchardefined}? A more precise question is the following. The usual characteristic function is the Laplace transform of the indicator function of the dual convex cone, and it follows straightforwardly that it is completely monotone. If it could be shown that $e^{F_{\Om}}$ were completely monotone, then by the classical theorem of Bochner, it would be representable as the Laplace transform of some measure on the dual cone (see Theorem $4.2.2$ of \cite{Bochner}). Here that a function $A$ on the proper open convex cone $\Om$ be completely monotone means that for any $k \geq 0$ and any vectors $v(1), \dots v(k)$ in $\Om$ there hold $(-1)^{k}v(1)^{i_{1}}\dots v(k)^{i_{k}}A_{i_{1}\dots i_{k}} \geq 0$. For $A = e^{F_{\Om}}$ this is straightforward for $0 \leq k \leq 2$, but it is not clear whether it is true for larger $k$.

\subsection{}\label{parakahlersection}
In this section it is explained how to construct from the canonical potential of a convex cone a mean curvature zero conical Lagrangian submanifold of the standard flat para-K\"ahler space. 

\begin{lemma}\label{ubijectionlemma}
Let $F$ be the canonical potential of the proper open convex cone $\Om \subset \rea^{n+1}$ and let $u = -((n+1)/2)e^{-2F/(n+1)}$. Then the map $x \to -du(x)$ is a diffeomorphism from $\Om$ to $\Om^{\ast}$.
\end{lemma}
\begin{proof}
If $x, \bar{x} \in \Om$, then $-\bar{x}^{p}u_{p}(x) = e^{-2F(x)/(n+1)}\bar{x}^{p}F_{p}(x) > 0$ by the proof of Lemma \ref{bijectionlemma}, so $-du(x) \in \Om^{\ast}$. By \eqref{hpsif} there holds $\H(u) = -1$, so $-du(x)$ is a local diffeomorphism, and to prove the lemma it suffices to show that $-du$ is a bijection. Since the components of $dF$ have homogeneity $-1$, if $du(x) = du(\bar{x})$ then $F_{i}(e^{2F(x)/(n+1)}x)= F_{i}(e^{2F(\bar{x})/(n+1)}\bar{x})$, and by the injectivity of $-dF$ this implies $\bar{x} = e^{2(F(x) - F(\bar{x}))/(n+1)}x$. That is there is $t \in \rea$ such that $\bar{x} = e^{t}x$. Since $u$ has homogeneity $2$, there results $u_{i}(x) = u_{i}(\bar{x}) = u_{i}(e^{t}x) = e^{t}u_{i}(x)$, so that $t = 0$ and $\bar{x} = x$, showing that $-du$ is injective. If $y \in \Om^{\ast}$ then by Lemma \ref{bijectionlemma} there is $\bar{x} \in \Om$ such that $y_{i} = -F_{i}(\bar{x})$. As $F$ maps each ray in $\Om$ diffeomorphically onto $\rea$ there is a unique $x$ on the ray generated by $\bar{x}$ such that $F(x) = - F(\bar{x})$. Since $F(e^{2F(x)/(n+1)}x) = F(x) - 2F(x) = -F(x) = F(\bar{x})$, $\bar{x} = e^{2F(x)/(n+1)}x$. Consequently, $y_{i} = - F_{i}(\bar{x}) = -F_{i}( e^{2F(x)/(n+1)}x = - e^{-2F(x)/(n+1)}F_{i}(x) = -u_{i}(x)$, showing that $-du$ is surjective.
\end{proof}

Endow the vector space $\sW = \rea^{n+1}\times \rea^{n+1\,\ast}$ with the flat para-K\"ahler structure generated by the symplectic form $\sOmega((x, y), (\bar{x}, \bar{y})) = \bar{x}^{p}y_{p} - x^{p}\bar{y}_{p}$ and the split signature metric $\sG((x, y), (\bar{x}, \bar{y})) = x^{p}\bar{y}_{p} + \bar{x}^{p}y_{p}$. The graph $\Ga_{\be, \Om}$ in $\sW$ of a one-form $\be$ defined on an open subset $\Om \subset \rea^{n+1}$ means the image of the associated smooth map $\overline{\be}:\Omega \to \sW$ defined by $\overline{\be}(x) = (x, \be_{x})$. It is Lagrangian if and only if $\be$ is closed. A submanifold of $\sW$ is \textbf{conical} if its intersection with any ray in $\sW$ is a connected open subinterval of the ray. The graph $\Ga_{\be, \Om}$ is conical is and only if $\be$ has homogeneity $2$ in the sense that $\lie_{\eul}\be = 2 \be$. There holds $\overline{\be}^{\ast}(\sG)_{ij} = 2\nabla_{(i}\be_{j)}$, so that the restriction of $\sG$ to $\Ga_{\be, \Om}$ is a pseudo-Riemannian metric if and only if $\nabla_{(i}\be_{j)}$ is nondegenerate, in which case $\be$ and $\Ga_{\be, \Om}$ are said to be \textbf{nondegenerate}. Lowering its last index using $\sOmega$ identifies the second fundamental form of a Lagrangian submanifold of $\sW$ with a completely symmetric covariant three tensor on the submanifold. A bit of calculation shows that for a closed one-form $\be$ the second fundamental form of $\Ga_{\be,\Om}$ is identified in this way with $\nabla_{i}\nabla_{j}\be_{k}$. It follows that the one form symplectically dual to the mean curvature vector field of a nondegenerate closed one-form is a constant multiple of $\nabla_{i}\log\det(\nabla\be)$, and so $\Ga_{\be, \om}$ has mean curvature zero if and only if $\det \nabla \be$ is $\nabla$ parallel. In particular, the graph of the differential $dv$ of a homogeneity $2$ function $v$ is a conical Lagrangian submanifold of $(\sW, \sOmega, \sG)$ that is nondegenerate if $\nabla_{i}dv_{j}$ is nondegenerate, and has moreover mean curvature zero if and only if $\H(v)$ is constant. Closely related constructions are described in section $5$ of \cite{Hitchin-speciallagrangian}.

\begin{theorem}\label{parakahlertheorem}
Let $F_{\Om}$ and $F_{\Om^{\ast}}$ be the canonical potentials of the proper open convex cone $\Om \subset \rea^{n+1}$ and its dual $\Om^{\ast} \subset \rea^{n+1\,\ast}$, and let $u_{\Om} = -((n+1)/2)e^{-2F_{\Om}/(n+1)}$ and $u_{\Om^{\ast}} = -((n+1)/2)e^{-2F_{\Om^{\ast}}/(n+1)}$. Then the graph $\Ga(-du_{\Om}, \Om)$ of $-du_{\Om}$ over $\Om$ and the graph $\Ga(-du_{\Om^{\ast}}, \Om^{\ast})$ of $-du_{\Om^{\ast}}$ over $\Om^{\ast}$ coincide and $\Ga(-du_{\Om}, \Om)= \Ga(-du_{\Om^{\ast}}, \Om^{\ast})$ is a mean curvature zero nondegenerate conical Lagrangian submanifold of the flat para-K\"ahler structure $(\sW, \sOmega, \sG)$ that projects diffeomorphically onto $\Om$ and $\Om^{\ast}$ under the canonical projections from $\sW$ onto its factors $\rea^{n+1}$ and $\rea^{n+1\,\ast}$.
\end{theorem}
\begin{proof}
 Combining Lemma \ref{ubijectionlemma} and the discussion preceeding the statement of the theorem proves everything except for the equality $\Ga(-du_{\Om}, \Om)= \Ga(-du_{\Om^{\ast}}, \Om^{\ast})$. Since for the Legendre transform $u_{\Om}^{\ast}$ the graph over $-\Om^{\ast}$ of $du^{\ast}_{\Om}$ equals $\Ga(-du_{\Om}, \Om)$, it suffices to show that $u_{\Om^{\ast}}(y) = u_{\Om}^{\ast}(-y)$. Since $u_{\Om}$ has homogeneity $2$, so too has $u_{\Om}^{\ast}$ homogeneity $2$ and so by definition of the Legendre transform and the homogeneity of $u_{\Om}^{\ast}$ there holds
\begin{align}\label{uo1}
\begin{split}
\tfrac{(n+1)^{2}}{4} & = \tfrac{(n+1)^{2}}{4}u_{\Om}(x)^{-1}u_{\Om}^{\ast}(du_{\Om}(x))
= u_{\Om}(x)u_{\Om}^{\ast}(dF_{\Om}(x)).
\end{split}
\end{align}
On the other hand, by \eqref{fomegaast} and the definition of the Legendre transform there holds
\begin{align}\label{uo2}
\begin{split}
u_{\Om}(x)u_{\Om^{\ast}}(-dF_{\Om}(x)) & = \tfrac{(n+1)^{2}}{4}e^{-2(F_{\Om}(x) + F_{\Om^{\ast}}(-dF_{\Om}(x)))/(n+1)} 
= \tfrac{(n+1)^{2}}{4}.
\end{split}
\end{align}
Comparing \eqref{uo1} and \eqref{uo2} shows $u_{\Om^{\ast}}(y) = u_{\Om}^{\ast}(-y)$ and completes the proof.
\end{proof}

A Hamiltonian action of the group $\reat \times \reat$ on $\sW$ is generated by ordinary dilations of $\sW$ and the product of dilation on $\rea^{n+1}$ with the contragredient dilation on $\rea^{n+1\,\ast}$. Symplectic reduction gives rise to a para-K\"ahler manifold of constant para-holomorphic sectional structure on either connected component of the complement of the canonical incidence correspondence in $\projp(\rea^{n+1})\times \projp(\rea^{n+1\,\ast})$ in the same manner that symplectic reduction gives rise to the Fubini-Study metric on complex projective space. Moreover, it can be shown that mean curvature zero nondegenerate conical Lagrangian submanifolds of $\sW$ descend to mean curvature zero Lagrangian submanifolds of these para-K\"ahler space forms in the same way that mean curvature zero conical Lagrangian submanifolds of complex Euclidean space descend to mean curvature zero Lagrangian submanifolds of complex projective space (a succinct review of the latter construction can be found in section $2$ of \cite{Mcintosh}). Since it would require considerable space to record the straightforward details in full, they are omitted.

\section{Applications to K\"ahler affine and Monge-Amp\`ere metrics}
In this section there are derived some general results about K\"ahler affine metrics involving conditions on the K\"ahler affine Ricci and scalar curvatures and utilizing Theorem \ref{cyestimatetheorem}.

\subsection{}
Associated to a K\"ahler affine metric $(\nabla, g)$ is the \textbf{dual} K\"ahler affine metric $(\bnabla, g)$ with the same underlying metric and the flat affine connection $\bnabla$ defined by $\bnabla = \nabla + F_{ij}\,^{k}$. That $\bnabla$ is flat follows straightforwardly from \eqref{hessiancurvature}, while that $(\bnabla, g)$ is K\"ahler affine follows from $\bar{F}_{ijk} = \bnabla_{i}g_{jk} = - F_{ijk}$. Here, as in what follows, a bar indicates an object associated to $(\bnabla, g)$. It follows that the Koszul forms are related by $\bar{H}_{i} = \bar{F}_{ip}\,^{p} = -F_{ip}\,^{p} = H_{i}$. If $\Psi$ is a parallel volume form for $\nabla$ and $F$ is a local potential for $(\nabla, g)$ then $\bar{\Psi} = \H(F)\Psi$ is $\bnabla$-parallel. Similarly, $\bnabla_{i}F^{j} = \delta_{i}\,^{j}$. It follows that if there is a vector field $\eul^{i}$ such that $\nabla_{i}\eul^{j} = \delta_{i}\,^{j}$ then $\bar{F} = n+1 + \eul^{p}F_{p} - F$ is a local potential for $(\bnabla, g)$. 

Let $\Om$ be a proper open convex cone with dual $\Om^{\ast}$. Combining the preceeding observations with Theorem \ref{fdualitytheorem} shows that the pullback of the canonical K\"ahler affine structure on $\Om^{\ast}$ via the map $\cp_{\Om}$ is the K\"ahler affine structure dual to the canonical K\"ahler affine structure on $\Om$. In particular, the pullback of the canonical potential of $\Om^{\ast}$ is $\bar{F}_{\Om} = n+1 + \eul^{p}F_{\Om\,p} - F_{\Om}$. That is, $\bar{F}_{\Om}$ is essentially the pullback of the Legendre transform of $F_{\Om}$ (transported to $\Om^{\ast}$). This last statement is made precise by the following lemma.

\begin{lemma}\label{duallegendrelemma}
Let $\Om$ be a domain in $\rea^{n+1}$ and let $F \in C^{3}(\Om)$ be a convex function such that the Legendre transform of $F$ is a diffeomorphism onto a domain $D \subset \rea^{n+1\,\ast}$. Let $\nabla$ be a flat affine connection on $\Om$, let $x^{i}$ be coordinates on $\Om$ such that $\nabla dx^{i} = 0$, let $y_{i} = F_{i}(x)$ be coordinates on $D$, and write $\hnabla$ for a flat affine connection on $D$ such that $\hnabla dy_{i} = 0$. Then the pullback of $\nabla$ via the Legendre transform of $F$ is the connection dual to $\nabla$ with respect to the Hessian metric $\nabla dF$.
\end{lemma}
\begin{proof}
This straightforward, though notationally awkward, computation, is left to the reader.
\end{proof}

Since $\bar{\ak}_{ij} = -\bnabla_{i}\bar{H}_{j} = \bnabla_{i}H_{j} = -\ak_{ij} - F_{ijp}H^{p}$, the K\"ahler affine Ricci and scalar curvatures of $(\nabla, g)$ and $(\bnabla, g)$ are related by
\begin{align}\label{akplusakbar}
&\ak_{ij} + \bar{\ak}_{ij} = - F_{ijp}H^{p},& &\aks + \bar{\aks} = -|H|^{2}_{g}.
\end{align}
On the other hand, writing $H^{\sharp}$ for the vector field $H^{i}$ and combining $(\lie_{H^{\sharp}}g)_{ij} = 2D_{(i}H_{j)} = - 2\ak_{ij} - F_{ijp}H^{p}$ with \eqref{akplusakbar} yields
\begin{align}\label{akakbardiff}
\bar{\ak}_{ij} - \ak_{ij} = (\lie_{H^{\sharp}}g)_{ij} = 2D_{i}H_{j}.
\end{align}
In particular, the vector field $H^{i}$ is $g$-Killing if and only if the K\"ahler affine Ricci tensor of $(\nabla, g)$ agrees with that of the dual K\"ahler affine structure.

\subsection{}
A flat affine manifold is \textbf{convex} if it is a quotient of a convex domain in flat affine space by a free and properly discontinuous action of a discrete group of affine transformations. It is moreover \textbf{properly convex} if the convex domain is proper. Theorem $2.1$ of H. Shima and K. Yagi's \cite{Shima-Yagi} shows that the affine developing map of a complete K\"ahler affine manifold is an affine diffeomorphism onto a convex domain in flat affine space. That is, the affine structure underlying a complete K\"ahler affine structure is convex.

An interesting question is what conditions characterize those complete K\"ahler affine manifolds for which the underlying affine structure is properly convex. In \cite{Koszul-deformations}, J.~L. Koszul showed that a compact flat affine manifold $(M, \nabla)$ is properly convex if and only if $M$ admits a closed one-form $\al$ such that $\nabla_{i}\al_{j} > 0$. A proper convex domain is said to be \textbf{divisible} if it admits a compact quotient by a discrete group of affine transformations. By Theorem $4$ of J. Vey's \cite{Vey}, a divisible proper convex domain is a cone. In particular the theorems of Koszul and Vey imply that the universal cover of a compact K\"ahler affine manifold with negative K\"ahler affine Ricci curvature is a proper convex cone. 

The Euclidean metric is Hessian, so any Euclidean manifold gives an example of a (flat) K\"ahler affine manifold with universal cover equal to the full affine space. Corollary $2.3$ of \cite{Cheng-Yau-realmongeampere} shows the related result that if a compact K\"ahler affine manifold $(M, g, \nabla)$ admits a $\nabla$-parallel volume form then there is $F \in \cinf(M)$ such that $g_{ij} + \nabla_{i}dF_{j}$ is a (flat) Riemannian metric with Levi-Civita connection $\nabla$. In particular, since on a compact manifold a Levi-Civita connection is complete, this means $\nabla$ is complete, and so the universal cover of $M$ is the full flat affine space. Combining this theorem of Koszul and Vey with this result of Cheng and Yau yields:
\begin{theorem}
On a compact manifold a flat affine structure admitting a parallel volume form is not properly convex.
\end{theorem}
\begin{proof}
By the theorem of Koszul and Vey, a properly convex affine structure on a compact manifold $M$ has as its universal cover a proper convex cone, so is incomplete. The canonical potential of this cone is the potential of a negative scalar curvature K\"ahler affine Einstein metric which is invariant under the action of the holonomy representation of $M$, and so descends to give a K\"ahler affine Einstein metric on $M$. On the other hand, if the flat affine structure underlying a K\"ahler affine metric on a compact manifold admits a parallel volume form, then it is complete, a contradiction.
\end{proof}
It is an interesting question whether a K\"ahler affine structure for which the underlying affine connection is complete must admit a parallel volume form. This is a version of the Markus conjecture for K\"ahler affine manifolds.

\subsection{}
Let $(\nabla, g)$ be a K\"ahler affine metric and let $D$ be the Levi-Civita connection of $g$. Since $\nabla_{[i}\ak_{j]k} = -\nabla_{[i}\nabla_{j]}H_{k} = 0$, $\nabla_{i}\ak_{jk}$ is completely symmetric, and so $2D_{[i}\ak_{j]k} = -F_{k[i}\,^{p}\ak_{j]p}$. Tracing this in $jk$ yields
\begin{align}\label{dak}
d\aks_{i} = D_{p}\ak_{i}\,^{p} -  \tfrac{1}{2}F_{i}\,^{pq}\ak_{pq} + \tfrac{1}{2}H^{p}\ak_{ip} = g_{ia}g^{bc}\nabla_{b}\ak_{c}\,^{a}.
\end{align}
From \eqref{dak} there results $g^{pq}\nabla_{p}((n+1)\ak_{q}\,^{i} - \aks \delta_{q}\,^{i}) = ng^{ip}d\aks_{p}$, from which there follows:
\begin{lemma}\label{keakslemma}
The K\"ahler affine scalar curvature of a K\"ahler affine Einstein structure is constant. 
\end{lemma}

For $A \in \cinf(M)$, a straightforward calculation using the classical Bochner formula
\begin{align}\label{bochner}
\lap_{g}|dA|_{g}^{2} = 2|DdA|_{g}^{2} + 2A^{p}d_{p}\lap_{g}A + 2A^{i}A^{j}R_{ij},
\end{align}
shows that
\begin{align}\label{bbochner}
\blap_{g}|dA|_{g}^{2} = 2|DdA|_{g}^{2} + 2A^{p}d_{p}\blap_{g}A + 2A^{i}A^{j}\bric_{ij}.
\end{align}
Where $u > 0$ there holds $\blap_{g}u^{\la} = \la u^{\la-1}(\blap_{g}u + (\la - 1)u^{-1}|du|_{g}^{2})$. Taking $u = |dA|^{2}_{g}$ and $\la = 1/2$ and using \eqref{bbochner} yields that wherever $|dA|^{2}_{g}$ is not zero there holds
\begin{align}\label{bbochner2}
\begin{split}
|dA|_{g}\blap_{g}|dA|_{g} &= |DdA|_{g}^{2}  - |d|dA|_{g}|_{g}^{2} + A^{p}d_{p}\blap_{g}A + A^{i}A^{j}\bric_{ij} \\
&\geq  A^{p}d_{p}\blap_{g}A + \tfrac{1}{4(n+1)}(A^{p}H_{p})^{2} + \tfrac{1}{2}A^{i}A^{j}\ak_{ij},
\end{split}
\end{align}
the final inequality by the Kato inequality, \eqref{nricdefined}, and \eqref{nr1}. 

Let $\si_{ij} = D_{i}H_{j} - \tfrac{1}{n+1}D^{p}H_{p}g_{ij}$. By \eqref{aks},
\begin{align}\label{bb2}
|DH|^{2}_{g} = |\si|^{2}_{g} + \tfrac{1}{n+1}(\aks + \tfrac{1}{2}|H|^{2}_{g})^{2}.
\end{align}
Taking $A = \log\H(F)$ in \eqref{bbochner} and using \eqref{aks}, \eqref{bb2}, \eqref{nricdefined}, and \eqref{nr1} yields
\begin{align}\label{sz}
\begin{split}
\blap_{g}|H|_{g}^{2} &= 2|\si|^{2}_{g} - 2H^{p}d_{p}\aks +  \tfrac{1}{2(n+1)}(2\aks + |H|^{2}_{g})^{2}  + 2H^{i}H^{j}\bric_{ij}\\
& \geq  2|\si|^{2}_{g} - 2H^{p}d_{p}\aks +  \tfrac{1}{n+1}(|H|^{4}_{g} + 2\aks |H|^{2}_{g} + 2\aks^{2}) + H^{i}H^{j}\ak_{ij}.
\end{split}
\end{align}
Taking $A = \log\H(F)$ in \eqref{bbochner2} and using \eqref{aks} yields that wherever $|H|_{g} \neq 0$ there holds
\begin{align}\label{sz0}
\begin{split}
|H|_{g}\blap_{g}|H|_{g} &\geq  -H^{p}d_{p}\aks + \tfrac{1}{4(n+1)}|H|^{4}_{g} + \tfrac{1}{2}H^{i}H^{j}\ak_{ij}.
\end{split}
\end{align}

\begin{theorem}\label{scalartheorem}
Let $\nabla$ be a flat affine connection on an $(n+1)$-dimensional manifold $M$ and let $g_{ij}$ be a complete Riemannian signature K\"ahler affine metric on $M$. If the K\"ahler affine Ricci curvature is bounded from below, $\ak_{ij} \geq -2cg_{ij}$, for some $c \geq 0$, and the K\"ahler affine scalar curvature $\aks$ is constant, then $|H|^{2}_{g} \leq 4(n+1)c$ and $-2(n+1)c \leq \aks \leq (\sqrt{2} - 1)(n+1)c/2$. In particular, the following are equivalent
\begin{enumerate}
\item\label{kas1} The K\"ahler affine Ricci curvature $\ak_{ij}$ is nonnegative and $\aks$ is constant.
\item\label{kas2}The K\"ahler affine Ricci curvature $\ak_{ij}$ vanishes.
\item\label{kas3} The metric $g$ is flat and $\nabla$ is its Levi-Civita connection.
\end{enumerate}
\end{theorem}
\begin{proof}
Suppose that $\ak_{ij} \geq -2cg_{ij}$ and $\aks$ is constant. Then \eqref{sz0} reduces to $\blap_{g}|H|_{g} \geq \tfrac{1}{4(n+1)}|H|^{3}_{g} - c|H|_{g}$ wherever $|H|_{g}$ is not zero, which, by Theorem \ref{cyestimatetheorem}, yields the bound $|H|^{2}_{g} \leq 4(n+1)c$. Hence, if $\ak_{ij} \geq 0$, then $H_{i} = 0$. By passing to the universal cover of $M$ it suffices to prove \eqref{kas1}$\implies$\eqref{kas3} in the case $M$ is simply connected. In this case, by the theorem of Shima-Yagi, $M$ can be identified with a convex domain in flat affine space and there exists a global potential $F \in \cinf(M)$ such that $g_{ij} = \nabla_{i}dF_{j}$ and $\H(F)$ is constant. By the corollary to the main theorem of \cite{Calabi-improper} in this case $F$ is a quadratic polynomial defined on all of affine space, so $F_{ijk}$ vanishes, and $\nabla$ is the Levi-Civita connection of $g$. This proves that \eqref{kas1} implies \eqref{kas3}. As obviously \eqref{kas3}$\implies$\eqref{kas2}$\implies$\eqref{kas1}, this proves the equivalence of \eqref{kas1}-\eqref{kas3}. There remains to prove the claimed bounds on $\aks$.

By \eqref{sz},
\begin{align}\label{sz2}
\blap_{g}|H|_{g}^{2} \geq 2|\si|^{2}_{g} + \tfrac{1}{n+1}\left(|H|^{4}_{g}  + 2(\aks  - (n+1)c)|H|^{2}_{g} + 2\aks^{2}\right).
\end{align}
Because $|H|^{2}_{g}$ is bounded, it follows from \eqref{xxric3} that the ordinary Ricci curvature of $g$ is bounded from below. Consequently the Omori-Yau maximum principle can be applied to \eqref{sz2}, and, writing $u = |H|^{2}_{g}$ and $u^{\ast} = \sup_{M}u$, it yields a sequence $\{x_{k}\} \in M$ such that $u(x_{k}) \geq u^{\ast} - k^{-1}$, $|du(x_{k})|_{g} < k^{-1}$, and $\lap_{g}u(x_{k}) < k^{-1}$. Substituting these relations into \eqref{sz2} shows that at $x_{k}$ there holds
\begin{align}
\tfrac{1}{k}( 1+ 2(n+1)c) \geq \tfrac{1}{k}(1 + \tfrac{1}{2}|H|_{g}) \geq \tfrac{1}{n+1}\left((u^{\ast} - \tfrac{1}{k})^{2}  + 2(\aks  - (n+1)c)(u^{\ast}- \tfrac{1}{k}) + 2\aks^{2}\right).
\end{align}
Letting $k \to \infty$ shows that the polynomial $p(x) = x^{2} + 2(\aks - (n+1)c)x + 2\aks^{2}$ has a nonnegative real root. This forces $(\sqrt{2} - 1)(n+1)c \geq \aks$. 
\end{proof}

Corollary \ref{sccorollary} is due to H. Shima in \cite{Shima-compacthessian}. It raises the question of whether the upper bound on $\aks$ in Theorem \ref{scalartheorem} can be improved. 
\begin{corollary}[\cite{Shima-compacthessian}]\label{sccorollary}
On a compact manifold a K\"ahler affine structure $(\nabla, g)$ has nonnegative K\"ahler affine scalar curvature if and only if $g$ is a flat metric with Levi-Civita connection $\nabla$.
\end{corollary}
\begin{proof}
Suppose $\aks \geq 0$. Integrating \eqref{aks} yields $2\int_{M}\aks\,d\vol_{g} = -\int_{M}|H|^{2}_{g}\,d\vol_{g}$, so that it must be that $\aks = 0$ and $H_{i} = 0$, and so also $\ak_{ij} = 0$. The conclusion follows from Theorem \ref{scalartheorem}.
\end{proof}

This subsection concludes with a digression to explain that the completeness assumption in Theorem \ref{scalartheorem} is essential. In particular, there exist K\"ahler affine Einstein metrics $(\nabla, g)$ having positive K\"ahler affine scalar curvature and for which $g$ extends to a complete metric on some larger, compact manifold. That $g$ extends to a complete metric on some larger manifold might appear to contradict Theorem \ref{scalartheorem}. However, while that $(\nabla, g)$ be complete as a K\"ahler affine metric on $M$ means that $g$ is a complete Riemannian metric, it also means that $\nabla$ is defined on all of $M$. This imposes on $M$ the a priori quite restrictive condition that it admit a flat affine structure. In the cases where $\nabla dG$ with $G$ as in Theorem \ref{wangzhutheorem} extends to a complete metric on some larger manifold, it is the case that this manifold does not admit a flat affine structure restricting to give $\nabla$. 

The most important special case of the following theorem is due to X-J. Wang and X. Zhu in \cite{Wang-Zhu}. The version stated here is essentially that of \cite{Berman-Berndtsson}; see also \cite{Donaldson-toric}. In the statement, a system of flat affine coordinates on $\rea^{n+1\,\ast}$ dual to the coordinates $x^{i}$ on $\rea^{n+1}$ is written $y_{i}$.
\begin{theorem}[\cite{Wang-Zhu}, \cite{Donaldson-toric}, \cite{Berman-Berndtsson}]\label{wangzhutheorem}
For a bounded convex open domain $\Om \subset \rea^{n+1}$ with barycenter $b \in \rea^{n+1}$ there exists a convex function $G \in \cinf(\rea^{n+1\,\ast})$ such that $\H(G)(y) = ce^{-G(y) + b^{p}y_{p}}$ for some nonzero constant $c$ depending only on $\Om$, and such that the map $y \to dG(y)$ is a diffeomorphism from $\rea^{n+1}$ onto $\Om$.
\end{theorem}
The function $v(y) = G(y) - b^{p}y_{p}$ solves $\H(v) = ce^{-v}$, and its differential maps $\rea^{n+1}$ onto the image of $\Om$ under the translation sending $b$ to the origin. The Hessian metric $(\nabla, g = \nabla dG)$ generated by $G$ on $\rea^{n+1\,\ast}$ is K\"ahler affine Einstein with positive K\"ahler affine scalar curvature. By Theorem \ref{scalartheorem} this means that $g$ is not complete. As will be explained briefly next, it can happen that the metric $\nabla dG$ extends to be complete on some manifold for which $\rea^{n+1\,\ast}$ is a coordinate chart. By \cite{Wang-Zhu}, this happens when $G$ arises as the restriction to an open orbit of the K\"ahler potential of the K\"ahler Einstein metric on a toric Fano manifold with vanishing Futaki invariant. 

Let $P$ be a Delzant polytope. This is an $(n+1)$-valent convex polytope satisfying certain integrality conditions (see \cite{Guillemin-toric} for the definition). As is explained in section $2$ of \cite{Donaldson-toric} from $P$ there can be constructed a $2(n+1)$-dimensional symplectic manifold $X$ equipped with an effective torus action having moment map $\mu$ with image $P$, and a Lagrangian submanifold $M \subset X$ such that the restriction to $M$ of $\mu$ is a $2^{n+1}$-fold covering over the interior $P^{\circ}$ of $P$. Let $G \in \cinf(\rea^{n+1\,\ast})$ be the function given by Theorem \ref{wangzhutheorem} such that $dG$ maps $\rea^{n+1\,\ast}$ diffeomorphically onto $P^{\circ}$ and let $g = \nabla dG$ where $\nabla$ is the flat affine structure on $\rea^{n+1\,\ast}$. Then $(\nabla, g)$ is K\"ahler affine Einstein with K\"ahler affine scalar curvature $1$. By Lemma \ref{duallegendrelemma} the pullback via $dG$ of the standard flat affine connection on $P$ forms with $g$ the dual K\"ahler affine structure. Pushing $(\nabla, g)$ forward to $P^{\circ}$ via $dG$, and pulling the result back to $M$ via $\mu$ yields a positive K\"ahler affine Einstein structure on the dense open subset $M^{\circ} = \mu^{-1}(P^{\circ})$ of $M$. In the case $P$ is the moment polytope of a toric Fano manifold with vanishing Futaki invariant it follows from \cite{Wang-Zhu} that $g$ extends to all of $M$.

An explicit illustrative example in which this occurs comes from the realization of complex projective space as a toric variety. In this case the relevant potential is $G(y) = (n+2)\log(1 + \sum_{i = 0}^{n}e^{y_{i}/(n+2)})$. It solves $\H(G)  = (n+2)^{-(n+1)}e^{-G(y) + \sum_{i = 0}^{n}y_{i}/(n+2)}$, and its differential maps $\rea^{n+1}$ onto the open standard simplex $\simplex = \{s_{i} \in \rea^{n+1}: s_{i} \geq 0, \sum_{i}s_{i} \leq 1\}$ with vertices at the origin and the standard coordinate vectors $e_{0}, \dots, e_{n}$; the barycenter is $b = (e_{0} + \dots +e_{n})/(n+2)$ (see, e.g., section $4$ of \cite{Klartag} for the relevant computations). If the standard simplex is regarded as the Delzant polytope corresponding to the standard torus action on complex projective space, then the Legendre transform $G^{\ast}(s) = (n+2)\left(\sum_{i}s_{i}\log s_{i} + (1 - \sum_{i}s_{i})\log(1 - \sum_{i}s_{i}) \right)$ of $G$ is a constant multiple of the canonical symplectic potential of the Fubini-Study symplectic structure.

By the projective $(n+1)$-sphere is meant the space of rays in a $(n+2)$-dimensional vector space. A standard affine coordinate chart is given by setting one of the homogeneous coordinates equal to $1$. On such a chart, which it is convenient to regard as $\rea^{n+1\,\ast}$, there are (inhomogeneous) coordinates $z_{i}$ such that the one-forms $dz_{i}$ are parallel with respect to the standard flat affine connection $\hnabla$. Let $u(z) = (1 + |z|^{2})^{1/2}$, where the norm is the standard Euclidean norm. The Levi-Civita connection $D$ of the metric $h = u^{-1}u_{ij} = u^{-1}\hnabla du$ is $D = \hnabla - 2u^{-1}u_{(i}\delta_{j)}\,^{k}$, which is evidently projectively flat, and it follows that $h$ is a metric of constant positive sectional curvature defined on all of $\rea^{n+1\,\ast}$. The $(n+1)$-dimensional projective sphere can be covered by standard affine coordinate charts in which the standard Fubini-Study metric has this form. The metric $u_{ij}$ is by construction Hessian with respect to $\hnabla$. It is straightforward to check that $\hat{H}(u) = u^{-n-3}$, where $\hat{\H}(u)$ is defined with respect to the $\hnabla$-parallel volume $\hat{\Psi} = dz_{0}\wedge \dots \wedge dz_{n}$ and $\hnabla$. Define a diffeomorphism from the orthant $\orthant = \{z_{i} > 0\}$ in $\rea^{n+1, \ast}$ to $\rea^{n+1\,\ast}$ by $y_{i} = 2(n+2)\log z_{i}$. Then $u(z) = e^{G(y)/2(n+2)}$, and a straightforward computation shows that the pullback of the metric $g$ equals $4(n+2)h$, showing directly $g$ has constant positive sectional curvature. Note, however, that the pullback of the flat affine connection $\nabla$ for which $dy_{i}$ is a parallel coframe is neither equal to $\hnabla$ nor projectively equivalent to $\hnabla$. Geometrically the description of the projective sphere in terms of the potential $G$ is obtained by regarding it as the standard cross polytope the faces of which are regular simplices.

The preceeding discussion suggests that the notion of K\"ahler affine metric is inadequate, or, rather, in some sense the specialization for affine structures of some structure defined for manifolds equipped with a projective structure. This sense is reinforced by the remark that there are no compact examples with positive Chern class in the sense of Cheng and Yau (\cite{Cheng-Yau-realmongeampere}), and suggests enlarging the notion of a K\"ahler affine manifold. Some much more general notions motivated in part by such considerations were proposed in \cite{Fox-ahs} and the relevant special case is described now. Say that a pair $(\en, [g])$, on an $(n+1)$-dimensional manifold $M$, comprising a projective structure $\en$ and a Riemannian conformal structure $[g]$ is \textbf{locally K\"ahler affine} if for every $p \in M$ there are an open neighborhood $U$ of $p$, a representative $\nabla$ of the restriction to $U$ of $\en$, and a representative $g$ of the restriction to $U$ of $[g]$, such that on $U$, $\nabla$ is flat as an affine connection and $\nabla_{[i}g_{j]k} = 0$. In particular this necessitates that $\en$ be projectively flat. For any pair $(\en, [g])$ there is a unique representative $\nabla \in \en$, said to be \textbf{aligned}, such that $(n+1)g^{jk}\nabla_{j}g_{ki} = g^{jk}\nabla_{i}g_{jk}$ for any $g \in [g]$. It can be proved that, if $n+1 \geq 3$, a pair $(\en, [g])$ with $\en$ projectively flat is locally K\"ahler affine if and only if for any $g \in [g]$ the aligned representative $\nabla \in \en$ satisfies $\nabla_{[i}g_{j]k} = 2\tau_{[i}g_{j]k}$ for a closed one-form $\tau_{i}$ (necessarily equal to $(1/2(n+1))g^{pq}\nabla_{i}g_{pq}$). A pair $(\en, [g])$ satisfying this latter condition is what is called a projectively flat AH structure in \cite{Fox-ahs}.

\subsection{}
Suppose $g_{ij} = \nabla_{i}dF_{j}$ is a Hessian metric with global potential $F$. Using \eqref{nricdefined} there results
\begin{align}
\label{normddf}
\begin{split}
|DdF|_{g}^{2} &= (n+1) - F^{p}H_{p} +  \tfrac{1}{4}F^{i}F^{j}F_{ip}\,^{q}F_{jq}\,^{p}= (n+1) -  F^{p}H_{p} +  F^{i}F^{j}(\bric_{ij} - \tfrac{1}{2}\ak_{ij}).
\end{split}
\end{align}
Taking $A = F$ in \eqref{bbochner}, using $\blap_{g}F = n+1$, and substituting \eqref{normddf} and \eqref{nricdefined} yields
\begin{align}\label{blapdf}
\begin{split}
\blap_{g}|dF|_{g}^{2} &= 2|DdF|_{g}^{2} + 2F^{i}F^{j}\bric_{ij}.
\end{split}
\end{align}
Substituting \eqref{nricdefined} and \eqref{normddf} into \eqref{blapdf}, and using \eqref{nr1} yields
\begin{align}\label{blapdf2}
\begin{split}
\blap_{g}|dF|_{g}^{2}& = 2(n+1) - 2F^{p}H_{p} + F^{i}F^{j}F_{ip}\,^{q}F_{jq}\,^{p}  + F^{i}F^{j}\ak_{ij}\\
& = (n+1) + \left(\tfrac{F^{p}H_{p}}{\sqrt{n+1}} - \sqrt{n+1}\right)^{2} + \left(F^{i}F^{j}F_{ip}\,^{q}F_{jq}\,^{p} - \tfrac{1}{n+1}(H^{p}F_{p})^{2}\right) + F^{i}F^{j}\ak_{ij} \\
& \geq (n+1) + F^{i}F^{j}\ak_{ij}.
\end{split}
\end{align}
Using \eqref{blapdf} yields that wherever $|dF|^{2}_{g}$ is not zero there holds
\begin{align}\label{blapdf3}
\begin{split}
|dF|_{g}\blap_{g}|dF|_{g} & = \tfrac{1}{2}\left(\blap_{g}|dF|^{2}_{g} - \tfrac{1}{2}|dF|_{g}^{-2}|d|dF|^{2}_{g}|^{2}_{g}\right) = |Ddf|_{g}^{2}  - |d|dF|_{g}|_{g}^{2} + F^{i}F^{j}\bric_{ij} \\
&\geq F^{i}F^{j}\bric_{ij}\geq \tfrac{1}{4(n+1)}(F^{p}H_{p})^{2} + \tfrac{1}{2}F^{i}F^{j}\ak_{ij}.
\end{split}
\end{align}
the penultimate inequality by the Kato inequality, and the last inequality by \eqref{nricdefined} and \eqref{nr1}. 

Since both $H^{p}F_{p}$ and $|dF|^{2}_{g}$ are unchanged if $F$ is replaced by $e^{r}F$, an inequality of the form $H^{p}F_{p} \geq b |dF|^{2}_{g}$ makes sense.

\begin{lemma}\label{dfboundedlemma}
Let $\nabla$ be a flat affine connection on the $(n+1)$-dimensional manifold $M$, and suppose $F \in \cinf(M)$ is such that $g_{ij} = \nabla_{i}dF_{j}$ is a complete Riemannian metric on $M$. If there are constants $b > 0$ and $c > 0$ such that $H^{p}F_{p} \geq b|dF|^{2}_{g}$ and $\ak_{ij} \geq -2cg_{ij}$ then $\sup_{M}|dF|^{2}_{g} \leq 4(n+1)cb^{-2}$. If, moreover, $|H|^{2}_{g}$ is bounded from above, then $\sup_{M}|dF|^{2}_{g} \geq (n+1)/(2c)$.
\end{lemma}
\begin{proof}
The inequality \eqref{blapdf3} simplifies to $\blap_{g}|dF|_{g} \geq \tfrac{b^{2}}{4(n+1)}|dF|_{g}^{3} - c|dF|_{g}$. By Theorem \ref{cyestimatetheorem}, $\sup_{M}|dF|^{2}_{g} \leq 4(n+1)cb^{-2}$. Let $u = |dF|^{2}_{g}$ and $u^{\ast} = \sup_{M}|dF|^{2}_{g}$.  If $|H|^{2}_{g}$ is bounded from above by $Q^{2} > 0$ then by \eqref{xxric3} the ordinary Ricci curvature of $g$ is bounded from below. Since the Ricci curvature is bounded from below and $g$ is complete, by the Omori-Yau maximum principle there is a sequence of points $\{x_{k}\} \subset \Omega$ such that $u(x_{k}) \geq u^{\ast} - k^{-1}$, $|du(x_{k})|_{g} < k^{-1}$, and $\lap_{g}u(x_{k}) < k^{-1}$. Substituting this into \eqref{blapdf2} yields that at $x_{k}$ there holds
\begin{align}
\begin{split}
k^{-1}(1 + Q/2)& \geq k^{-1} + \tfrac{1}{2}|H|_{g}|du|_{g} \geq \blap_{g}u \geq n+1 -2cu \geq n+1 - 2cu^{\ast}.
\end{split}
\end{align}
Letting $k \to \infty$ yields $u^{\ast} \geq (n+1)/(2c)$.
\end{proof}

\begin{corollary}
The canonical potential $F$ of a proper convex domain $\Om \subset \rea^{n+1}$ satisfies $(n+1)/2 \leq \sup_{\Om}|dF|_{g}^{2} \leq n+1$.
\end{corollary}
\begin{proof}
If $\H(F) = e^{\al F}$ for some $\al \neq 0$ and $g_{ij} = \nabla_{i}dF_{j}$ is complete, then the hypotheses of Lemma \ref{dfboundedlemma} are satisfied with $b = \al$ and $c = \al/2$, and Lemma \ref{dfboundedlemma} yields that $(n+1)\al^{-1} \leq \sup_{M}|dF|^{2}_{g} \leq 2(n+1)\al^{-1}$. 
\end{proof}
In fact, by Theorem \ref{affinespheretheorem}, when $M$ is a proper convex cone $F$ must be logarithmically homogeneous and $|dF|^{2}_{g} = 2(n+1)\al^{-1}$, but this is a nontrivial consequence of Corollary \ref{gfcorollary}. On the other hand, for a bounded convex domain the canonical potential, being a strictly convex function tending to $+\infty$ on the boundary of the domain, has a unique minimum in the domain, so in this case $|dF|_{g}^{2}$ is certainly not constant.

\subsection{}
Skewing
\begin{align}\label{df}
D_{i}F_{jkl} = F_{ijkl} - \tfrac{3}{2}F_{i(j}\,^{p}F_{kl)p},
\end{align}
in $ij$ and tracing it in $il$ yield 
\begin{align}\label{df2}
&D_{[i}F_{j]kl} = 0,&& D^{p}F_{ijp} = D_{i}F_{jp}\,^{p} = D_{i}H_{j}. 
\end{align}
The completely symmetric trace-free tensor $A_{ijk}$ defined by
\begin{align}\label{aijkdefined}
A_{ijk} = F_{ijk} - \tfrac{3}{n+1}H_{(i}g_{jk)} + \tfrac{2}{n+1}|H|_{g}^{-2}H_{i}H_{j}H_{k},
\end{align}
satisfies
\begin{align}
\begin{split}
H^{p}A_{ijp} &= H^{p}F_{ijp} - \tfrac{1}{n+1}|H|^{2}_{g}g_{ij} = -2\left(\ak_{ij} - \tfrac{1}{n+1}\aks g_{ij} + D_{i}H_{j} - \tfrac{1}{n+1}D^{p}H_{p}g_{ij} \right),\\
\end{split}
\end{align}
so that $H^{p}A_{ijp} = 0$ if $H^{i}$ is Killing and $(\nabla, g)$ is K\"ahler affine Einstein.  This observation and the identity \eqref{df2} suggest that interesting conditions on a K\"ahler affine metric are that $H^{i}$ be Killing or conformal Killing for $g$. In this regard, observe:
\begin{lemma}\label{killinglemma}
For a K\"ahler affine Einstein metric $(\nabla, g)$ on a compact manifold $M$ there hold
\begin{enumerate}
\item\label{kl1} The vector field $H^{i}$ is parallel with respect to the Levi-Civita connection $D$ of $g$, and if $g$ is not flat then $H^{i}$ is nowhere vanishing and satisfies $H^{i}H^{j}R_{ij} = 0$.
\item\label{kl2} The dual K\"ahler affine structure $(\bnabla, g)$ is K\"ahler affine Einstein, and $\bar{\aks} = \aks = -\tfrac{1}{2}|H|_{g}^{2}$.
\item\label{kl3} The affine structures determined by $\nabla$ and $\bnabla$ are radiant with radiant vector field $\eul^{i} = (n+1)\aks^{-1}H^{i}$.
\end{enumerate}
\end{lemma}
\begin{proof}
By Theorem \ref{scalartheorem}, a complete K\"ahler affine Einstein metric is either flat or has negative K\"ahler affine scalar curvature. In the flat case $H^{i}$ is $\nabla$-parallel and $\nabla$ is the Levi-Civita connection of $g$, so $H^{i}$ is $D$-parallel. In the case $\aks < 0$, since $\aks$ is constant there holds $-\nabla_{i}('aks^{-1}(n+1)H_{j}) = g_{ij} >0$. If $M$ is compact, then by the theorem of Koszul and Vey the universal cover of $M$ is a proper convex cone. Since the universal cover is simply connected, there is a globally defined primitive $F$ of $-(\tfrac{n+1}{\aks})H_{j}$, and so $F$ is also a global potential of $g$. Since the lifted metric is complete, by Corollary \ref{gfcorollary}, there is a constant $c$ such that $\tfrac{-2\aks}{n+1}F + c$ is the canonical potential. This implies $D_{i}F_{j} = 0$, so also $D_{i}H_{j} = 0$, which shows that the vector field on the universal cover dual to the Koszul form is Killing, and so the same is true on the original manifold. That $(\bnabla, g)$ is K\"ahler affine Einstein, and $\bar{\aks} = \aks = -\tfrac{1}{2}|H|_{g}^{2}$ follow from \eqref{akplusakbar} and \eqref{akakbardiff}. Since $H_{i}$ is parallel, it is either identically zero or nowhere-vanishing, and if $\aks < 0$, it must be that $H_{i}$ is nowhere vanishing. It then follows from \eqref{sz} that $H^{i}H^{j}R_{ij} = 0$. Since $D_{i}H_{j} = 0$, by \eqref{akplusakbar} there holds $(n+1)H^{p}F_{ijp} = -2\aks g_{ij}$, and so 
\begin{align}
\nabla_{i}H^{j} = - \ak_{i}\,^{j} - F_{i}\,^{jp}H_{p} = \tfrac{\aks}{n+1}\delta_{i}\,^{j}.
\end{align}
This shows that $\eul^{i} = (n+1)\aks^{-1}H^{i}$ satisfies $\nabla_{i}\eul^{j} = \delta_{i}\,^{j}$, so $\nabla$ is radiant. By \eqref{kl2}, $\bnabla$ is radiant with the same radiant vector field.
\end{proof}

It is unclear whether any kind of converse to Lemma \ref{killinglemma} is true, that is whether a complete K\"ahler affine metric for which $H^{i}$ is $g$-Killing must be K\"ahler affine Einstein. Suppose $D_{i}H_{j} = 0$. Then $D^{p}H_{p} = 0$ and $|H|^{2}_{g}$ is constant, so $\aks$ is constant as well. Substituting these observations in \eqref{sz} and simplifying the result using \eqref{nricdefined} yields $H^{i}H^{j}R_{ij} = 0$, or, what is the same, $H^{i}H^{j}F_{ip}\,^{q}F_{jq}\,^{p} = H^{i}H^{j}H^{k}F_{ijk}$. However, this does not seem to be enough to conclude that the given K\"ahler affine metric is Einstein.

\subsection{}\label{scalarcurvaturesection}
This section gives the proofs of Theorems \ref{sctheorem}-\ref{scalarcurvature2theorem}. For the remainder of this section, let $F$ be the canonical potential of the proper open convex cone $\Om \subset \rea^{n+1}$. There will be used repeatedly and without further comment the identities $\eul^{i} = -F^{i}$, $H_{i} = F_{ip}\,^{p} = 2F_{i}$, $D_{i}F_{j} = 0$, and those obtained from them by differentiation. The tensor $A_{ijk}$ defined in \eqref{aijkdefined} becomes 
\begin{align}\label{aijkdefined2}
A_{ijk} = F_{ijk} + \tfrac{6}{n+1}\eul_{(i}g_{jk)} - \tfrac{4}{(n+1)^{2}}\eul_{i}\eul_{j}\eul_{k}.
\end{align}
It is determined by the requirements that it be trace-free, that $\eul^{k}A_{ijk} = 0$, and that in the directions tangential to a level set of $F$ it agree with $F_{ijk}$. Since by Theorem \ref{affinespheretheorem} a level set of $F$ is a totally geodesic submanifold of $\Om$, using \eqref{hglobal2} it can be shown that the restriction of $A_{ijk}$ to a level set of $F$ is the Pick form of the level set. Since $A_{ijk}$ differs from $F_{ijk}$ by a parallel tensor, it follows from \eqref{df2} that $D_{[i}A_{j]kl} = 0$ and $D^{p}A_{ijp} = D^{p}F_{ijp} = D_{i}H_{j} = 2D_{i}F_{j} = 0$. As a result of these identities, wherever $A$ is not zero there holds the refined Kato inequality
\begin{align}\label{kato}
|DA|^{2}_{g} \geq \tfrac{n+5}{n+2}|d|A|_{g}|_{g}^{2}.
\end{align}
The inequality \eqref{kato} is a special case of Lemma $6.4$ of \cite{Fox-ahs}, or could be deduced from the results in section $6$ of \cite{Calderbank-Gauduchon-Herzlich}. Using the Ricci identity and the aforementioned symmetries of $D_{i}A_{jkl}$ yields
\begin{align}\label{df4}
\begin{split}
\lap_{g}A_{ijk} &= D^{p}D_{p}A_{ijk} = D_{p}D_{i}A_{jk}\,^{p} = -A_{k}\,^{pq}R_{pijq} - A_{j}\,^{pq}R_{pikq} + A_{jk}\,^{p}R_{ip}.
\end{split}
\end{align}
It is convenient to introduce $A_{ijkl} = 2A_{k[i}\,^{p}A_{j]lp}$, which has the symmetries of a Riemannian curvature tensor. Note that $A_{pij}\,^{p} = A_{ip}\,^{q}A_{jq}\,^{p}$. Using \eqref{hessiancurvature} there result 
\begin{align}\label{df5}
\begin{split}
F_{ij}\,^{p}F_{klp} & = A_{ij}\,^{p}A_{klp} + \tfrac{4}{n+1}g_{ij}g_{kl} - \tfrac{8}{n+1}\eul_{(i}A_{jkl)} + \tfrac{16}{(n+1)^{2}}\eul_{(i}g_{j)(k}\eul_{l)}  - \tfrac{16}{(n+1)^{3}}\eul_{i}\eul_{j}\eul_{k}\eul_{l},
\end{split}
\end{align}
\begin{align}\label{df6}
\begin{split}
R_{ijkl} & = \tfrac{1}{4}A_{ijkl} + \tfrac{2}{n+1}g_{k[i}g_{j]l} + \tfrac{2}{(n+1)^{2}}\left(\eul_{l}\eul_{[i}g_{j]k} - \eul_{k}\eul_{[i}g_{j]l}\right),\\
R_{ij} & = \tfrac{1}{4}A_{ip}\,^{q}A_{jq}\,^{p} - \tfrac{n-1}{n+1}\left(g_{ij} - \tfrac{1}{n+1}\eul_{i}\eul_{j}\right), \qquad R_{g}  = \tfrac{1}{4}|A|_{g}^{2} - \tfrac{n(n-1)}{n+1}.
\end{split}
\end{align}
The lower bound in \eqref{riccibound} of Theorem \ref{riccicurvaturetheorem} follows from \eqref{df6}. Substituting \eqref{df6} in \eqref{df4} yields
\begin{align}\label{df7b}
\begin{split}
\lap_{g}A_{ijk} & = \tfrac{3}{4}A_{(ij}\,^{a}A_{k)q}\,^{p}A_{ap}\,^{q} - \tfrac{1}{2}A_{ia}\,^{b}A_{jb}\,^{c}A_{kc}\,^{b} - A_{ijk}.
\end{split}\\
\label{df7}
\begin{split}
A^{ijk}\lap_{g}A_{ijk} & = \tfrac{3}{4}A_{ip}\,^{q}A_{jq}\,^{p}A^{i}\,_{a}\,^{b}A^{j}\,_{b}\,^{a} - \tfrac{1}{2}A^{ijk}A_{ia}\,^{b}A_{jb}\,^{c}A_{kc}\,^{b} - |A|^{2}_{g}.
\end{split}
\end{align}
The next part of the argument is an adaptation of section $5$ of Calabi's \cite{Calabi-completeaffine}. The details of some tensorial computations omitted in \cite{Calabi-completeaffine} are included because although elementary, they are not apparent at a glance. Define $\phi(x) = \sup_{u \in T_{x}\Om}\{(u^{i}u^{j}A_{ip}\,^{q}A_{jq}\,^{p})^{1/2}: |u|^{2}_{g} = 1\}$. Fix $p \in \Om$ and let $v$ be a $g$-unit vector in $T_{p}\Om$ for which $v^{i}v^{j}A(p)_{ip}\,^{q}A(p)_{jq}\,^{p}$ takes its maximum value. Since $e^{t}v$ is a $g$-unit vector in $T_{e^{t}p}\Om$, it follows from the logarithmic homogeneity of $F$ that $\phi(tp) \geq e^{2t}v^{i}v^{j}A(e^{t}p)_{iq}\,^{p}A(e^{t}p)_{jp}\,^{q} = \phi(p)$. Since the same argument shows $\phi(p) \geq \phi(e^{t}p)$, it follows that $\phi$ has homogeneity $0$. In a $g$-geodesically convex neighborhood of $p$ extend $v$ to a vector field by $D$-parallel transport along $D$-geodesics emanating from $p$ and consider $\bphi(x) = (v(x)^{i}v(x)^{j}A(x)_{ip}\,^{q}A(x)_{jq}\,^{p})^{1/2}$. The vector $v$ itself can be viewed as a $\nabla$-parallel homogeneity $0$ vector field $V$ on $\Om$. Since for any homogeneity $0$ vector field $A$ there holds $\eul^{i}D_{i}A^{j} = 0$, it follows that the vector field $e^{-F/(n+1)}V$ is $D$-parallel along the integral curves of $\eul$. Hence $v(e^{t}p)$ must equal $e^{-F(e^{t}p))/(n+1)}v(p) = e^{t}v(p)$. This shows that $d\bphi(\eul)$ vanishes at $p$. By definition $\bphi(x) \leq \phi(x)$, with equality when $x$ is a multiple of $p$, in particular when $x = p$.

To show a differential inequality for $\blap_{g}\phi$ in the \textit{barrier sense} (what Calabi calls the \textit{weak sense} in \cite{Calabi-hopfmaximum}), it suffices to show the differential inequality for $\blap_{g} \bphi$. By construction $D_{i}v^{j}$ and $\lap_{g}v^{j}$ vanish at $p$. Because $d\bphi(\eul)$ vanishes at $p$, at $p$ there holds $\blap_{g}\bphi = \lap_{g}\bphi$. If $\bphi(p) = \phi(p) \neq 0$, then at $p$ there holds
\begin{align}\label{dfc1}
\begin{split}
\bphi \blap_{g}\bphi & = \bphi \lap_{g}\bphi = v^{i}v^{j}A_{i}\,^{pq}\lap_{g}A_{jpq} \\
&+ \bphi^{-2}v^{i}v^{j}v^{k}v^{l}\left(A_{ku}\,^{v}A_{lv}\,^{u}D_{a}A_{ip}\,^{q}D^{a}A_{jq}\,^{p} - A_{ip}\,^{q}A_{ku}\,^{v}D_{a}A_{jq}\,^{p}D^{a}A_{lv}\,^{u} \right) \\
&\geq  v^{i}v^{j}A_{i}\,^{pq}\lap_{g}A_{jpq}.
\end{split}
\end{align}
Here has been used
\begin{align}\label{dfc1b}
\begin{split}
2v^{i}&v^{j}v^{k}v^{l}\left(A_{ku}\,^{v}A_{lv}\,^{u}D_{a}A_{ip}\,^{q}D^{a}A_{jq}\,^{p} - A_{ip}\,^{q}A_{ku}\,^{v}D_{a}A_{jq}\,^{p}D^{a}A_{lv}\,^{u} \right) \\
&=v^{i}v^{j}v^{k}v^{l}\left(A_{k}\,^{uv}D^{a}A_{i}\,^{pq} - A_{k}\,^{pq}D^{a}A_{i}\,^{uv}\right)\left(A_{luv}D_{a}A_{jpq} - A_{lpq}D_{a}A_{puv}\right)\\
&=\left(v^{i}v^{k}\left(A_{k}\,^{uv}D^{a}A_{i}\,^{pq} - A_{k}\,^{pq}D^{a}A_{i}\,^{uv}\right)\right)\left(v^{j}v^{l}\left(A_{luv}D_{a}A_{jpq} - A_{lpq}D_{a}A_{puv}\right)\right) \geq 0,
\end{split}
\end{align}
the final inequality because the penultimate expression is simply the squared norm of the $5$ index tensor $v^{j}v^{l}\left(A_{luv}D_{a}A_{jpq} - A_{lpq}D_{a}A_{puv}\right)$.
Substituting \eqref{df7b} in \eqref{dfc1} yields
\begin{align}\label{dfc2}
\begin{split}
\bphi \blap_{g}\bphi & \geq \tfrac{1}{4}v^{i}v^{j}A_{ip}\,^{q}A_{aq}\,^{p}A_{ju}\,^{v}A^{a}\,_{v}\,^{u} + \tfrac{1}{4}v^{i}v^{j}A_{iabc}A_{j}\,^{abc} -\bphi^{2}.
\end{split}
\end{align}
For any symmetric two tensor $B_{ij}$, the nonnegativity of $v^{i}v^{j}B_{ia}B_{jb}$ implies that its trace satisfies $v^{i}v^{j}B_{ip}B_{j}\,^{p} \geq v^{i}v^{j}v^{k}v^{l}B_{ij}B_{kl}$. Applied with $B_{ij} = A_{ip}\,^{q}A_{jq}\,^{p}$, this yields 
\begin{align}\label{dfc4}
v^{i}v^{j}A_{ip}\,^{q}A_{aq}\,^{p}A_{ju}\,^{v}A^{a}\,_{v}\,^{u} \geq \bphi^{4}. 
\end{align}
Similarly,
\begin{align}\label{dfc5}
\begin{split}
v^{i}v^{j}A_{iabc}A_{j}\,^{abc} &\geq v^{i}v^{j}v^{k}v^{l}(A_{iabk}A_{j}\,^{ab}\,_{l} + A_{iakc}A_{j}\,^{a}\,_{l}\,^{c}) \\
&= 2 v^{i}v^{j}v^{k}v^{l}A_{iabk}A_{j}\,^{ab}\,_{l}= 2 v^{i}v^{j}v^{k}v^{l}A_{i(ab)k}A_{j}\,^{(ab)}\,_{l}.
\end{split}
\end{align}
In $n$ dimensions any symmetric two tensor $C_{ij}$ satisfies $nC^{pq}C_{pq} \geq (C_{p}\,^{p})^{2}$. Since $C_{ab} = v^{i}v^{k}A_{i(ab)k}$ vanishes when contracted with $\eul^{a}$ or with $v^{a}$, it can be viewed as a tensor on the $(n-1)$-dimensional orthogonal complement of the span of $\eul$ and $v$. Since $C_{p}\,^{p} = v^{i}v^{j}A_{ip}\,^{q}A_{jq}\,^{p} = \bphi^{2}$, there results $(n-1) v^{i}v^{j}v^{k}v^{l}A_{i(ab)k}A_{j}\,^{(ab)}\,_{l} \geq \bphi^{4}$. In \eqref{dfc5} this yields
\begin{align}\label{dfc6}
v^{i}v^{j}A_{iabc}A_{j}\,^{abc} \geq \tfrac{2}{n-1}\bphi^{4}. 
\end{align}
Substituting \eqref{dfc4} and \eqref{dfc6} into \eqref{dfc2} yields that at $p$ there holds
\begin{align}\label{dfc3}
\begin{split}
\bphi \blap_{g}\bphi & \geq \tfrac{n+1}{4(n-1)}\bphi^{4}  -\bphi^{2}.
\end{split}
\end{align}
In the case that $\bphi(p) = \phi(p) = 0$ the inequality \eqref{dfc3} is trivially true. It follows that
\begin{align}\label{dfc8}
\blap_{g}\phi \geq \tfrac{n+1}{4(n-1)}\phi^{3} - \phi,
\end{align}
holds in the barrier sense. Were Theorem \ref{cyestimatetheorem} known to hold also for a differential inequality valid in the barrier sense, applying it to \eqref{dfc8} would yield $\phi^{2} \leq \tfrac{4(n-1)}{n+1}$, and so $\tfrac{1}{4}A_{ip}\,^{q}A_{jq}\,^{p} \leq \tfrac{n-1}{n+1}$, which in \eqref{df6} implies $R_{ij} \leq 0$. Since it has not been shown that Theorem \ref{cyestimatetheorem} is valid in this generality, there is outlined now a direct proof. This follows closely the proofs of Lemma $5.3$ and Theorem $5.4$ of \cite{Calabi-completeaffine} which show the nonpositivity of the Ricci curvature of a complete hyperbolic affine sphere. The only substantive modifications of Calabi's arguments are that the affine mean curvature, dimension, and distance comparison theorem have to be replaced by the appropriate corresponding objects in the metric measure sense. It is convenient to define $\ka = (2n+1)^{-1/2}$ and 
\begin{align}
\tphi = \sqrt{\frac{n+1}{4(n-1)(2n+1)}}\phi,
\end{align}
so that, by \eqref{dfc8}, there holds
\begin{align}\label{dfc9}
\blap_{g}\tphi \geq (2n+1)\tphi^{3} - \tphi.
\end{align}
That the Ricci curvature of $g$ be positive at $p$ is equivalent to the requirement that there be a positive constant $a$ such that $\tphi^{2}(p) > a^{2} > \ka^{2}$. Let $f(t)$ solve the differential equation
\begin{align}
\ddot{f}(t) + (2n+1)\ka\coth(\ka t)\dot{f}(t) = (2n+1)(f^{3}(t) - \ka^{2}f(t)),
\end{align}
with the initial conditions $f(0) = a$ and $\dot{f}(0) = 0$. Since 
\begin{align}
\tfrac{d}{dt}\left(\sinh^{2n+1}(\ka t)\dot{f}\right) = (2n+1)\sinh^{2n+1}(\ka t)(f^{3} - \ka^{2}f)
\end{align}
is positive when $f(t) \geq a$, $\dot{f}$ is strictly positive for $t > 0$. For the positive constant
\begin{align}
b = \tfrac{\max\{4, 2(n+1)\}\ka^{2}}{(2n+3)(a^{2} - \ka^{2})},
\end{align}
the function $v(t) = ab(b+1 - \cosh(\ka t))^{-1}$ satisfies the differential inequality
\begin{align}
\begin{split}
\ddot{v} &+ (2n+1)\ka\coth(\ka t)\dot{v} - (2n+1)(v^{3} - \ka^{2} v)\\
&= \tfrac{v^{3}\ka^{2}}{a^{2}b^{2}}\left( \cosh^{2}(\ka t) - 2n(b+1)\cosh(\ka t) + (2n+1)(b+1)^{2} - (2n+1)a^{2}b^{2}\right) \leq 0,
\end{split}
\end{align}
as well as the initial conditions $v(0) = a$ and $\dot{v}(0) = 0$. As in the proof of Lemma $5.3$ of \cite{Calabi-completeaffine}, it follows that $f(t) - v(t)$ does not achieve a maximum in the interior of the common domain of regularity of $f$ and $v$, and that in consequence $f$ blows up as $t$ approaches the boundary of the interval $(-\delta, \delta)$, where 
\begin{align}
\delta = \ka^{-1}\operatorname{arccosh}(1 + b) \leq \ka^{-1}\operatorname{arccosh}\left(1 +  \tfrac{\max\{4, 2(n+1)\}\ka^{2}}{(2n+3)(\tphi^{2}(p) - \ka^{2})}\right).
\end{align}
Let $r$ be the $g$-distance from $p$ and let $u = f(r)$. By the preceeding, $\tphi - u$ tends to $-\infty$ uniformly on the boundary of a compact domain on which it therefore attains a maximum at some point $q$. At $q$ its value is necessarily positive, since its value at $p$ is positive. Arguing as in the proof of the corollary on page $53$ of \cite{Calabi-hopfmaximum}, and using Theorem \ref{bqhessiantheorem} with $A = 1$, yields that the function $u$ satisfies, in the barrier sense, the differential inequality
\begin{align}\label{dfc9b}
\blap_{g}u \leq \ddot{f}(r) + \ka^{-1}\coth(\ka r)\dot{f}(r) = (2n+1)(u^{3} - \ka^{2}u) = (2n+1)u^{3} - u.
\end{align}
Together \eqref{dfc9} and \eqref{dfc9b} show that there holds $\blap_{g}(\tphi - u) > 0$ at the point $q$. This contradicts the generalized maximum principle of \cite{Calabi-hopfmaximum}. It follows that the Ricci curvature of $g$ is nonpositive.

If $Q_{ijk}$ is a completely symmetric trace-free tensor, then the nonnegativity of the completely trace-free part of $Q_{k[i}\,^{p}Q_{j]lp}$ implies the inequality
\begin{align}\label{df8b}
\tfrac{3}{4}Q_{ip}\,^{q}Q_{jq}\,^{p}Q^{i}\,_{a}\,^{b}Q^{j}\,_{b}\,^{a} - \tfrac{1}{2}Q^{ijk}Q_{ia}\,^{b}Q_{jb}\,^{c}Q_{kc}\,^{b}  \geq \tfrac{n+2}{4n(n+1)}|Q|_{g}^{4}.
\end{align}
This is proved in section $6$ of \cite{Fox-ahs} and is implicit in section $2$ of \cite{Calabi-improper}. However, since $A_{ijk}\eul^{k} = 0$, $A_{ijk}$ can be treated as a tensor on a space of one-dimension less, and so \eqref{df8b} holds for $A_{ijk}$, but with $n+1$ replaced by $n$. That is,
\begin{align}\label{df8}
 \tfrac{3}{4}A_{ip}\,^{q}A_{jq}\,^{p}A^{i}\,_{a}\,^{b}A^{j}\,_{b}\,^{a} - \tfrac{1}{2}A^{ijk}A_{ia}\,^{b}A_{jb}\,^{c}A_{kc}\,^{b}\geq \tfrac{n+1}{4(n-1)n}|A|_{g}^{4}.
\end{align}
Combining \eqref{kato}, \eqref{df4}, and \eqref{df8} yields
\begin{align}\label{df9}
\begin{split}
\tfrac{1}{2}\lap_{g}|A|^{2}_{g} &= |DA|^{2}_{g} + A^{ijk}\lap_{g}A_{ijk} \geq \tfrac{n+5}{n+2}|d|A|_{g}|^{2}_{g} +  \tfrac{n+1}{4(n-1)n}|A|^{4}_{g} - |A|^{2}_{g}.
\end{split}
\end{align}
Setting $u = |A|_{g}^{(n-1)/(n+2)}$, it follows that where $|A|_{g} \neq 0$ there holds
\begin{align}\label{df10}
\begin{split}
\lap_{g}u \geq \tfrac{n-1}{n+2}\left( \tfrac{n+1}{4(n-1)n}u^{1 + 2(n+2)/(n-1)} - u\right).
\end{split}
\end{align}
Applying Theorem \ref{cyestimatetheorem} to \eqref{df10} yields the bound
\begin{align}\label{df11}
0 \leq |A|_{g}^{2} \leq \tfrac{4(n-1)n}{n+1},
\end{align}
which in conjunction with the expression for $R_{g}$ in \eqref{df6} shows that $R_{g} \leq 0$. Of course this follows from $R_{ij} \leq 0$. However the cases of equality in \eqref{df11} are of interest in their own right. If $R_{g}$ is constant and equal to either $0$ or $-n(n-1)/(n+1)$ then one of the equalities holds in \eqref{df11}. In \eqref{df9} this forces $DA = 0$, and so also $D_{i}F_{jkl} = 0$. Hence $D_{i}R_{jklp} = 0$, so $(\Om, g)$ is a locally Riemannian symmetric space. Since $g$ is complete and $\Om$ is simply connected, $(\Om, g)$ is in fact a globally Riemannian symmetric space, and so homogeneous. This completes the proof of Theorem \ref{riccicurvaturetheorem}.

\begin{proof}[Proof of Theorem \ref{sctheorem}]
For $u, v \in \rea^{n+1}$ the ordinary Schwarz inequality applied to the pairing of the symmetric tensors $v^{p}F_{pij}$ and $u_{i}u_{j}$ yields $(v^{i}u^{j}u^{k}F_{ijk}(x))^{2} \leq v^{i}v^{j}F_{ip}\,^{q}F_{jq}\,^{p}|u|_{g}^{4}$. By \eqref{hessiancurvature}, the nonpositivity of $R_{ij}$ implies $ v^{i}v^{j}F_{ip}\,^{q}F_{jq}\,^{p} \leq 4|v|^{2}_{g}$. Taking $u = v$ there results $(v^{i}v^{j}v^{k}F_{ijk}(x))^{2} \leq 4|v|^{6}_{g}$, showing that $F$ is $1$-self-concordant. The $-(n+1)$-logarithmic homogeneity of $F$ shows $(v^{i}F_{i}(x))^{2} \leq (n+1)|v|_{g}^{2}$, and so $F$ is an $(n+1)$-normal barrier for $\Om$.
\end{proof}

\begin{corollary}[Corollary of Theorem \ref{riccicurvaturetheorem}]\label{riccicurvaturecorollary}
On a compact $(n+1)$-dimensional manifold $M$ the ordinary Ricci curvature $R_{ij}$ of a K\"ahler affine Einstein metric $(\nabla, g)$ with negative K\"ahler affine scalar curvature satisfies $0 \geq R_{ij} \geq \tfrac{1-n}{n+1}g_{ij}$ and is degenerate in the direction of $H^{i}$. 
\end{corollary}
\begin{proof}
By the main theorem of \cite{Koszul-deformations}, since $\nabla_{i}H_{j} > 0$, the universal cover of $M$ is a proper convex domain in flat affine space. By Theorem $4$ of \cite{Vey} a divisible proper convex domain is a cone, and so the universal cover of $M$ is a proper open convex cone and the claimed bounds on $R_{ij}$ follow by applying Theorem \ref{riccicurvaturetheorem} to the lift of the given K\"ahler affine structure to the universal cover. The degeneracy of $R_{ij}$ in the direction $H^{i}$ was proved in Lemma \ref{killinglemma}
\end{proof}

\subsection{}\label{mongeamperesection}
Recall from section \ref{mamsection} that a Monge-Amp\`ere metric means a K\"ahler affine metric for which $\ak_{ij} = 0$. Let $F \in \cinf(\Omega)$ be the canonical potential of the nonempty proper open convex cone $\Omega \subset \rea^{n+1}$. In the following there are sought functions $\psi$ so that $\psi(F)_{ij}$ is a Monge-Amp\`ere metric on some subset of $\Omega$. This idea for finding such metrics goes back to \cite{Calabi-nonhomogeneouseinstein}, where it was used to construct a Monge-Amp\`ere metric with a potential radial on $\rea^{n+1} \setminus \{0\}$ (see the penultimate paragraph beginning on page $18$ of \cite{Calabi-nonhomogeneouseinstein}). The same idea was applied  in \cite{Loftin-Yau-Zaslow} to construct more examples; see remarks below.

Let $B \in \reat$. By \eqref{hpsif}, $\H(\psi(F)) = B$ if $\psi$ solves the differential equation $\dot{\psi}(t)^{n}((n+1)\ddot{\psi}(t) + \dot{\psi}(t))e^{2t} = B$, or, what is equivalent,
\begin{align}\label{hpde}
\tfrac{d}{dt}\left(e^{t}\dot{\psi}(t)^{n+1} + Be^{-t}\right) = 0. 
\end{align} 
There must be a constant $C \in \rea$ such that $\dot{\psi}(t)^{n+1} = e^{-t}(C - Be^{-t})$. Since $\psi(F)_{ij} = \dot{\psi}(F)F_{ij} + \ddot{\psi}(F)F_{i}F_{j}$, it is evident that for $\psi(F)_{ij}$ to be a metric it must be that $\dot{\psi}(F) \neq 0$. In this case 
\begin{align}\begin{split}
(n+1)\ddot{\psi} =\dot{\psi}\left( Be^{-2t}\dot{\psi}(t)^{-n-1} - 1\right) = \dot{\psi}\left(\tfrac{2Be^{-t} - C}{C - Be^{-t}}\right), 
\end{split}
\end{align}
and so
\begin{align}
\psi(F)_{ij} = \dot{\psi}(F)\left(g_{ij} + \tfrac{1}{n+1}\left(\tfrac{2Be^{-F} - C}{C - Be^{-F}}\right)F_{i}F_{j} \right).
\end{align}
It follows that on the tangent space to a level set of $F$, $\psi(F)_{ij}$ is positive or negative definite according to whether $\dot{\psi}(F)$ is positive or negative. Hence where nondegenerate $\psi(F)_{ij}$ is either definite or has signature $(n,1)$ or $(1, n)$. Since replacing $\psi(t)$ by $-\psi(t)$ flips the signature, attention will be restricted to the case in which $\dot{\psi}(F) > 0$. Since $\psi(F)_{ij}\eul^{i}\eul^{j} = (n+1)\dot{\psi}(F)B(Ce^{F} - B)^{-1}$, the resulting metric will be Riemannian or Lorentzian according to whether $B(Ce^{F} - B)^{-1}$ is positive or negative. If $B$ and $C$ have opposite signs or $C = 0$ then $B(Ce^{F} - B)^{-1}< 0$, so any resulting metric will be Lorentzian; however it can be defined on all of $\Omega$. If $B$ and $C$ have the same sign (in particular, $C \neq 0$) then $B(Ce^{t} - B)^{-1} > 0$ if and only if $t > \log(B/C)$, so any resulting metric will be positive definite on the region $\{x \in \Omega: F(x) > \log(B/C)\}$; however such a metric will not extend to all of $\Omega$.

The three cases $C = 0$, $C > 0 > B$, and $B, C > 0$ (yielding $\dot{\psi}(t) > 0$ for at least some $t$) will be considered separately. In the third (Riemannian) case, $B, C > 0$ and 
\begin{align}\label{mametric}
\begin{split}
\psi(t) &= \int_{\log(B/C)}^{t}\left(Ce^{-s} - Be^{-2s}\right)^{1/(n+1)}\,ds \\
&= \tfrac{B^{1/(n+1)}}{n+1}\int_{e^{-t/(n+1)}}^{(C/B)^{1/(n+1)}} \left( \tfrac{C}{B} - r^{n+1}\right)^{1/(n+1)}\, dr,
\end{split}
\end{align}
is well defined for $t \geq \log(B/C)$, and for $t> \log(B/C)$ satisfies $\dot{\psi}(t) > 0$ and $Be^{-t}(C - Be^{-t})^{-1} > 0$, so the resulting metric $\psi(F)_{ij}$ is a Riemannian signature Monge-Amp\`ere metric on $\{x \in \Omega: F(x) > \log(B/C)\}$. The second equality in \eqref{mametric} results from the change of variables $r = e^{-s/(n+1)}$, and corresponds to working with the homogeneity $1$ function $e^{-F/(n+1)}$ in place of $F$. The resulting metric is determined up to positive homothety by the parameter $B/C$. Taking $B = C = (n+1)^{n+1}$ the resulting metric \eqref{mametric} is just as in Proposition $1$ of \cite{Loftin-Yau-Zaslow-erratum}.

Since $F(x)$ tends to $\infty$ as $x$ tends to the boundary of $\Omega$ and to $-\infty$ as $x$ runs out to spatial infinity along a ray contained in $\Omega$, the region $\{x \in \Omega: F(x) > \log(B/C)\}$ is the open subset of $\Omega$ bounded by the boundary of $\Omega$ and the affine sphere $\lc_{\log(B/C)}(F, \Omega)$, or, what is the same, the union of the open line segments contained in $\Omega$ and running from the origin to some point of $\lc_{\log(B/C)}(F, \Omega)$. This proves Theorem \ref{mongeamperetheorem}.

For $a \in \reat$ the function $\Psi(t) = ae^{-2t/(n+1)}$ solves \eqref{hpde} with $B = (-1)^{n}(2a/(n+1))^{n+1}$ and $C = 0$. In order that $\dot{\psi}(t) > 0$ suppose $a < 0$. Since, by its logarithmic homogeneity, $F$ maps $\Omega$ onto $\rea$, the resulting Lorentzian metric $\psi(F)_{ij}$ is defined on all of $\Omega$. If $C > 0 > B$ then 
\begin{align}
\psi(t) = \int_{0}^{t}\left(Ce^{-s} - Be^{-2s}\right)^{1/(n+1)}\,ds  = (-B)^{1/(n+1)}\int_{0}^{t}e^{-s/(n+1)}\left(e^{-s} - C/B\right)^{1/(n+1)}\,ds
\end{align}
solves \eqref{hpde} and has $\dot{\psi}(t) > 0$ for all $t \in \rea$. Since $B(Ce^{t} - B)^{-1} < 0$, the resulting metric $\psi(F)_{ij}$ is Lorentzian and defined on all of $\Omega$. Evidently this metric is determined up to positive homothety by the ratio $C/B$. If $B$ is normalized to be $-1$ (in the $C = 0$ case, this amounts to taking $a = -(n+1)/2$) then
\begin{align}\label{lma}
\psi(F)_{ij} = e^{-F/(n+1)}\left(e^{-F} + C \right)^{1/(n+1)}\left(F_{ij} - \tfrac{1}{n+1}\left(\tfrac{2e^{-F}+ C}{e^{-F} + C}  \right)F_{i}F_{j}\right).
\end{align}

\begin{proof}[Proof of Theorem \ref{lmatheorem}]
The level set $\lc_{r}(F, \Omega)$ is a hyperbolic affine sphere complete in the equiaffine metric. From \eqref{sffdefined} and \eqref{lma} with $C = 0$ it follows that the equiaffine metric is a positive constant multiple of the restriction to $\lc_{r}(F, \Omega)$ of $k$, so that $\lc_{r}(F, \Omega)$ is a complete spacelike hypersurface. Let $v = \sqrt{n+1}e^{-F/(n+1)}$, so that $u = -v^{2}/2$. The equiaffine metric on $\lc_{r}(F, \Omega)$ is the restriction of the tensor $h_{ij}$ defined in \eqref{hglobal2}, and from the identity
\begin{align}\label{mongesplit}
k_{ij} = v^{2/(n+2)}h_{ij} - v_{i}v_{j}, 
\end{align}
it is apparent that $(\Omega, k_{ij})$ is isometric to $\reap \times \lc_{0}(F, \Omega)$ equipped with the metric on the righthand side of \eqref{mongesplit}. The global hyperbolicity of $k_{ij}$ follows.
\end{proof}

\subsection{}
For a proper open convex cone $\Om$ with canonical potential $F$ the associated functions $v = -(n+1)e^{-F/(n+1)}$ and $u = -v^{2}/2$ have the following properties:
\begin{enumerate}
\item\label{ma1} The function $v$ is negative, has positive homogeneity $1$, is convex, vanishes on the boundary of $\Om$, has Hessian of rank $n$, and solves $\H(v) = 0$.
\item\label{ma2} The function $u$ is negative, has positive homogeneity $2$, is convex, vanishes on the boundary of $\Om$, has Hessian with signature $(n, 1)$, and solves $\H(u) = -1$.
\end{enumerate}
The foliation determined by the distribution spanning the kernel of the rank $n + 1 -k$ Hessian of a convex function on $\rea^{n+1}$ is called in \cite{Foote} the \textbf{Monge-Amp\`ere foliation} of the function (see also \cite{Hartman-Nirenberg}). The Monge-Amp\`ere foliation of the function $v$ of \eqref{ma1} is the foliation of $\Om$ by rays through the origin.  
It would be interesting to prove the uniqueness, on a proper convex cone $\Om$, of either a negative convex solution of $\H(v) = 0$ having Hessian of rank $n$ and vanishing on $\pr \Om$ or a negative solution of $\H(u) = -1$ having Hessian of signature $(n, 1)$ and vanishing on $\pr \Om$. 

\bibliographystyle{amsplain}
\def\cprime{$'$} \def\cprime{$'$} \def\cprime{$'$}
  \def\polhk#1{\setbox0=\hbox{#1}{\ooalign{\hidewidth
  \lower1.5ex\hbox{`}\hidewidth\crcr\unhbox0}}} \def\cprime{$'$}
  \def\Dbar{\leavevmode\lower.6ex\hbox to 0pt{\hskip-.23ex \accent"16\hss}D}
  \def\cprime{$'$} \def\cprime{$'$} \def\cprime{$'$} \def\cprime{$'$}
  \def\cprime{$'$} \def\cprime{$'$} \def\cprime{$'$} \def\cprime{$'$}
  \def\cprime{$'$} \def\cprime{$'$} \def\cprime{$'$} \def\cprime{$'$}
  \def\dbar{\leavevmode\hbox to 0pt{\hskip.2ex \accent"16\hss}d}
  \def\cprime{$'$} \def\cprime{$'$} \def\cprime{$'$} \def\cprime{$'$}
  \def\cprime{$'$} \def\cprime{$'$} \def\cprime{$'$} \def\cprime{$'$}
  \def\cprime{$'$} \def\cprime{$'$} \def\cprime{$'$} \def\cprime{$'$}
  \def\cprime{$'$} \def\cprime{$'$} \def\cprime{$'$} \def\cprime{$'$}
  \def\cprime{$'$} \def\cprime{$'$} \def\cprime{$'$} \def\cprime{$'$}
  \def\cprime{$'$} \def\cprime{$'$} \def\cprime{$'$} \def\cprime{$'$}
  \def\cprime{$'$} \def\cprime{$'$} \def\cprime{$'$} \def\cprime{$'$}
  \def\cprime{$'$} \def\cprime{$'$} \def\cprime{$'$} \def\cprime{$'$}
  \def\cprime{$'$} \def\cprime{$'$} \def\cprime{$'$} \def\cprime{$'$}
\providecommand{\bysame}{\leavevmode\hbox to3em{\hrulefill}\thinspace}
\providecommand{\MR}{\relax\ifhmode\unskip\space\fi MR }
\providecommand{\MRhref}[2]{%
  \href{http://www.ams.org/mathscinet-getitem?mr=#1}{#2}
}
\providecommand{\href}[2]{#2}

\end{document}